\documentclass[english]{paper}

\usepackage{mama}
\usepackage{quiver}

\title{Actegories for the working amthematician}
\author{Matteo Capucci and Bruno Gavranovi\'c}
\address{Mathematically Structured Programming Group,\\
  Department of Computer and Information Sciences, \\
  University of Strathclyde,\\
Glasgow, Scotland\\[1ex]}
\eaddress{matteo.capucci@strath.ac.uk \CR bruno@brunogavranovic.com}


\copyrightyear{2022}

\usepackage{todonotes}
\usepackage{fonttable}
\usepackage{stmaryrd}
\usepackage{wasysym}
\usepackage{tabularx}
\usepackage{expl3}
\usepackage{xparse}

\graphicspath{{img/}}

\DeclareFontFamily{U}{FdSymbolA}{}
\DeclareFontShape{U}{FdSymbolA}{m}{n}{
    <-> s * [1] FdSymbolA-Book
}{}
\DeclareFontShape{U}{FdSymbolA}{m}{b}{
    <-> s * [1] FdSymbolA-Medium
}{}
\DeclareSymbolFont{fdsymbols}{U}{FdSymbolA}{m}{n}
\SetSymbolFont{fdsymbols}{bold}{U}{FdSymbolA}{m}{b}

\DeclareMathSymbol{\smallblackdiamond}{\mathbin}{fdsymbols}{130}
\DeclareMathSymbol{\whitestar}{\mathbin}{fdsymbols}{146}

\DeclareMathOperator{\upperhalfcirc}{\raisebox{1pt}{\scalebox{0.6}{\LEFTcircle}}}

\usepackage{tikzit}

\tikzstyle{node}=[fill=black, draw=black, shape=circle, scale=0.5]
\tikzstyle{medium_box}=[fill=white, draw=black, shape=rectangle, minimum height=0.8cm, minimum width=0.5cm]

\tikzstyle{pointy}=[->]
\tikzstyle{bluearrow}=[->, fill=none, draw={rgb,255: red,29; green,206; blue,255}, thick]
\tikzstyle{lightnone}=[-, draw={rgb,255: red,191; green,191; blue,191}]

\tikzset{
	Rightarrow/.style={double equal sign distance,>={Implies},->},
	triple/.style={-,preaction={draw,Rightarrow}},
	quadruple/.style={preaction={draw,Rightarrow,shorten >=0pt},shorten >=1pt,-,double,double distance=0.2pt}
}

\newcommand{\eval}{\mathrm{eval}}
\newcommand{\curr}{\mathrm{curr}}
\newcommand{\repl}{\mathrm{repl}}

\newcommand{\id}{\mathrm{id}}

\newcommand{\cat}[1]{\mathcal{#1}}
\newcommand{\ncat}[1]{\mathbf{#1}}
\newcommand{\twocat}[1]{\mathbb{#1}}

\newcommand{\Cat}{\twocat{C}\ncat{at}}

\newcommand{\Set}{\ncat{Set}}

\renewcommand{\Vec}{\ncat{Vec}}
\newcommand{\Met}{\ncat{Met}}
\newcommand{\Meas}{\ncat{Meas}}
\newcommand{\Msbl}{\ncat{Msbl}}
\newcommand{\Prob}{\ncat{Prob}}

\newcommand{\Top}{\ncat{Top}}

\newcommand{\op}{\mathsf{op}}

\newcommand{\coop}{\mathsf{coop}}

\newcommand{\iso}[1][]{\overset{#1}{\cong}}
\newcommand{\equi}{\simeq}

\newcommand{\adj}{\dashv}

\newcommand{\Para}{\ncat{Para}}
\newcommand{\Copara}{\ncat{Copara}}
\newcommand{\Optic}{\ncat{Optic}}

\DeclareMathOperator{\extch}{\&}
\newcommand{\tensor}[1][\cat M]{\otimes_{#1}}
\newcommand{\mtimes}{\otimes}
\newcommand{\ntimes}{\oslash}
\newcommand{\ctimes}{\boxtimes}

\newcommand{\cplus}{\boxplus}
\newcommand{\copow}{\otimes}
\newcommand{\pow}{\pitchfork}

\newcommand{\action}{\bullet}
\newcommand{\actionn}{\circ}
\newcommand{\ostar}{\circledast}
\newcommand{\laction}{\smallblackdiamond}
\newcommand{\raction}{\diamond}
\newcommand{\waff}{\circledcirc}
\newcommand{\bothaction}{{\upperhalfcirc}}
\newcommand{\prodaction}{\bothaction^\times}
\newcommand{\coprodaction}{\bothaction^+}
\newcommand{\extchaction}{\bothaction^{\extch}}

\newcommand{\comp}{\fatsemi}

\newcommand{\PsdMon}{\ncat{PsdMon}}
\newcommand{\MonCat}{{\twocat{M}\ncat{onCat}}}
\newcommand{\SymMonCat}{{\twocat{S}\ncat{ymMonCat}}}
\newcommand{\BrMonCat}{{\twocat{B}\ncat{rMonCat}}}

\newcommand{\deloop}{\twocat{B}}

\newcommand{\Alg}[1]{{#1}\dash\twocat{A}\ncat{lg}}

\newcommand{\lax}{\mathrm{lx}}
\newcommand{\oplax}{\mathrm{ox}}
\newcommand{\pseudo}{\mathrm{ps}}
\newcommand{\strict}{\mathrm{s}}
\newcommand{\braided}{\mathrm{br}}
\newcommand{\cart}{\mathrm{cart}}

\newcommand{\Actt}{\twocat{A}\ncat{ct}}
\newcommand{\Act}[1]{{#1}\dash\Actt}

\newcommand{\Biact}[2][\cat M]{{#1}\dash\Actt\dash{#2}}

\newcommand{\rev}{\mathrm{rev}}
\newcommand{\swap}{\mathrm{swap}}

\newcommand{\Bal}{\mathrm{Bal}}

\newcommand{\equalto}{=\mathrel{\mkern-3mu}=}

\newcommand{\Day}{\mathrm{Day}}

\newcommand{\drinfeld}{\mathcal Z}

\newcommand{\Waff}{\cat W}

\makeatletter
\newcommand{\superimpose}[2]{%
  {\ooalign{$#1\@firstoftwo#2$\cr\hfil$#1\@secondoftwo#2$\hfil\cr}}}
\makeatother

\newcommand{\revrel}[1]{\mathrel{\overset{\rev}{#1}}}
\newcommand{\colim}{\operatorname{colim}}

\allowdisplaybreaks
\thirdleveltheorems

\begin{document}
	\maketitle

	\begin{abstract}
		Actions of monoidal categories on categories, also known as actegories, have been familiar to category theorists for a long time, and yet a comprehensive overview of this topic seems to be missing from the literature. Recently, actegories have been increasingly employed in applied category theory, thereby encouraging an effort to fill this gap according to the new needs of these applications. This work started as an investigation of the notion of monoidal actegory, a compatible pair of monoidal and actegorical structures, and ended up including a sizable reference on the elementary theory of actegories. We cover basic definitions and results on actegories and biactegories, spelling out explicitly many folkloric definitions, including their tensor product and their hom-tensor adjunction. We give new definitions of actegories with monoidal, braided monoidal and symmetric monoidal structure.
		In the last section, we provide three Cayley-like classification results for these structures.
	\end{abstract}

	\tableofcontents

	\section{Introduction}
\label{sec:introduction}
Monoid actions are ubiquitous structures in mathematics.
They provide a concrete implementation of the abstract compositional laws of a monoid as operations on some other object.
Each element $m$ of a monoid $(M, 1, \cdot)$ \emph{acts} on a given set (or manifold, space, etc.) $C$ by specifying an operation $m \action - : C \to C$, and the complex of these operations satisfy laws which are the `image' of those of a monoid:
\begin{equation*}
	n \action (m \action c) = (n\cdot m) \action c, \quad 1 \action c = c.
\end{equation*}
There are many examples of monoid actions (Figure~\ref{fig:actions}): scalar multiplication of vectors is an action of the field of scalars on vectors of a vector space, automata are essentially actions of free monoids on a space of `states', in dynamical systems monoids of time ($\N, \R^+$, etc.) act on spaces of states by tracing out trajectories.

\begin{figure}[H]
	\centering
	\includegraphics[width=\textwidth]{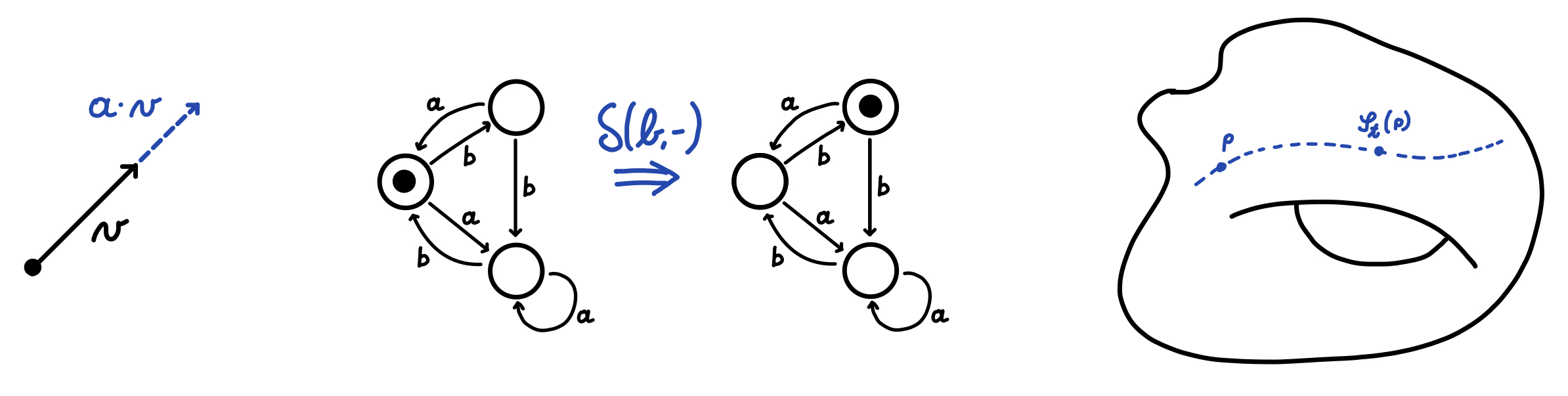}
	\caption{Examples of monoid actions from left to right: scalar multiplication of vectors, automata transitions, trajectories in spaces of configurations.}
	\label{fig:actions}
\end{figure}

Like monoidal categories promote the concept of monoid to a higher categorical dimension, the protagonists of this work, \emph{actegories}, promote the concept of a monoid action to its higher-dimensional equivalent (Figure \ref{diag:overview}).
In doing so, we lose the sharpness of equality we have in sets and we gain the expressivity of isomorphisms in categories.
Hence monoidal categories and actegories alike satisfies their laws (inherited from monoid and their actions, respectively) only up to isomorphism, which are themselves subjected to a new set of coherence equations.

\begin{figure}[ht]
	\centering
	\begin{equation*}
		\let\scriptstyle\textstyle
		\begin{tikzcd}[sep=3ex]
			{} & {\substack{\text{monoidal} \\ \text{category}}} && {\text{actegory}} \\
			&&&& {\phantom{}} \\
			{} & {\text{monoid}} && {\substack{\text{monoid} \\ \text{action}}} \\
			& {} && {}
			\arrow[hook, from=3-2, to=3-4]
			\arrow[hook, from=3-2, to=1-2]
			\arrow[hook, from=3-4, to=1-4]
			\arrow[hook, from=1-2, to=1-4]
			\arrow["{\text{action by another object}}"', dashed, from=4-2, to=4-4]
			\arrow["{\substack{\text{vertical} \\ \text{categorificaton}}}", dashed, from=3-1, to=1-1]
		\end{tikzcd}
	\end{equation*}
	\caption{An actegory as a combination of two different generalizations of a monoid. A monoid is to a monoid action what a monoidal category is to an actegory. Likewise, a monoid is to a monoidal category what a monoid action is to an actegory.}
	\label{diag:overview}
\end{figure}
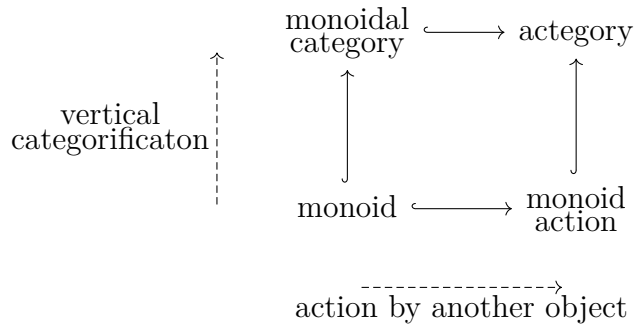

Actegories are not a new invention\footnote{or discovery!}, and not even a recent one. As we learned from \cite[Footnote 1]{cockett2009logic}, the concept was first entertained by B\'enabou in \cite{benabou1967introduction}, while the name itself (a portmanteu of `action' and `category') originated in the Australian Category Seminar at the end of the last century, and first appeared in print in \cite{mccrudden2000categories}.

Actegories have since appeared in categorical algebra \cite{mccrudden2000categories,skoda2009some,pareigis1977non} and in other various works spanning different subdisciplines in category theory \cite{garner2018embedding, street2020, szlachanyi2009fiber, janelidze2001note}; though the lack of consistent terminology makes it quite hard to get a grasp on the extent of the literature on actegories.

Nonetheless, in the past two decades there's been a surge of interest in actegories in a more `applied' context.
Indeed, the reason \emph{we} are interested in actegories in the first place is their effectiveness in capturing the idea of `agency' in categories of systems.
Part of these ideas have been published in \cite{capucci2021towards} which proposes a mathematical-conceptual framework for thinking about cybernetic systems; but our work doesn't stand alone.

The work in categorical cybernetics has been borrowing from the flurry of activity regarding mixed optics \cite{boisseau, riley2018categories, categoricalupdate}, of which actegories, and specifically Tambara modules\footnote{Tambara modules are lax equivariant profunctor between actegories on the same monoidal category. More details can be found in \cite[§5]{roman2020profunctor}.} \cite{pastrostreet, roman2020profunctor}, form the mathematical backbone.

The use of actegories in categorical cybernetics is also quite close to the work of Pastro and Cockett \cite{cockett2009logic, cockettACCATtalk, yeasin2012linear} on the categorical vertex of a Curry--Howard--Lambek correspondence for concurrency.
They resort to \emph{linear actegories} in order model message-passing concurrency, by having a `sequential world' of messages acting on a `concurrent world' of message-passing processes. The action of messages on processes reifies the first as values for the latter.

The work of Nishiwaki and Asai \cite{nishiwaki2020logic} fits in the same pattern.
They develop a generalization of Moggi's calculus of computational effects based on actegories, called \emph{semi-effects calculus}, whose categorical semantics is a linear functor between an actegory of `values' and one of `computations', with the intended meaning of representing a sort of `let' binding.

In a completely different setting, \cite{stefanou2018dynamics} shows how actegories and their morphisms can be fruitfully used to talk about the kind of gadgets topological data analysis produces, namely `persistency modules'. The way actegories are used in this work recalls the actions of monoids of time which give rise to dynamical systems, explaining his adoption of the terminology \emph{dynamical categories}.

Related to actegories, \emph{module categories} seemed to have enjoyed more attention.
These are actions of \emph{(multi-)fusion categories}, which are categories enriched in vector spaces with some extra properties\footnotemark. Their existence also motivates the choice of a distinct terminology for the `vanilla' notion.
Ostrik gives a systematic account of the theory of module categories in \cite{ostrik2003module}, where they indicate \cite{grayson1976higher} as the first appearance of said concepts. We also mention \cite{greenough2010monoidal, greenough2010bimodule}.
\footnotetext{Specifically, multifusion categories are `rigid semisimple $k$-linear tensor categories with finitely many simple objects and finite dimensional spaces of morphisms', where $k$ is an algebraically closed field \cite[p.~583]{etingof2005fusion}.}

Since, ultimately, multi-fusion categories are still monoidal categories, much of the theory of module categories can be reproduced in less structured contexts.
What really makes compendia such as Ostrik's less relevant for the uses of applied or even pure category theory is their focus, which is often strongly biased (and rightly so) on the interaction between the fusion stucture and the actegorical structure.
Moreover, much of the structure of multi-fusion categories is not available in contexts such as Kleisli categories of commutative monads on $\Set$ (i.e.~cats of effectful computations), categories of measurable spaces, or categories of manifolds.
Hence the need of compiling a more streamlined version of the theory behind module categories, closer to the sensibility of modern applied category theory.

In this work, and in all the works cited above, actions and actegories are treated in the same way, as dictated by common algebraic wisdom.
There is a different way to talk about actions though, which is connected with the ideas of cofunctors \cite{ahman2016directed, clarke2020internal} and differential forms (whose pullback is known in computer science as `reverse differentiation').
In this guise, monoidal actions have been studied by Alvarez-Picallo under the name of \textit{change actions} \cite{ChangeActionsThesis, ChangeActionModels}, in the context of incremental computation and differentiation.
To explore this connection more closely is left for future work, together with many other directions that would deserve to be explored.

\subsection{Overview of the paper}
This work has been born out of the practical need of identifying and understanding the structures involved in the proposed foundations of `categorical cybernetics', as outlined in \cite{capucci2021towards}.
In that work, we singled out the $\Para$ and $\Optic$ constructions as useful abstractions for modelling systems with agency, and moreover realized actegories provide an interesting conceptual view of the way agents and systems interact. This is remindful of the ideas put forward in \cite{cockettACCATtalk}.

The original motivating question for this work concerned the monoidal structures of $\Para$ and related constructions.
In \cite{capucci2021towards}, we singled out some structure morphisms necessary for promoting a (symmetric) monoidal product on an actegory $(\cat C, \action)$ to a (symmetric) monoidal product on $\Para_\action(\cat C)$.
It remained unclear how that structure arose and most importantly what were the constraints on that.
Here we clarify these questions by methodically analyzing the compatibility structures between actegorical and monoidal categories, thereby answering the previous question.
In doing so, we ended up compiling a much needed reference on actegories.

In Section~\ref{sec:monads}, we briefly go over some preliminary definitions concerning pseudomonads and their categories of algebras. This is done in order to ground the definitions of Section~\ref{sec:actegories}, which introduce actegories and their morphisms, all the way up to the detailed definition of the indexed 2-category of actegories. We also provide several examples and ancillary results. These two sections do not contain much new content, though the detailed exposition and meticulous definitions can be considered a new contribution.

Instead, Sections~\ref{sec:composition} and~\ref{sec:monoidal-actegories} contain the lion share of original content, and form the central part of the work.

The first is dedicated to the study of `composition of actegories', i.e.~ways in which new actegories can be systematically generated from old ones. Some of these ways are proper monoidal structures on categories of actegories (cartesian and cocartesian product, tensor product), but others are mathematically given by distributive laws, interpreted as decoration on the sequential composition of actegories seen as parametric morphisms.
The latter, Section~\ref{sec:monoidal-actegories}, is devoted to the study of compatibility structures between monoidal and actegorical structure. We explore various ways to combine them, introducing two new definitions: that of monoidal actegory and that of distributive algebroidal actegory. We show how the first is instrumental in making categories of optics monoidal, while the latter has connections to the theory of hybrid composition of optics. We conclude the section by analyzing the way braiding interacts with each of these new notions, and by proving three classification results for monoidal actegories, showing they are indeed equivalent to monoidal functors into suitably defined classifying objects.

Many definitions and proofs of coherence have been moved to Appendix~\ref{appendix:defs} and Appendix~\ref{appendix:proofs} in order to not obstruct the flow of the paper.

\subsection{Acknowledgements}
This paper wouldn't have been possible without the many hours of discussion with
collaborators at the MSP group.
We thank in particular Eigil F. Rischel who suggested many results to us (specifically, Lemma~\ref{lemma:strictification} and the classification theorems of Section~\ref{subsec:classifying}).
Their ranks are joined by Bojana Femi\'c, whose interest in the first author's talk at CT2021 \cite{paratalk} has been a source of inspiration, and who first pointed the first author to the vast literature on module categories; and Emily de Oliveira Santos, with whom the first author has been enthusiastically exchanging notes and opinions on categorified algebra.
We also thank the mantainers of \href{https://q.uiver.app}{quiver} for freely providing a tool which has proven invaluable in the preparation of this work.

\subsection{Notation and conventions}
Throughout the paper, we use $(\cat M, j, \mtimes)$ to refer to a monoidal
category~\cite[Definition 1.2.1]{johnson2021}, using $\lambda, \rho, \alpha$ to refer to left and right unitors and the associator respectively.
When braiding or symmetry are needed, we point it out explicitly.
We reserve the names $\cat M, \cat N$ for \emph{acting} monoidal categories, and
$\cat C, \cat D$ for categories that are being \emph{acted on}.
We make extensive use of diagrammatic notation for composition, using the symbol
$\comp$.
Actions are plentiful in this paper, and we use a variety of symbols: $\action$,
$\actionn$, $\ostar$, and so on, depending on the context.

Many of the structures we are going to contemplate are \emph{pseudo}-somethings (pseudoalgebras, pseudomonoids, etc.), which tends to make prose cumbersome and clumsy. Therefore we sometimes dropped the prefix with the understanding this does not hinder readability as it improves it.


	\newpage
	\section{Preliminaries}
\label{sec:monads}
As eloquently explained by \cite{blackwell1989two}, if universal algebra can be thought as the theory of 1-monads on $\Set$ (`0-categories'), universal 1-algebra\footnotemark\ can be studied as the theory of 2-monads on $\Cat$.
\footnotetext{The reference calls it 2-algebra, but we adopt the same perspective on dimension as HoTT, where 0-types are sets and (directed) 1-types are categories.
Therefore n-dimensional structure is structure on (possibly directed) n-types, and universal n-algebra is concerned with the study of algebraic structure on n-types.}

For the algebraic structures we are interested in, this is very much true.
Let's see what the algebraic story is for actions of monoids (`0-dimensional actegories') from this point of view.
Every monoid $(M, j', \mtimes) : \Set$ induces a monad $(M \times -, j, m) : \Set \to \Set$ where the unit and join are defined from the unit and product of $M$ itself.
We define them componentwise below, for a given $C : \Set$.

\vspace{1ex}
\begin{minipage}{\textwidth}
	\begin{minipage}{.4\textwidth}
		\begin{eqalign}
			j_C : C &\longto M \times C\\
				c &\longmapsto (j', c)
		\end{eqalign}
	\end{minipage}%
	\begin{minipage}{.6\textwidth}
		\begin{eqalign}
			m_C  : M \times M \times C &\longto 	  M \times C\\
					(m,n,c) 			&\longmapsto (m \mtimes n, c)
		\end{eqalign}
	\end{minipage}
\end{minipage}
\vspace{1ex}

An algebra $(C : \Set, \action : M \times C \to C)$ for this monad is precisely a monoid action, where algebra laws correspond to monoid action laws.
We state the laws below both as commutative diagrams, and in equational form.

\vspace{1ex}
\begin{minipage}{\textwidth}
	\begin{minipage}{.4\textwidth}
		\begin{diagram}
		\label{diag:monoid-action-unit}
			{M \times C} & C \\
			C
			\arrow["\action"', from=1-1, to=2-1]
			\arrow["{j_X}"', from=1-2, to=1-1]
			\arrow[Rightarrow, no head, from=2-1, to=1-2]
		\end{diagram}
		\vspace{1ex}
		\begin{center}
			$j \action c = c$, for all $c \in C$.
		\end{center}
	\end{minipage}%
	\begin{minipage}{.6\textwidth}
		\begin{diagram}
			\label{diag:monoid-action-comp}
			{M \times M \times C} & {M \times C} \\
			{M \times C} & C
			\arrow["{M \times \action}"', from=1-1, to=2-1]
			\arrow["{m_X}", from=1-1, to=1-2]
			\arrow["\action"', from=2-1, to=2-2]
			\arrow["\action", from=1-2, to=2-2]
		\end{diagram}
		\begin{center}
			$m \action (n \action c) = (m \mtimes n) \action c$, for all $m,n \in M$, $c \in C$.
		\end{center}
	\end{minipage}
\end{minipage}
\vspace{1ex}

Analogously, ${\cat M}$-actegories turn out to be \emph{pseudo}algebras for the \emph{pseudo}monad ${\cat M} \times -$.
A pseudomonad is the higher equivalent of a monad in the same way monoidal categories are the higher version of a monoid. Indeed, both (pseudomonads and monoidal categories) are instances of the categorified notion of `monoids in monoidal categories', which is `pseudomonoid in monoidal bicategories'.
A pseudomonoid is a monoid whose associativity and unitality only hold up to coherent isomorphism \cite[§2]{mccrudden2000balanced}, a structure which is very natural to find in 2-dimensional contexts, since structures there tend to be defined only up to higher equivalence.

\begin{center}
	\begin{tabular}{c|c}
		\textbf{n = 0} & \textbf{n = 1}\\
		\hline
		monoid $M$ & monoidal category $\cat M$\\[1.5ex]
		monoid homomorphism $f: M \to N$ & strong monoidal functor $F : \cat M \to \cat N$\\[1.5ex]
		free $M$-action monad $M \times -$ & free $\cat M$-actegory pseudomonad $\cat M \times -$\\[1.5ex]
		\begin{tabular}{@{}c@{}}action $M \times X \to X$\\(algebra of $M \times -$)\end{tabular}
		 &
		\begin{tabular}{@{}c@{}}pseudoaction $\cat M \times \cat C \to \cat C$\\(pseudoalgebra of $\cat M \times -$)\end{tabular}\\[2.5ex]
		\begin{tabular}{@{}c@{}}linear morphisms $g : (C, \action) \to (D, \actionn)$\\(algebra morphisms)\end{tabular}
		 &
		\begin{tabular}{@{}c@{}}linear functors $G : (\cat C, \action) \to (\cat D, \actionn)$\\(strong pseudoalgebra morphisms)\end{tabular}
	\end{tabular}
\end{center}
\vspace{2ex}

For example monoidal categories are pseudomonoids in $(\Cat, 1, \times)$, whereas pseudomonads on a 2-category $\twocat X$ are pseudomonoids in $([\twocat X, \twocat X], \id, \circ)$ \cite[§3]{marmolejo1999distributive}:

\subsubsection{Definition.}
\label{def:pseudomonad}
\begin{resdefinition}{defpseudomonad}
	Let $\twocat X$ be a 2-category.
	A \textbf{pseudomonad} on $\twocat X$ is a strict 2-functor $T: \twocat X \to \twocat X$, together with 2-morphisms $j : 1 \twoto T$ (\emph{unit}) and $m : TT \twoto T$ (\emph{multiplication}) and modifications (\emph{left unitor} $\lambda$, \emph{right unitor} $\rho$, \emph{associator} $\alpha$):
	\begin{equation}
		\begin{tikzcd}[sep=4ex,ampersand replacement=\&]
			T \&\& TT \&\& TT \&\& T \\
			\\
			\&\& T \&\& T
			\arrow["{T j}"', Rightarrow, from=1-7, to=1-5]
			\arrow["m"', Rightarrow, from=1-5, to=3-5]
			\arrow[""{name=0, anchor=center, inner sep=0}, Rightarrow, no head, from=3-5, to=1-7]
			\arrow["{j T}", Rightarrow, from=1-1, to=1-3]
			\arrow["m", Rightarrow, from=1-3, to=3-3]
			\arrow[""{name=1, anchor=center, inner sep=0}, Rightarrow, no head, from=1-1, to=3-3]
			\arrow["\rho", triple, shorten <=6pt, shorten >=6pt, to=1-5, from=0]
			\arrow["\lambda", triple, shorten <=6pt, shorten >=6pt, to=1-3, from=1]
		\end{tikzcd}
	\end{equation}

	\begin{equation}
		\begin{tikzcd}[sep=4ex,ampersand replacement=\&]
			TTT \&\& TT \\
			\\
			TT \&\& T
			\arrow["{m T}"', Rightarrow, from=1-1, to=3-1]
			\arrow["{T m}", Rightarrow, from=1-1, to=1-3]
			\arrow["m", Rightarrow, from=1-3, to=3-3]
			\arrow["m"', Rightarrow, from=3-1, to=3-3]
			\arrow["\alpha"{description}, triple, shorten <=11pt, shorten >=11pt, from=1-3, to=3-1]
		\end{tikzcd}
	\end{equation}
\end{resdefinition}
\noindent that satisfy the coherence laws reported in Appendix~\ref{appendix:defs}.

\begin{remark}
\label{rmk:pseud-vs-2}
	This notion of `higher dimensional monad' is intermediate between that of a `fully weak 2-monad', where $T$, $j$ and $m$ are possibly not strict, and a `fully strict 2-monad' (or just 2-monad), where $\lambda$, $\rho$ and $\alpha$ are assumed to be identities \cite{blackwell1989two}.
\end{remark}


\subsection{The `free $\cat M$-actegory' pseudomonad}
\label{subsec:m-actegory-monad}
If $(\cat M, j, \mtimes, \lambda, \rho, \alpha)$ is a monoidal category, all its data and coherence properties map directly to those of a pseudomonadic structure on $T=\cat M \times - : \Cat \to \Cat$, as we prove in Proposition~\ref{prop:m-times-is-pseudomonad}:
\begin{diagram}[sep=4ex]
	{\cat M \times -} && {\cat M \times \cat M \times -} && {\cat M \times \cat M \times -} && {\cat M \times -} \\
	\\
	&& {\cat M \times -} && {\cat M \times -}
	\arrow["{\cat M \times j \times -}"', Rightarrow, from=1-7, to=1-5]
	\arrow["{\mtimes \times -}"', Rightarrow, from=1-5, to=3-5]
	\arrow[""{name=0, anchor=center, inner sep=0}, Rightarrow, no head, from=3-5, to=1-7]
	\arrow["{j \times \cat M \times -}", Rightarrow, from=1-1, to=1-3]
	\arrow["{\mtimes \times -}", Rightarrow, from=1-3, to=3-3]
	\arrow[""{name=1, anchor=center, inner sep=0}, Rightarrow, no head, from=1-1, to=3-3]
	\arrow["{\rho \times -}", triple, shift right=2, shorten <=8pt, shorten >=8pt, from=1-5, to=0]
	\arrow["{\lambda \times -}"', triple, shift left=3, shorten <=4pt, shorten >=8pt, from=1-3, to=1]
\end{diagram}

\begin{diagram}
	{\cat M \times \cat M \times \cat M \times -} && {\cat M \times \cat M \times -} \\
	\\
	{\cat M \times \cat M \times -} && {\cat M \times -}
	\arrow["{\mtimes \times \cat M \times -}"', Rightarrow, from=1-1, to=3-1]
	\arrow["{\cat M \times \mtimes \times -}", Rightarrow, from=1-1, to=1-3]
	\arrow["{\mtimes \times -}", Rightarrow, from=1-3, to=3-3]
	\arrow["{\mtimes \times -}"', Rightarrow, from=3-1, to=3-3]
	\arrow["{\alpha \times -}"', triple, shift left=2, shorten <=20pt, shorten >=20pt, from=1-3, to=3-1]
\end{diagram}

\begin{remark}
	It is a very important fact for this paper that the pseudomonad $\cat M \times -$ can always be strictified to the 2-monad $\cat M_\strict \times -$, where $\cat M_\strict$ is the strictification of the monoidal category $\cat M$ given by MacLane's theorem~\cite[Theorem 1.2.23]{johnson2021}.
	So even though we are, in principle, working with pseudomonads on $\Cat$, in practice we'll reap results about 2-monads (especially from~\cite{blackwell1989two}) by appealing to this strictification.
\end{remark}

\subsection{Algebras for pseudomonads}
Algebras of higher monads also come in various degrees of weakness.
We settle with the following general definition:

\subsubsection{Definition.}
\label{def:pseudoalgebra}
\begin{resdefinition}{defpseudoalgebra}
	A \textbf{pseudoalgebra} for a pseudomonad $(T, j, m)$ on $\twocat X$ consists of an object $x : \twocat X$ and a morphism
	\begin{equation}
		t : Tx \longto x
	\end{equation}
	in $\twocat X$ equipped with invertible 2-cells
	\begin{diagram}[sep=3ex, ampersand replacement=\&]
		x \&\& Tx \&\& TTx \&\& Tx \\
		\\
		\&\& x \&\& Tx \&\& x
		\arrow[""{name=0, anchor=center, inner sep=0}, Rightarrow, no head, from=1-1, to=3-3]
		\arrow["{j_x}", from=1-1, to=1-3]
		\arrow["t", from=1-3, to=3-3]
		\arrow["t"', from=3-5, to=3-7]
		\arrow["t", from=1-7, to=3-7]
		\arrow["{m_x}"', from=1-5, to=3-5]
		\arrow["Tt", from=1-5, to=1-7]
		\arrow["\mu"', shorten <=17pt, shorten >=17pt, Rightarrow, from=1-7, to=3-5]
		\arrow["\eta", shorten <=6pt, Rightarrow, from=0, to=1-3]
	\end{diagram}
\end{resdefinition}
\noindent satisfying some coherence equations (see Appendix~\ref{appendix:defs}).

When $\eta$ and $\mu$ are not invertible, we call the algebra \emph{lax} (or, dually, \emph{oplax}); if they are identities, we say the pseudoalgebra is \textbf{strict}.
For $T=\cat M \times -$, strict algebras (strict $\cat M$-actegories) can be hard to come by.
Still, every actegory is equivalent (in a suitable sense) to a strict one (Lemma~\ref{lemma:strictification}), hence we can make use of this fact to greatly simplify some proofs.

Instead, there is no way to avoid working with `pseudomorphisms' of algebras---similarly to the case of monoidal categories, once you have fixed two objects, the strict morphisms aren't comprehensive enough to exhaust the needs of the theory.

\subsubsection{Definition.}
\label{def:morphism-of-pseudoalgebras}
\begin{resdefinition}{deflaxmorphism}
	Let $(T,j,m)$ be a pseudomonad on $\twocat X$ and let $(x, t, \eta, \mu)$ and $(y, t', \eta', \mu')$ be pseudoalgebras.
	A \textbf{lax morphism} $(f,\ell): (x, t, \eta, \mu) \to (y, t', \eta', \mu')$ consists of a morphism $f: x \to y$ and a 2-cell
	\begin{diagram}[sep=3.5ex, ampersand replacement=\&]
		Tx \&\& Ty \\
		\\
		x \&\& y
		\arrow["t"', from=1-1, to=3-1]
		\arrow["t'", from=1-3, to=3-3]
		\arrow["f"', from=3-1, to=3-3]
		\arrow["Tf", from=1-1, to=1-3]
		\arrow["\ell"', shorten <=17pt, shorten >=17pt, Rightarrow, from=1-3, to=3-1]
	\end{diagram}
\end{resdefinition}
\noindent \noindent satisfying suitable coherence axioms (see Appendix~\ref{appendix:defs}).

A lax morphism $(f, \ell)$ is called \textbf{strong} if $\ell$ is invertible, and \textbf{strict} if $\ell$ is an identity.
When we omit any qualification, we default to strong.

\begin{definition}
\label{def:pseudoalg-transformation}
  Let $(x, t, \eta, \mu)$ and $(y, t', \eta', \mu')$ be pseudoalgebras, and let $(f,\ell)$ and $(g,\ell')$ be lax morphisms $(x, t, \eta, \mu) \to (y, t', \eta', \mu')$.
	A \textbf{transformation of morphisms} is a 2-cell $\xi: f \twoto g$  so that the following axiom holds:

	\begin{diagram}[sep=3.5ex]
		Tx && Ty && Tx && Ty \\
		\\
		x && y && x && y
		\arrow[""{name=0, anchor=center, inner sep=0}, "Tf"{description}, curve={height=-18pt}, from=1-1, to=1-3]
		\arrow[""{name=1, anchor=center, inner sep=0}, "Tg"{description}, curve={height=18pt}, from=1-1, to=1-3]
		\arrow["t"', from=1-1, to=3-1]
		\arrow[""{name=2, anchor=center, inner sep=0}, "f"{description}, from=3-1, to=3-3]
		\arrow[""{name=3, anchor=center, inner sep=0}, "t'", from=1-3, to=3-3]
		\arrow[""{name=4, anchor=center, inner sep=0}, "Tf"{description}, from=1-5, to=1-7]
		\arrow[""{name=5, anchor=center, inner sep=0}, "t"', from=1-5, to=3-5]
		\arrow[""{name=6, anchor=center, inner sep=0}, "f"{description}, curve={height=-18pt}, from=3-5, to=3-7]
		\arrow["t'", from=1-7, to=3-7]
		\arrow[""{name=7, anchor=center, inner sep=0}, "g"{description}, curve={height=18pt}, from=3-5, to=3-7]
		\arrow["T\xi", shorten <=5pt, shorten >=5pt, Rightarrow, from=0, to=1]
		\arrow["{\ell'}", shorten <=6pt, shorten >=6pt, Rightarrow, from=1, to=2]
		\arrow["\ell"{pos=0.4}, shorten <=6pt, shorten >=9pt, Rightarrow, from=4, to=6]
		\arrow["\xi", shorten <=5pt, shorten >=5pt, Rightarrow, from=6, to=7]
		\arrow[shorten <=24pt, shorten >=24pt, Rightarrow, no head, from=3, to=5]
	\end{diagram}
\end{definition}

For any given pseudomonad $T$, pseudoalgebras, algebra morphisms and transformations thereof assemble into various bicategories, depending on the degree of laxity used:

\begin{definition}
\label{def:talg}
	Given a pseudomonad $T$ on $\twocat X$, there is a 2-category $\Alg{T}$ of pseudoalgebras, strong morphisms and transformations \cite[Proposition 4.2]{lack_coherent_2000}.
	If $T$ is a 2-monad, we further have the 2-categories $\Alg{T}_\strict$ of strict algebras, strong morphisms and transformations, and $\Alg{T}_\strict^\strict$ of strict algebras, strict morphisms and tranformations \cite[§1]{blackwell1989two}.
	Clearly there are forgetful functors
	\begin{equation}
		\Alg{T}_\strict^\strict \longto \Alg{T}_\strict \longto \Alg{T} \longto \twocat X.
	\end{equation}
\end{definition}

To denote the various combinations of lax/oplax/strong/strict objects and morphisms, we adopt a notation based on subscripts and superscripts:
\begin{equation}
	\Alg{T}_{\text{laxity of the objects}}^{\text{laxity of the morphisms}}
\end{equation}

The symbols denoting laxity are described in the following table:

\begin{center}
	\begin{tabular}{c|c}
		\textbf{laxity} & \textbf{symbol}\\
		\hline
		lax & $\lax$\\
		oplax & $\oplax$\\
		strong/pseudo & $\pseudo$, (empty string)\\
		strict & $\strict$
	\end{tabular}
\end{center}

This notation is compatible with that used in \cite{blackwell1989two}, except they denote `lax' simply with $l$.


	\newpage
	\section{Actegories and their morphisms}
\label{sec:actegories}

Having defined the necessary pseudoalgebraic machinery in the previous section, we proceed to build on top of it and define $\cat M$-actegories and their morphisms.
We give a number of examples, a strictification lemma, describe the free and cofree adjunction, and lastly, describe the indexed 2-category of all actegories.

\subsection{Actegories}

As anticipated, we can develop the algebraic theory of $\cat M$-actegories systematically by instantiating the definitions of pseudoalgebras and their morphisms for $T=\cat M \times -$.

\begin{definition}
\label{def:left-actegory}
	Let $\cat M$ be a monoidal category. A \textbf{left $\cat M$-actegory} $\cat C$ (or \textbf{left $\cat M$-action}) is a category $\cat C$ equipped with a functor
	\begin{equation}
		-\action {=} : {\cat M} \times {\cat C} \longto {\cat C},
	\end{equation}
	and two natural isomorphisms
	\begin{equation}
		\eta_c : c \isolongto j \action c, \qquad \mu_{m,n,c} : m \action (n \action c) \isolongto (m \mtimes n) \action c, \qquad \text{for $m, n : \cat M$, $c:\cat C$},
	\end{equation}
  respectively called \emph{unitor} and \emph{multiplicator}, satisfying the following coherence laws:
	\begin{diagram}[row sep=4ex, column sep=-4.5ex]
	\label{diag:action-coherence-pentag}
		& {m \action (n \action (p \action c))} && {(m \mtimes n) \action (p \action c)} \\
		{m \action ((n \mtimes p) \action c)} &&&& {((m \mtimes n) \mtimes p) \action c} \\
		&& {(m \mtimes (n \mtimes p)) \action c}
		\arrow["{\mu_{m, n, p \action c}}", from=1-2, to=1-4]
		\arrow["{\mu_{m \mtimes n, p, c}}", from=1-4, to=2-5]
		\arrow["{\alpha_{m, n, p} \action c}", from=2-5, to=3-3]
		\arrow["{m \action \mu_{n, p, c}}"', from=1-2, to=2-1]
		\arrow["{\mu_{m, n \mtimes p, c}}"', from=2-1, to=3-3]
	\end{diagram}
	\begin{diagram}
	\label{diag:action-coherence-ltriang}
		{j \bullet (m \bullet c)} && {(j \mtimes m) \action c} \\
		& {m \action c}
		\arrow["{\eta_{m \action c}}", from=2-2, to=1-1]
		\arrow["{\mu_{j, m, c}}", from=1-1, to=1-3]
		\arrow["{\lambda^{-1}_m\action c}"', from=2-2, to=1-3]
	\end{diagram}
	\begin{diagram}
	\label{diag:action-coherence-rtriang}
		{m \action (j \action c)} && {(m \mtimes j) \action c} \\
		& {m \action c}
		\arrow["{m \action \eta_c}", from=2-2, to=1-1]
		\arrow["{\rho^{-1}_m \action c}"', from=2-2, to=1-3]
		\arrow["{\mu_{m, j, n}}", from=1-1, to=1-3]
	\end{diagram}
  for all $m, n, p : \cat M$ and $c : \cat C$.
\end{definition}

\begin{remark}
	Recall that the triangle axioms for monoidal categories are redundant: any of them is redundant in the presence of the other one and the pentagon.
	The same applies to actegories.
	Indeed, by inspecting Kelly's simplification of MacLane's coherence conditions for monoidal categories \cite{kelly1964maclane}, one can realize that part of his proof can be reproduced \emph{mutatis mutandis} for actegories, so that either of Diagram~\eqref{diag:action-coherence-ltriang} or \eqref{diag:action-coherence-rtriang} can be omitted.
	We learned of this fact from \cite{mccrudden2000categories}.
\end{remark}

\begin{remark}
	Analogous to \emph{left} actegories, it is possible to define \emph{right} actegories.
	Let $c : \cat C$ and $m, n : \cat M$. When the action is on the left side, successive actions combine as follows:
	\begin{equation}
		m \action (n \action c) \iso (m \mtimes n) \action c.
	\end{equation}
	On the other hand, when the action is on the right side, actions combine as:
	\begin{equation}
		(c \action n) \action m \iso c \action (n \mtimes m).
	\end{equation}
	Comparing the two equations above should make the difference clear: multiplying by $m$ and then $n$ is like multiplying by $m \mtimes n$ for left actions, and by $n \mtimes m$ for right actions.
	Indeed, a right $\cat M$-actegory is exactly the same as a left $\cat M^\rev$-actegory, where $\cat M^\rev$ is the monoidal category obtained by equipping the underlying category of $\cat M$ with the product $m \revrel{\mtimes} n := n \mtimes m$ and suitably tweaking the rest of the structure \cite[Example 1.2.9]{johnson2021}.

	Notice that right $\cat M$-actegories can also be conceived as pseudoalgebras for the same endofunctor $\cat M \times -$ but \emph{equipped with a different monad structure}, namely the one where multiplication is precomposed with a symmetry:
	\begin{diagram}
		{\cat M \times \cat M \times -} & {\cat M \times \cat M \times -} & {\cat M \times -}
		\arrow["{\swap \times -}", Rightarrow, from=1-1, to=1-2]
		\arrow["{\mtimes \times -}", Rightarrow, from=1-2, to=1-3]
	\end{diagram}
	Indeed, $\tilde \mtimes := \mathrm{swap} \comp \mtimes$, so that the pseudomonad of left $\cat M^\rev$-actegories is identical to this one, which we might rightly call the \textbf{pseudomonad of right $\cat M$-actegories}.

	As customary in algebra, whenever we do not qualify an action with an handedness we default to left.
\end{remark}

\begin{remark}
\label{rmk:diactegories}
	If a category receives both a left and a right action (i.e.~a $\cat M$-action and a $\cat N^{\rev}$-action), and these interact `nicely', we have a \emph{biactegory}, a structure we treat extensively in Section~\ref{sec:biactegories}.
	Likewise, a category may also receive a `forward' and a `backward' action (i.e.~an $\cat M$-action and an $\cat N^\op$-action) which interact nicely.
	These deserve the name of \textit{diactegories}, but we leave their study for future work.
\end{remark}

\begin{remark}
	Contrary to the praxis in the related field of Tambara theory \cite{pastrostreet, tambara}, in the optics literature~\cite{boisseau, categoricalupdate,roman2020profunctor} definitions implicitly assume every action to be on the left.
	This doesn't make much of a difference when $\cat M$ is symmetric (a common occurence in practice), and when the actions are not lax, but this obfuscates a nice correspondence to algebra in practice.
	In particular, there is a distinction between left and right actions in the algebra of $\Para$ and $\Copara$ \cite[Definition 2]{capucci2021towards}.
\end{remark}

We must mention an equivalent presentation of the structure of actegory, which is handy to have in mind:

\begin{proposition}
\label{prop:curried-action}
	The data of a left $\cat M$-actegory is equivalent to that of a strong monoidal functor, in the sense that there is a bijection of sets:
	\begin{equation}
			\{\,\text{$\cat M \times \cat C \longto \cat C$\; left action}\,\}
		\quad
		\iso
		\quad
			\{\,\text{$\cat M \longto [\cat C, \cat C]$\; strong monoidal}\,\}
	\end{equation}
	The equivalence is supported by the tensor-hom adjunction in $\Cat$.
\end{proposition}
\begin{proof}
	It is a routine matter to verify that $\eta$ and $\mu$ for a left action $\action$ correspond to a strong monoidal structure on its currying $\curr(\action)$, and that Diagrams~\eqref{diag:action-coherence-pentag}-\eqref{diag:action-coherence-rtriang} are equivalent to the laws obeyed by such structure (see \cite[§XI.2]{WorkingMathematician}).
\end{proof}

\begin{remark}
	This proposition can be read as saying `$[\cat C, \cat C]$ classifies actions on $\cat C$'.
	Indeed, in Section~\ref{subsec:classifying} we are going to extend this result to a full equivalence of categories, as well as generalizing it to more structured notions of actegory.
\end{remark}

\begin{remark}
	The presentation of actegories as monoidal functor into $[\cat C, \cat C]$ is preferred in the literature on \textbf{graded (co)monads} \cite{GradedMonads}.
	In that case, the monoidal functor $\cat M \to [\cat C, \cat C]$ is more often considered to be just lax (for graded monads) or oplax (for graded comonads).
	These correspond to lax and oplax algebras of $\cat M \times -$ (considered as a pseudomonad) from our point of view.
	Much of the definitions we give here still work for these weaker structures, but many of the results do not hold for them.
	Moreover, in the applications we are interested in this is not a limitation, hence in this work we focus on \emph{pseudo}algebras.
\end{remark}

\begin{remark}
	Despite making it easier to state the definition of actegories, the curried presentation of monoidal actions is considerably less elegant when it comes to their morphisms, 2-morphisms and other constructions.
	This can be directly motivated by observing most of the facts we care about for $\cat M$-actegories are instances of general algebraic theories, while Proposition~\ref{prop:curried-action} is a very peculiar property of actegories.
	Therefore we'll continue to prefer Definition~\ref{def:left-actegory}.
\end{remark}

As prescribed by the microcosm principle, actegories can be seen as the structure necessary to define the structure of an `action of a monoid' in a category, in the same way monoidal categories are the minimal setting in which monoids themselves can be defined.

\begin{definition}
\label{def:action-in-actegory}
	Let $(\cat C, \action)$ be an $\cat M$-actegory, and let $c : \cat C$ and $m : \cat M$.
	An action of $m$ on $c$ in $(\cat C, \action)$ is an arrow $\ast : m \action c \to c$ such that the following commute (compare them with Diagrams~\eqref{diag:monoid-action-unit}--\eqref{diag:monoid-action-comp}):
	\begin{diagram}
	\label{diag:internal-monoid-action}
		{m \action c} & {j \action c} & {(m \mtimes m) \action c} && {m \action (m \action c)} \\
		c & c & {m \action c} & c & {m \action c}
		\arrow["\ast"', from=1-1, to=2-1]
		\arrow["{i \action c}"', from=1-2, to=1-1]
		\arrow["{\eta_c}"', from=2-2, to=1-2]
		\arrow[Rightarrow, no head, from=2-1, to=2-2]
		\arrow["{{\cdot} \action c}"', from=1-3, to=2-3]
		\arrow["\ast"', from=2-3, to=2-4]
		\arrow["{\mu_{m,m,c}}"', from=1-5, to=1-3]
		\arrow["{m \action \ast}", from=1-5, to=2-5]
		\arrow["\ast", from=2-5, to=2-4]
	\end{diagram}
\end{definition}

In the same fashion we can define coactions of comonoids and biactions of bimonoids, for which actegories remain the natural setting.\footnote{Though biactions could be further abstracted as to take place in biactegories instead.}

\subsection{Examples}
We describe here some examples, some of which are taken from the literature, although with no pretence to be exhaustive.
In particular, we omit most of the examples from categorical algebra (e.g.~those described in \cite{ostrik2003module}).

\begin{example}[Trivial actegories]
\label{ex:trivial-action}
	Every category $\cat C$ is trivially a $1$-actegory, where $1$ is the one-object category whose monoidal product is defined in the only way possible.
	The action $\action : 1 \times \cat C \to \cat C$ is an isomorphism, and the unitor and multiplicator are trivial.

	This can be generalised: every category $\cat C$ is trivially an $\cat M$-actegory for any monoidal category $\cat M$, with the action given by the projection $\pi_{\cat C} : \cat M \times \cat C \to \cat C$.
\end{example}

\begin{example}[Free actegories]
\label{ex:free-act}
	Applying the pseudomonad of left $\cat M$-actegories to a category $\cat C$ yields the free actegory on $\cat C$, whose multiplication is thus the same as the pseudomonad:
	\begin{equation}
		\cat M \times \cat M \times \cat C \xrightarrow{\otimes \times -} \cat M \times \cat C
	\end{equation}
	Likewise for the unitors and multiplicators, which come from $\cat M \times -$ and thus, ultimately, from $\cat M$ itself.
\end{example}

\begin{example}[Monoid action]
	Any action of a monoid $M$ on a set $X$ is a `discrete actegory'.
	That is, it is a $\operatorname{disc}(M)$-actegory $\operatorname{disc}(X)$, where $\operatorname{disc} : \Set \to \Cat$ casts a set as a discrete category.
	In \cite{ChangeActionsThesis, ChangeActionModels} monoid actions are called \textit{change actions}, although their maps are different.
\end{example}

\begin{example}[Monoidal category]
\label{ex:moncat-self-actions}
	Every monoidal category $\cat M$ is canonically a left and right $\cat M$-actegory. The tensor ${- \mtimes =} : \cat M \times \cat M \to \cat M$, equipped with the left unitor as counitor and the associator as comultiplicator, defines a left action.
	On the right, one has the `same' functor but typed as ${= \mtimes -} : \cat M^\rev \times \cat M \to \cat M$, equipped with the right unitor as counitor and the inverse of the associator as comultiplicator.
	We call these actions \emph{canonical left} and \emph{canonical right self-action}, respectively.
\end{example}

\begin{example}[Action of natural numbers]
\label{ex:natural-numbers-act}
	Any monoidal category $(\cat M, j, \mtimes)$ is naturally an $\N$-actegory, where $\N$ is considered as a discrete category with multiplication as monoidal structure.
	Given $n : \cat M$ and $c : \cat C$, the action $(=)^{\mtimes -} : \N \times \cat M \to \cat M$ takes the pair $(n, c)$ to the $n$-fold monoidal product of $c$ with itself, i.e.
	\begin{equation}
		c^{\mtimes n} :=\; \underbrace{c \mtimes \dots \mtimes c}_{\text{$n$ times}}.
	\end{equation}
	It is clear then that $(c^{\otimes n})^{\otimes m} \iso c^{\otimes mn}$ is given by the associator of $\cat M$ while $c^{\otimes 0} = j$ by definition. Coherence for monoidal categories implies coherence of this data.
\end{example}

\begin{example}[Evaluation]
\label{ex:endomorphisms}
	Every category $\cat C$ is a $([\cat C, \cat C], 1_{\cat C}, \circ)$-actegory, under the action given by evaluation. In fact, this is the action corresponding to the identity functor of $[\cat C, \cat C]$ under the equivalence described in Proposition~\ref{prop:curried-action}.
	This also means than any monoidal subcategory of endofunctors over $\cat C$ acts on $\cat C$.
	For instance, the action of applicative functors (a class of endofunctors definable when $\cat C$ is symmetric monoidal) can be used to define a class of optics called \emph{kaleidoscopes} \cite[§3.2.2]{categoricalupdate}.
\end{example}


\begin{example}[Copower]
\label{ex:copower}
	When a category $\cat C$ has all small coproducts, there is an action $\otimes : \Set \times \cat C \to \cat C$ of $(\Set, 1, \times)$
	\begin{equation}
		A \copow c := \coprod_{a \in A} c
	\end{equation}
  where $A : \Set$ and $c : \cat C$.
	This action is called \emph{$\Set$-copower} and enjoys the universal property of being `left adjoint to the hom functor':
	\begin{equation}
	\label{eq:tensoring}
		\cat C(A \copow c, y) \iso \Set(A, \cat C(c, y)).
	\end{equation}
	In general, we can replace $\Set$ with any closed monoidal category $\cat V$, and talk about \emph{$\cat V$-copowered} (or \emph{$\cat V$-tensored}) \emph{enriched categories} \cite{kelly1982basic}.
	In \cite{janelidze2001note}, the authors prove that $\cat V$-copowered enriched categories are equivalent to $\cat V$-actegories $(\cat C, \action)$ such that $- \action c$ has a right adjoint for each $c : \cat C$ (so called \emph{closed actegories}).
\end{example}

\begin{example}[Power]
	Dually, if $\cat C$ has all products one gets an action $\pow : \Set^\op \times \cat C \to \cat C$ given by
	\begin{equation}
		A \pow c := \prod_{a \in A} c
	\end{equation}
	which enjoys the following universal property:
	\begin{equation}
		\cat C(d, A \pow c) \iso \Set(A, \cat C(d,c))
	\end{equation}
	In \cite[Chapter 1, §7]{wood1976} Wood proves that $\cat M$-actegories are equivalent to `$\cat M$-powered' $[\cat M^\op, \Set]$-categories.\footnote{We learned of this fact from \cite{garner2018embedding}.}
\end{example}

\begin{example}[Day convolaction]
\label{ex:day}
	A given $\cat M$-action $\action$ on $\cat C$ can be canonically extended to an action of $[\cat M, \Set]$ on $[\cat C, \Set]$.
	First of all, the monoidal structure of $\cat M$ extends canonically to a monoidal structure on its category of copresheaves $[\cat M, \Set]$, known as \emph{Day convolution} \cite{day1970closed}.\footnote{We recall the unit of this monoidal structure is $\cat M(j, -)$ while product is given as
	\begin{equation}
	\label{eq:day-prod}
		M \otimes_\Day N = \int^{m,n : \cat M} \cat M(m \otimes n, -) \times M(m) \times N(n).
	\end{equation}}
	The extension of $\action$ is given, \emph{mutatis mutandis}, in the same way:
	\begin{equation}
		M \action_\Day P := \int^{m : \cat M, c : \cat C} \cat C(m \action c, -) \times M(m) \times P(c).
	\end{equation}
	We call this action \emph{Day convolaction}.
	The unitor of the action is given by repeated Yoneda reduction~\cite[Proposition~2.2.1]{loregian_coend_2021} and by using the unitor of $\action$:
	\begin{eqalign}
		\cat M(j, -) \action_\Day P &= \int^{m : \cat M, c : \cat C} \cat C(m \action c, -) \times \cat M(j, m) \times P(c)\\
		&\iso \int^{c : \cat C} \cat C(j \action c, -) \times P(c)\\
		&\iso \int^{c : \cat C} \cat C(c, -) \times P(c)\\
		&\iso P.
	\end{eqalign}
	The multiplicator uses Yoneda reduction, the multiplication for $\action$, the definition of $\otimes_\Day$~\eqref{eq:day-prod} and Fubini's rule~\cite[§1.3]{loregian_coend_2021}:
	\begin{eqalign}
		M \action_\Day &(M' \action_\Day P) = \int^{m : \cat M, c' : \cat C} \cat C(m \action c', -) \times M(m) \times (M' \action_\Day P)(c')\\
		&\iso \int^{m' : \cat M, c : \cat C, m : \cat M, c' : \cat C} \cat C(m \action c', -) \times \cat C(m' \action c, c') \times M(m) \times M'(m') \times P(c)\\
		&\iso \int^{m' : \cat M, c : \cat C, m : \cat M} \cat C(m \action (m' \action c), -) \times M(m) \times M'(m') \times P(c)\\
		&\iso \int^{m' : \cat M, c : \cat C, m : \cat M} \cat C((m \otimes m') \action c, -) \times M(m) \times M'(m') \times P(c)\\
		&\iso \int^{m' : \cat M, c : \cat C, m : \cat M, k : \cat M} \cat C(k \action c, -) \times \cat M(m \otimes m', k) \times M(m) \times M'(m') \times P(c)\\
		&\iso \int^{c : \cat C, k : \cat M} \cat C(k \action c, -) \times (M \otimes M')(k) \times P(c)\\
		&\iso (M \otimes_\Day M') \action_\Day P.
	\end{eqalign}
	As for Day convolution, this action is well-defined also when $\Set$ is replaced with a B\'enabou cosmos \cite{street1974elementary}, such as any quantale.
\end{example}

\begin{example}[{\cite[Theorem 9]{garner2018embedding}}]
\label{ex:tangent}
	In \textit{ibid.}, Garner proves that tangent categories are a special kind of actegory.
	Let $\cat W$ be the `Weil category' the subcategory of that of commutative rigs spanned by objects isomorphic to tensor products of rigs of the form $W_k := \N[x_1, \ldots, x_k]/(x_ix_j, i \leq j \leq k)$.
	A tangent structure on a category $\cat C$ is then a $\cat W$-action $\action$ such that $- \action c$ preserves a certain class of limits (called `tangent limits') for all $c : \cat C$.
\end{example}

\begin{example}[Metric dilation]
	\label{ex:metric-spaces}
	The category $\Met$ of metric spaces and short maps \cite[p.~38]{deza_encyclopedia_2009} receives an action of $(\R^+, 1, \cdot)$, the strict monoidal category of positive real numbers (morphisms are inequalities $\geq$). A scalar $m : \R^+$ acts on a metric space $(X, d)$ by scaling: the resulting metric space $(X, m \action d)$ is equipped with a metric scaled exactly by $m$.
	Maps are left unchanged, since their shortness is not affected by scaling both domain and codomain equally.

\end{example}

\begin{example}[Stochastic action]
\label{ex:stochastic-maps}
	Let $\Msbl$ be the category of measurable spaces and measurable maps between them, and let $\Prob$ be the category of probability spaces and measure preserving maps between them.
	There is an action of $\Prob$ on $\Msbl$ given by
	\begin{equation*}
		(\Omega, \mathcal F, \mathbb P) \action (X, \Sigma_X) := (\Omega \times X, \mathcal F \otimes \Sigma_X)
	\end{equation*}
	where on the right we have the product of measurable spaces.
	This action can be used to model stochastic processes, after being fed to the functor $\Para$, as described in \cite{paratalk}.
\end{example}

\begin{example}[{\cite[Example 2.1.8]{stefanou2018dynamics}}]
\label{ex:stefanou}
	Let $\Top$ be the category of topological spaces and continuous maps, fix a space $X$ and consider the frame $\mathcal O(X)$ of the open sets of its canonical topology, which is a monoidal category with product $\cap$ and unit $X$.
	There's an action of $(\mathcal O(X), X, \cap)$ on $\Top / X$ given on objects by
	\begin{equation}
		U \action (f: Y \to X) = f^{-1}U \text{ with the subspace topology}
	\end{equation}
	and on maps $\varphi : f \to g$, where $f:Y \to X$ and $g:Z \to X$ by
	\begin{equation}
		\varphi\vert_U : f^{-1}U \to g^{-1}U.
	\end{equation}
	To see this is well-defined, notice that since $\varphi \comp g = f$, the preimage of $g^{-1}U$ along $\varphi$ is exactly $f^{-1}U$.
	Clearly $U \action (V \action f) = (U \cap V) \action f$ and $X \action f = f$, so this is a strict $\mathcal O(X)$-action.
\end{example}

\subsection{Morphisms}
We have unpacked how pseudoalgebras of $\cat M \times -$ give $\cat M$-actegories.
We continue to unpack the definitions given in Section~\ref{sec:monads} to get to a full-fledged definition of the 2-category of $\cat M$-actegories. This is the 2-category $\Alg{T}^\lax$ (Definition~\ref{def:talg}) instatiated for $T=\cat M \times -$.

\begin{definition}
	For a fixed monoidal category $\cat M$ we denote the 2-category of $\cat M$-actegories, lax $\cat M$-linear functors and $\cat M$-linear transformations as $\Act{\cat M}^\lax$.
\end{definition}

To get morphisms of $\cat M$-actegories, it is sufficient to instantiate Definition~\ref{def:morphism-of-pseudoalgebras} for the pseudomonad $\cat M \times -$:

\begin{definition}
\label{def:linear-functor}
	Let $(\cat C, \action, \eta^\action, \mu^\action)$ and $(\cat D, \actionn, \eta^\actionn, \mu^\actionn)$ be left $\cat M$-actegories.
	A \textbf{lax $\cat M$-linear functor} between them is a functor $F:\cat C \to \cat D$ equipped with a natural morphism
	\begin{equation}
		\ell_{m,c} : m \actionn F(c) \longto F(m \action c) \qquad \text{for $m : \cat M$, $c:\cat C$},
	\end{equation}
	called \textbf{lineator} satisfying the following coherence laws for all $m, n : \cat M$ and $c : \cat C$:
	\begin{diagram}[row sep=4ex, column sep=-3ex]
	\label{diag:linear-coherence-pentag}
		& {m \actionn (n \actionn F(c))} && {m \actionn F(n \action c)} \\
		{(m \mtimes n) \actionn F(c)} &&&& {F(m \action (n \action c))} \\
		&& {F((m \mtimes n) \action c)}
		\arrow["{\mu^\actionn_{m, n, F(c)}}"', from=1-2, to=2-1]
		\arrow["{\ell_{m \mtimes n, c}}"', from=2-1, to=3-3]
		\arrow["{m \actionn \ell_{n, c}}", from=1-2, to=1-4]
		\arrow["{\ell_{m, n \action c}}", from=1-4, to=2-5]
		\arrow["{F(\mu^\action_{m, n, c})}", from=2-5, to=3-3]
	\end{diagram}
	\begin{diagram}
	\label{diag:linear-coherence-triang}
		{j \actionn F(c)} && {F(j \action c)} \\
		& {F(c)}
		\arrow["{\eta^\actionn_{F(c)}}", from=2-2, to=1-1]
		\arrow["{\ell_{j, c}}", from=1-1, to=1-3]
		\arrow["{F(\eta^\action_c)}"', from=2-2, to=1-3]
	\end{diagram}
	When $\ell$ is invertible, we call $(F, \ell)$ simply \textbf{$\cat M$-linear functor}, or \textbf{strong $\cat M$-linear functor} if we want to emphasize that.
\end{definition}

Morphisms between actegories on different monoidal categories will be defined in Proposition~\ref{prop:actegories-indexed-over-moncat}.

\begin{remark}
\label{rmk:coherence-as-inductive-def}
	The coherence diagrams in this definition (and in some of the later ones) can be interpreted as inductive definitions for the structure morphisms (here $\ell$).
	It is easy to spot this phenomenon in the above triangular diagram, where $\ell_{j,c}$ is being pinned to be ${\eta^\actionn_{F(c)}}^{-1} \comp F(\eta^\action_c)$.
	The same can be said for the pentagon, though: it is a definition of $\ell_{m \mtimes n, -}$ in terms of $\ell_{m,-}$ and $\ell_{n,-}$:
	\begin{equation}
	\label{eq:lineator-inductive-def}
		\ell_{m \mtimes n, c} = {\mu_{m,n,F(c)}^\actionn}^{-1} \comp (m \actionn \ell_{n,c}) \comp \ell_{m,n \action c} \comp F(\mu^\action_{n,m,c}).
	\end{equation}
	This gets particularly interesting when $\cat C$ and $\cat D$ are strict actegories, in which case $\ell_{j,c}$ is forced to be the identity and $\ell_{m \mtimes n, c}$ is simply $(m \actionn \ell_{n,c}) \comp \ell_{m,n \action c}$.
\end{remark}

\begin{example}[Monoidal functors are $\N$-linear]
\label{ex:monoidal-functors-n-linear}
	Following up from Example~\ref{ex:natural-numbers-act}, let $(F, \epsilon, \mu) : (\cat M, j, \mtimes) \to (\cat N, i, \ntimes)$ be a lax monoidal functor.
	Then we can form a $\N$-linear functor $(F, \mu^n) : (\cat M, (=)^{\mtimes -}) \to
	(\cat N, (=)^{\ntimes -})$, whose lineator is defined by repeated applications of the laxator $\mu$:
	\begin{equation}
		\mu^n : \underbrace{F(m) \ntimes \dots \ntimes F(m)}_{\text{$n$ times}} \to \underbrace{F(m \mtimes \dots \mtimes m)}_{\text{$n$ times}}.
	\end{equation}
  where $m : \cat M$.
	For $n=0$, we set $\mu^0 = \epsilon$.
\end{example}

\begin{example}[Tensorial strengths]
	\label{ex:tensorial_strength}
	Sometimes, a lax monoidal structure on an endofunctor (often, a monad) on a monoidal category can be too strong of a condition, in which case the structure of a \emph{tensorial strength} is often more suitable \cite{kock1972strong}.
	Indeed, let $\cat M$ be a monoidal category and $F : \cat M \to \cat M$ an endofunctor.
	Then the structure of a left tensorial strength $\beta_{m, m'} : m \otimes F(m') \to F(m \otimes m')$ on $F$ gives rise to a $\cat M$-linear endomorphism of $\cat M$ with the canonical left self-action (Example~\ref{ex:moncat-self-actions}).
	Likewise, right tensorial strengths on $F$ give rise to analogous endomorphisms of right self-actions.
\end{example}

\begin{example}
\label{ex:hausdorff-meas}
	Let $\Meas$ be the category of measure spaces and non-expanding\footnote{A map $f:(X, \Sigma_X, \mu) \to (Y, \Sigma_Y, \nu)$ is non-expanding if $f_*\mu(A) \leq \nu(A)$ for every $A \in \Sigma_Y$.} maps. Analogously to Example~\ref{ex:metric-spaces}, $\R^+$ acts by dilation on this category, this time dilating the measure instead of the metric.

	For $d > 0$, there is a functor $H^q: \Met \to \Meas$, mapping each metric space $(X,d)$ to the measure space $(X, \mathcal B(X), H^q_X)$ where $\mathcal B(X)$ is the Borel $\sigma$-field on $X$ and $H^q$ is the \emph{$q$-dimensional Hausdorff measure} \cite{noauthor_hausdorff_nodate} defined on a given $A \in \mathcal B(X)$ as:
	\begin{equation*}
		H^q(A) = \lim_{\delta \to 0}\ \inf \{ \Sigma_{i=1}^\infty (\diam U_i)^q \suchthat \text{$\{U_i\}_{i \in \N}$ covers $A$ and $\diam U_i < \delta$ $\forall i \in \N$}\}.
	\end{equation*}
	Since maps in $\Met$ are short, it is easy to see the induced measurable maps are non-expanding.
	When $q=1$, this functor is $\R^+$-linear, that is, there are maps $r \cdot H^q(X, d) \to H^q(X, r \cdot d)$.
	In fact, on the domain we have the space $X$ with measure $r\cdot H^q$, while on the codomain we have the same space with measure $r^q \cdot H^q$. Since $r \leq r^q$ iff $r \geq 1$, one is forced to choose $q=1$.

	Suppose instead we define an action of $(\R^+, 1, \cdot, \geq)$ on $\Msbl$ to be $r \ast (X,d) := \e^r \cdot (X,d)$, and analogously for $\Meas$. Now the functors $H^q$ are all lax $\R^+$-linear since $\e^r \geq 1$ as long as $r \geq 0$. $H^1$ remains the only \emph{strong} $\R^+$-linear functor among them.
\end{example}

One of the main properties of lax linear functors is that they send actions to actions:

\begin{lemma}
\label{lemma:linear-preserves-actions}
	Let $(F, \ell): (\cat C, \action) \to (\cat D, \actionn)$ be a lax $\cat M$-linear functor.
	Given a monoid $(m, i:j \to m, \cdot : m \otimes m \to m)$ in $\cat M$ and an action $\ast : m \action c \to c$ of $m$ on $c:\cat C$ (as defined in Definition~\ref{def:action-in-actegory}), then
	\begin{equation}
		m \action F(c) \nlongto{\ell_{m,c}} F(m \action c) \nlongto{F(\ast)} F(c)
	\end{equation}
	is an action of $m$ on $F(c)$ in $\cat D$.
\end{lemma}
\begin{proof}
	That $\ell_{m,c} \comp F(\ast)$ is an action of $m$ on $F(c)$ is witnessed by the following diagrams:
	\begin{diagram}[sep=6ex]
		{m \action F(c)} && {j \action F(c)} \\
		{F(m \action c)} & {F(j \action c)} \\
		{F(c)} && {F(c)}
		\arrow["\ast"', from=2-1, to=3-1]
		\arrow["{F(i \action c)}"', from=2-2, to=2-1]
		\arrow[""{name=0, anchor=center, inner sep=0}, "{\eta_c}", from=3-3, to=2-2]
		\arrow[Rightarrow, no head, from=3-1, to=3-3]
		\arrow["{\ell_{m,c}}"', from=1-1, to=2-1]
		\arrow["{i \action F(c)}"', from=1-3, to=1-1]
		\arrow[""{name=1, anchor=center, inner sep=0}, "{\ell_{j,c}}"', from=1-3, to=2-2]
		\arrow["{\eta_{F(c)}}"', from=3-3, to=1-3]
		\arrow["{\text{(\ref{diag:internal-monoid-action} left)}}"{description}, draw=none, from=2-2, to=3-1]
		\arrow["{\text{(nat. of $\ell$)}}"{description}, shift right=3, draw=none, from=1-1, to=1]
		\arrow["{\eqref{diag:linear-coherence-triang}}"{description}, Rightarrow, draw=none, from=0, to=1-3]
	\end{diagram}
	\begin{diagram}[row sep=4ex]
		{(m \mtimes m) \actionn F(c)} && {m \actionn (m \actionn F(c))} \\
		\\
		{F((m \mtimes m) \action c)} & {F(m \action (m \action c))} & {m \actionn F(m \action c)} \\
		&& {m \actionn F(c)} \\
		{F(m \action c)} && {F(m \action c)} \\
		& {F(c)}
		\arrow["{\mu^\actionn_{m,m,F(c)}}"', from=1-3, to=1-1]
		\arrow["{F(\ast)}"', from=5-1, to=6-2]
		\arrow["{F(\ast)}", from=5-3, to=6-2]
		\arrow["{\ell_{m,c}}", from=4-3, to=5-3]
		\arrow[""{name=0, anchor=center, inner sep=0}, "{m \actionn F(\ast)}", from=3-3, to=4-3]
		\arrow[""{name=1, anchor=center, inner sep=0}, "{m \actionn \ell_{m, c}}", from=1-3, to=3-3]
		\arrow[""{name=2, anchor=center, inner sep=0}, "{F({\cdot}\action c)}"', from=3-1, to=5-1]
		\arrow[""{name=3, anchor=center, inner sep=0}, "{F(m \action \ast)}"'{pos=0.2}, from=3-2, to=5-3]
		\arrow["{F(\mu^\action_{m,m,c})}"', from=3-2, to=3-1]
		\arrow[""{name=4, anchor=center, inner sep=0}, "{\ell_{m \mtimes m, c}}"', from=1-1, to=3-1]
		\arrow["{\ell_{m, m \action c}}"', from=3-3, to=3-2]
		\arrow["{\text{(\ref{diag:internal-monoid-action} right)}}"{description, pos=0.4}, Rightarrow, draw=none, from=2, to=5-3]
		\arrow["{\text{(nat.~of $\ell$)}}"{description}, shift left=3, draw=none, from=3, to=0]
		\arrow["{\eqref{diag:linear-coherence-pentag}}"{description}, draw=none, from=4, to=1]
	\end{diagram}
\end{proof}

\begin{remark}
	Dually, oplax linear functors preserve \emph{co}actions of \emph{co}monoid, hence strong linear functors preserve Hopf modules, i.e.~biactions of bimonoids \cite[14.1]{brzezinski2003london}.
\end{remark}

Finally, we instance Definition~\ref{def:pseudoalg-transformation}:

\begin{definition}
\label{def:linear-trans}
	An \textbf{$\cat M$-linear transformation} $\xi : (F, \ell) \twoto (G, \nu)$ between two lax linear functors $F,G:(\cat C, \action) \to (\cat D, \actionn)$ is a natural transformation between the carrier functors of $F$ and $G$ such that
	\begin{diagram}
	\label{diag:lin-transf-coherence}
		m \actionn F(c) \arrow[swap]{d}{\ell_{m, c}} \arrow{r}{m \actionn \xi_c} & m
    \actionn G(c) \arrow{d}{\nu_{m, c}}\\
		F(m \action c) \arrow{r}{\xi_{m \action c}} & G(m \action c)
	\end{diagram}
	commutes for every $m : \cat M$, $c: \cat C$.
\end{definition}

\begin{example}
	Following up from Example~\ref{ex:monoidal-functors-n-linear}, a monoidal natural transformation $\alpha : R \twoto S$ gives rise to a $\N$-linear transformation whose linearity condition follows from monoidality of $\alpha$.
\end{example}

\subsection{Strictification}
General actegories are \emph{pseudo}algebras for $\cat M \times -$.
We call \emph{strict} algebras of $\cat M \times -$ \textbf{strict actions} or \textbf{strict actegories}.
The reader might be glad to learn every actegory can be strictified.
This result is the actegorical analogue of MacLane's coherence theorem for monoidal categories~\cite[Theorem 1.2.23]{johnson2021}.
Indeed, the proof is very similar: higher stucture morphisms are simply moved from the actegory to the equivalence.

\subsubsection{Lemma}
\label{lemma:strictification}
\begin{reslemma}{strictification}
	Let $\cat{M}$ be strict monoidal, and suppose $(\cat C, \action)$ is a left $\cat M$-actegory.
	Then there exists a strict $\cat M$-actegory $(\cat C_\strict, \action_\strict)$ and an equivalence $S: \cat{C}_\strict \to \cat{C}$ in $\Act{\cat M}$.
\end{reslemma}
\begin{proof}
	A proof of this fact can be obtained by appealing to the general strictification result for algebras of pseudomonads with rank,\footnote{i.e.~preserving $\kappa$-small colimits for some cardinal $\kappa$.} namely \cite[Theorem 3.13]{blackwell1989two}.
	$\cat M \times -$ has rank since it preserves all colimits by virtue of being left adjoint.

	Concretely, the objects of $\cat{C}_\strict$ are pairs $(m: \cat{M}, c : \cat{C})$ and its morphisms $f: (m,c) \to (n, d)$ are given by arrows  $m \action c \to n \action d$ in $\cat C$. Then $S$ is identity on morphisms and on objects sends $(m, c)$ to $m \action c$.
\end{proof}

\begin{remark}
	In \cite[§4]{blackwell1989two}, $(\cat C_\strict, \action_\strict)$ is denoted by $(\cat C, \action)'$. It is also noted that strictification defines a functor left adjoint to the inclusion $\Act{\cat M^\strict}_\strict \to \Act{\cat M^\strict}$.
\end{remark}

As a consequence of this theorem, we may in general assume that an actegory is strict, similar to the way we routinely assume that monoidal categories are strict, although one has to be careful about which statements can actually be transported along actegory equivalences, especially if we are considering actions of multiple different monoidal categories.
A more precise statement is that \emph{any diagram which commutes under the assumption of strictness, automatically commutes}.

\begin{corollary}
\label{cor:strictification-equiv}
	There is an equivalence of 2-categories:
	\begin{equation}
		\Act{\cat M}^\lax \equi \Alg{(\cat M^\strict \times -)}_\strict^\lax.
	\end{equation}
	where on the left we have the 2-category of $\cat M$-actegories, lax linear functors and linear transformations and on the right we have the 2-category of strict $\cat M$-actegories, lax linear functors and linear transformations.
\end{corollary}
\begin{proof}
	\begin{equation}
		\Act{\cat M}^\lax = \Alg{(\cat M \times -)}_\pseudo^\lax \overset{\text{MacLane}}\equi \Alg{(\cat M^\strict \times -)}_\pseudo^\lax \overset{\text{Lemma~\ref{lemma:strictification}}}\equi \Alg{(\cat M^\strict \times -)}_\strict^\lax.
	\end{equation}
\end{proof}

\begin{remark}
	Note that Lemma~\ref{lemma:strictification} allows us to strictify \emph{objects} but not \emph{morphisms}: indeed, in the statements of the equivalence of Corollary~\ref{cor:strictification-equiv}, we still have non-strict morphisms on both sides!
	Anyway, we observe Corollary~\ref{cor:strictification-equiv} still holds if we replace `lax linear' morphisms with `strong linear' morphisms (i.e.~$\lax$ with $\pseudo$ on both sides of the equivalence).
\end{remark}

Moreover, we have the following:

\begin{lemma}
\label{lemma:strictification-functor}
	Let $(\cat M \times \cat C, \action)$ be a free $\cat M$-actegory, for $\cat M$ strict monoidal, and let $(\cat D, \actionn)$ be any other strict $\cat M$-actegory.
	Then given a strong $\cat M$-linear functor $(F, \ell) : (\cat M \times \cat C, \action) \to (\cat D, \actionn)$, there is a unique strict linear functor $(\tilde F, =) : (\cat M \times \cat C, \action) \to (\cat D, \actionn)$ linearly and naturally isomorphic to $F$.
\end{lemma}
\begin{proof}
	For existence, define
	\begin{equation}
		\tilde F(m, c) = m \actionn F(j, c).
	\end{equation}
	This functor is strictly linear:
	\begin{equation}
		m \actionn \tilde F(n,c) = m \actionn (n \actionn F(j,c)) = (m \mtimes n) \actionn F(j,c) = \tilde F(m \mtimes n, j) = \tilde F(m \action (n, c)).
	\end{equation}
	The natural isomorphism $\gamma_{m,c} : \tilde F(m,c) \twoto F(m, c)$ is given by $\ell$ itself:
	\begin{equation}
		\gamma_{m,c} := \tilde F(m,c) \equalto m \actionn F(j,c) \nlongto{\ell_{m,(j,c)}} F(m \mtimes j, c) \equalto F(m, c).
	\end{equation}
	To see it is linear we draw the following:
	\begin{diagram}[sep=3ex]
		{m \actionn (n \actionn F(j,c))} \\
		{m \actionn \tilde F(n,c)} & {m \actionn F(n,c)} \\
		{\tilde F(m \mtimes n, c)} & {F(m \mtimes n, c)} \\
		{(m \mtimes n) \actionn F(j,c)}
		\arrow[Rightarrow, no head, from=2-1, to=3-1]
		\arrow["{m \actionn \gamma_{n,c}}", from=2-1, to=2-2]
		\arrow["{\ell_{m, (n,c)}}", from=2-2, to=3-2]
		\arrow["{\gamma_{m \mtimes n, c}}", from=3-1, to=3-2]
		\arrow["{m \actionn \ell_{n,c}}", from=1-1, to=2-2]
		\arrow[Rightarrow, no head, from=1-1, to=2-1]
		\arrow["{\ell_{m \mtimes n,(j,c)}}"', from=4-1, to=3-2]
		\arrow[Rightarrow, no head, from=3-1, to=4-1]
	\end{diagram}
	The external perimeter corresponds to Equation~\eqref{eq:lineator-inductive-def} for the strict case.

	For uniqueness, notice the above diagram (drawn for $m=j$) forces $\gamma$ to track $\ell$.
\end{proof}

If the reader is interested in a deeper understanding of strictification, in \cite[§3]{blackwell1989two} the equivalence of strict and pseudoalgebras is analyzed in detail, and more refined statements of the strictification lemma are provided (e.g.~ regarding exactly how it looks on morphisms and transformations of pseudoalgebras).


\subsection{The free and cofree adjunctions}
If the reader is acquainted with the relationships between adjunctions and monads, the following is unsurprising:

\begin{proposition}
\label{prop:free-forgetful-actegory}
	There is a 2-adjunction
	\begin{diagram}[ampersand replacement={\&}]
		{\Act{\cat M}^\lax} \&\& \Cat
		\arrow[""{name=0, anchor=center, inner sep=0}, "{\cat M[-]}"', shift right=2, from=1-3, to=1-1]
		\arrow[""{name=1, anchor=center, inner sep=0}, "U_{\cat M}"', shift right=2, from=1-1, to=1-3]
		\arrow["\dashv"{anchor=center, rotate=-90}, draw=none, from=0, to=1]
	\end{diagram}
	where $U$ forgets the actegorical structure and $\cat M[-]$ is the free $\cat M$-actegory (Example~\ref{ex:free-act}) functor.
\end{proposition}
\begin{proof}
	By Corollary~\ref{cor:strictification-equiv}, $\Act{\cat M}^\lax$ is equivalent to $\Act{(\cat M^\strict \times -)}^\lax_\strict$.
	Moreover, \cite[Theorem~3.13]{blackwell1989two} gives us a 2-adjunction $\Act{(\cat M^\strict \times -)}^\lax_\strict \leftrightarrows \Act{(\cat M^\strict \times -)}_\strict$.
	It is thus sufficient to prove we have a 2-adjunction:
	\begin{equation}
		\begin{tikzcd}[ampersand replacement=\&]
			{\Act{(\cat M^\strict \times -)}_\strict} \&\& \Cat
			\arrow[""{name=0, anchor=center, inner sep=0}, "{\cat M^\strict[-]}"', shift right=2, from=1-3, to=1-1]
			\arrow[""{name=1, anchor=center, inner sep=0}, "{U_{\cat M^\strict}}"', shift right=2, from=1-1, to=1-3]
			\arrow["\dashv"{anchor=center, rotate=-90}, draw=none, from=0, to=1]
		\end{tikzcd}
	\end{equation}
	To do so, we appeal to~\cite[Proposition~3.1]{blackwell1989two} which says that if $\Act{(\cat M^\strict \times -)}_\strict$ admits arrow objects (i.e.~powers by the walking arrow category $\downarrow$) and $U_{\cat M^\strict}$ preserves them, then $U_{\cat M^\strict}$ admits a left 2-adjoint iff the underlying 1-functor admits a left 1-adjoint.
	First, let us show arrow objects in $\Act{(\cat M^\strict \times -)}_\strict$ exists. For an $\cat M$-actegory $(\cat C, \action)$, $(\cat C, \action)^\downarrow$ is the $\cat M$-action on $\cat C^\downarrow$ given
	\begin{equation}
		\begin{tikzcd}[ampersand replacement=\&]
			{\cat M \times \cat C^\downarrow} \& {\cat M^\downarrow \times \cat C^\downarrow} \& {\cat C^\downarrow.}
			\arrow["{\action^\downarrow}", from=1-2, to=1-3]
			\arrow["{\id \times \cat C^\downarrow}", from=1-1, to=1-2]
		\end{tikzcd}
	\end{equation}
	It's easy to verify that the projections $\cat C \xleftarrow{\rm dom} \cat C^\downarrow \xrightarrow{\rm cod} \cat C$ are $\cat M$-linear, as well as the natural transformation $\rm arr : dom \twoto cod$.
	Moreover, we now prove:
	\begin{equation}
		\Act{(\cat M^\strict \times -)}_\strict((\cat D, \actionn), (\cat C, \action))^\downarrow \iso \Act{(\cat M^\strict \times -)}_\strict((\cat D, \actionn), (\cat C, \action)^\downarrow).
	\end{equation}
	Indeed, the first category has linear transformations between lax linear functors $\cat D \to \cat C$ as objects.
	At the level of the underlying categories, functors $\cat D \to \cat C^\downarrow$ correspond to natural transformations.
	A lax linear structure of such a functor, call it $F$, induces analogous structures on $F^0 := F \comp {\rm dom}$ and $F^1 := F \comp {\rm cod}$, by composition.
	Moreover, the natural transformation $F^{\rm arr} : F^0 \twoto F^1$, whose components are picked by the action of $F$ on objects, is linear.
	In fact Diagram~\eqref{diag:lin-transf-coherence} is instantiated to
	\begin{equation}
		\begin{tikzcd}[ampersand replacement=\&]
			{m \action F^0(d)} \& {m \action F^1(d)} \\
			{F^0(m \actionn d)} \& {F_1(m \actionn d)}
			\arrow["{\ell^0_{m,d}}"', from=1-1, to=2-1]
			\arrow["{\ell^1_{m,d}}", from=1-2, to=2-2]
			\arrow["{m \action F^{\rm arr}_d}", from=1-1, to=1-2]
			\arrow["{F^{\rm arr}_{m \actionn d}}"', from=2-1, to=2-2]
		\end{tikzcd}
	\end{equation}
	whose commutativity is part of the well-definition of $F$ as a lax linear functor into $\cat C^\downarrow$.

	Finally, clearly $U_{\cat M^\strict}$ preserves arrow objects, hence the aforementioned result applies.
	Then we have (this is now at the level of 1-categories):
	\begin{equation}
		\Act{(\cat M^\strict \times -)}_\strict(\cat M \times \cat C, (\cat D, \actionn)) \cong \Cat(\cat C, \cat D)
	\end{equation}
	which is immediate to verify.

	To get the desired adjunction, we pre-compose the 2-adjunction we obtained with the adjoint equivalence equivalent to the strictification equivalence (see again~\cite[Remark~4.1]{blackwell1989two}).
\end{proof}

In lower dimension, the `free $M$-module' over a given \emph{set} $I$ is well-known to be given by direct sum of $I$-many copies of $M$. An analogous fact is true for actegories:

\begin{proposition}
\label{prop:free-act-is-copow}
	The free $\cat M$-actegory construction is a $\Cat$-copower in $\Act{\cat M}$:
	\begin{equation}
		\cat M[\cat C] \equi \cat C \copow \cat M.
	\end{equation}
\end{proposition}
\begin{proof}
	The $\cat C$-copower of $\cat M$ in $\Act{\cat M}$ is defined by the $\cat C$-weighted colimit of the one-object diagram at $\cat M$ \cite[§3.7]{kelly1982basic}, thus we have
	\begin{equation}
		\Act{\cat M}({\colim}^{\cat C} \cat M, \cat X) \iso \Cat(\cat C, \Act{\cat M}(\cat M, \cat X)) \overset{\text{Lemma}~\ref{lemma:strictification-functor}}\iso \Cat(\cat C, U_{\cat M}(\cat X)).
	\end{equation}
	This shows that $- \copow \cat M \adj U_{\cat M}$, implying $\cat C \copow \cat M \equi \cat M[\cat C]$.
\end{proof}

We also note there exists a \emph{cofree construction} for actegories, given by $[\cat M, -]$:
\begin{eqalign}
	\ast : \cat M \times [\cat M, \cat C] &\longto [\cat M, \cat C]\\
	(m,\, F)\ &\longmapsto\ F(m \mtimes -).
\end{eqalign}

Indeed, we have the following:

\begin{proposition}
\label{prop:comonadic}
	There is a comonadic 2-adjunction:
	\begin{diagram}[ampersand replacement={\&}]
	{\Act{\cat M}^\lax} \&\& \Cat
		\arrow[""{name=0, anchor=center, inner sep=0}, "{[\cat M, -]}", shift left=2, from=1-3, to=1-1]
		\arrow[""{name=1, anchor=center, inner sep=0}, "{U_{\cat M}}", shift left=2, from=1-1, to=1-3]
		\arrow["\dashv"{anchor=center, rotate=-90}, draw=none, from=1, to=0]
	\end{diagram}
	exhibiting $\Act{\cat M}^\lax$ as the category of pseudocoalgebras of the pseudocomonad $[\cat M, -]$, whose structure is given by:
	\begin{eqalign}
		k_{\cat C} : [\cat M, \cat C] &\longto \cat C\\
		F\ &\longmapsto\ F(j)\\[2ex]
		w_{\cat C} : [\cat M, \cat C] &\longto [\cat M, [\cat M, \cat C]]\\
		F\ &\longmapsto\ (m \mapsto F(m \mtimes -)),
	\end{eqalign}
	plus the obvious left/right counitors and coassociator.
\end{proposition}
\begin{proof}
	This is a consequence of the tensor-hom adjunction in $\Cat$:
	\begin{equation}
			\{\, \action : \cat M \times \cat C \longto \cat C\,\}
		\quad
		\iso
		\quad
			\{\, \action : \cat C \longto [\cat M, \cat C]\,\}
	\end{equation}
	One easily verifies the rest of a pseudoalgebra structure for $\cat M \times -$ is carried along to a pseudocoalgebra structure of $[\cat M, -]$, and the same happens to morphisms and 2-cells.
\end{proof}

Hence actegories also admit a presentation as coalgebras. The usefulness of such presentation can be seen, for instance, in Garner's result we reported in Example~\ref{ex:tangent}. In fact that is saying that tangent categories are coalgebras of the pseudocomonad $\ncat{Tang}(\cat W, -)$, i.e.~the subcomonad of $[\cat W, -]$ obtained by restricting to functors preserving tangent limits.

\begin{remark}
	The coalgebraic presentation of actegories has a relative one dimension down.
	In fact coalgebras of $M \times -$, for $M=\Sigma^*$ a free monoid on a given (usually finite) alphabet $\Sigma=\{a,b,\ldots\}$ are `half of an automaton', lacking a choice of output map.
	Nevertheless, this `dynamical' interpretation of actegories seems worthy to explore.
	In particular we might ask: what is a 1-automaton, i.e. an automaton whose underlying states form a category instead of a mere set, and likewise does its alphabet?
\end{remark}

Finally, knowing $\Act{\cat M}^\lax$ is monadic and comonadic has various technical advantages, among which:

\begin{corollary}
	$\Act{\cat M}^\lax$ admits all pseudo and lax limits and colimits.
\end{corollary}
\begin{proof}
	\cite[§2]{blackwell1989two} proves that categories of pseudoalgebras admit all pseudo and lax limits. By duality we obtain the rest.
\end{proof}

\subsection{The indexed 2-category \texorpdfstring{$\Actt$}{Act}}
So far we have been focusing on $\cat M$-actegories, for a single monoidal base $\cat M$.
In the below proposition we show that actegories admit a canonical notion of `change of base', or, as it is known in algebra, `restriction of scalars'.

\begin{proposition}
	\label{prop:actegories-change-of-base}
	Let $(\cat M, j, \mtimes)$ and $(\cat N, i, \ntimes)$ be monoidal categories.
	Then every strong monoidal functor $R : \cat M \to \cat N$ induces a 2-functor $R^* : \Act{\cat N}^\lax \to \Act{\cat M}^\lax$ given by \emph{\textbf{restriction of scalars}}:
	\begin{equation}
		R^*(\cat C, \action) := (\cat C,\ \cat M \times \cat C \nlongto{R \times \cat C} \cat N \times \cat C \nlongto{\action} \cat C).
	\end{equation}
\end{proposition}
\begin{proof}
	First of all, let's complete the definition on objects. Denote by $(\epsilon, \varpi)$ the monoidal structure of $R$. Then the unitor of $R^*(\cat C, \action)$ is given by
	\begin{equation}
		\eta^{R^*\action}_c : c \nlongto{\eta^\action_c} i \action c \nlongto{\epsilon \action c} R(j) \action c
	\end{equation}
	while the multiplicator is defined as
	\begin{equation}
		\mu^{R^*\action}_{m,m',c} : R(m) \action (R(m') \action c) \nlongto{\mu^\action_{R(m), R(m'), c}} (R(m) \ntimes R(m')) \action c \nlongto{\varpi_{m,m'} \action c} R(m \mtimes m') \action c.
	\end{equation}
  Now let $(\cat C, \action)$ and $(\cat D, \actionn)$ be $\cat N$-actegories and $(F, \ell)$ a lax $\cat N$-linear functor between them.
	We define the action of $R^*$ on this to be $R^*(F, \ell) := (F, R^*\ell)$ where $R^*\ell$ denotes the 2-cell:
	\begin{diagram}[sep=4.5ex]
		{\cat M\times \cat C} & {\cat M \times \cat D} \\
		{\cat N \times \cat C} & {\cat N \times \cat D} \\
		{\cat C} & {\cat D}
		\arrow["{R \times \cat C}"', from=1-1, to=2-1]
		\arrow["\action"', from=2-1, to=3-1]
		\arrow["F", from=3-1, to=3-2]
		\arrow["\actionn", from=2-2, to=3-2]
		\arrow["{\cat N \times F}", from=2-1, to=2-2]
		\arrow["\ell"', shorten <=12pt, shorten >=12pt, Rightarrow, from=2-2, to=3-1]
		\arrow["{R \times \cat D}", from=1-2, to=2-2]
		\arrow["{\cat M \times F}", from=1-1, to=1-2]
	\end{diagram}
	defined by whiskering along $R \times \cat C$, hence whose components are
	\begin{equation}
		R^*\ell_{n, c} : R(n) \actionn F(c) \isolongto F(R(n) \action c).
	\end{equation}

	Finally, let $\varphi : (F, \ell) \twoto (G, \nu)$ be an $\cat M$-linear natural transformation. Since $F$ and $G$ are mapped to themselves, $\varphi$ is still a well-defined natural transformation.
	Its $\cat N$-linearity can be concluded by observing that the squares which have to commute \eqref{diag:lin-transf-coherence} are a subset of those that commute by assumption of $\cat M$-linearity (since every action of a scalar $n:\cat N$ is mediated by the $\cat M$-action).

	Functoriality and 2-functoriality can be verified by routine.
\end{proof}

\begin{remark}
	In Proposition~\ref{prop:extension-of-scalars}, we prove each of these functors $R^*$ (when restricted to strong linear functors) is actually a right adjoint, with its left adjoint $R_!$ being known as \emph{extension of scalars}.
\end{remark}

\begin{example}
	In Example~\ref{ex:stefanou} we described an action of $\mathcal O(X)$ on $\Top/X$, for a given topological space $X$.
	Observe this action can be obtained from the canonical left self-action (Example~\ref{ex:moncat-self-actions}) of $(\Top/X, 1, \times_X)$ by restriction along the evident monoidal functor $\mathcal O(X) \to \Top/X$.
\end{example}

Furthermore, a monoidal natural transformation between monoidal functors induces a natural transformation between the corresponding restriction functors:

\begin{proposition}
\label{prop:actegories-change-of-base-nattrans}
	Given strong monoidal functors $R, S : \cat M \to \cat N$ and a monoidal natural transformation $\alpha : R \twoto S$, there is a strictly 2-natural transformation (notice the change in direction)
	\begin{equation}
		\alpha^* : S^* \longtwoto R^*
	\end{equation}
	between the corresponding restrictions $S^*, R^* : \Act{\cat N}^\lax \to \Act{\cat M}^\lax$.
\end{proposition}
\begin{proof}
	For each actegory $(\cat D, \actionn)$ in $\Act{\cat N}^\lax$ the natural transformation $\alpha^*$ picks out a lax $\cat M$-linear morphism $(1_{\cat D}, \ell^\alpha_{(\cat D, \actionn)}) : S^*(\cat D, \actionn) \longto R^*(\cat D, \actionn)$ defined as
	\begin{equation}
		\ell^\alpha_{(\cat D, \actionn)} := \alpha_{m} \actionn d : R(m) \actionn d \to S(m) \actionn d.
	\end{equation}
	It is evident $\ell^\alpha_{(\cat D, \actionn)}$ is strictly natural because is the whiskering of strictly natural transformations, and monoidality of $\alpha$ makes $\ell^\alpha_{(\cat D, \actionn)}$ a well-defined lineator (i.e.~its coherence as monoidal transformation maps to coherence for $\ell^\alpha_{(\cat D, \actionn)})$.
	The fact $\ell^\alpha$ is natural amounts to check that for a given lax $\cat N$-linear functor $(F, \nu) : (\cat D, \actionn) \to (\cat C, \action)$, the following commutes:
	\begin{diagram}
		{R(m) \action F(d)} & {S(m)\action F(d)} \\
		{F(R(m)\actionn d)} & {F(S(m) \actionn d)}
		\arrow["{\nu_{R(m),d}}"', from=1-1, to=2-1]
		\arrow["{\nu_{S(m),d}}", from=1-2, to=2-2]
		\arrow["{F(\alpha_m \actionn d)}"', from=2-1, to=2-2]
		\arrow["{\alpha_m \action F(d)}", from=1-1, to=1-2]
	\end{diagram}
	But this is a naturality square for $\ell$, hence commutes by assumption.
\end{proof}

\begin{proposition}
\label{prop:actegories-indexed-over-moncat}
	The assignment $\cat M \mapsto \Act{\cat M}^\lax$ extends to an indexed 2-category
	\begin{equation}
		\Act{(-)}^\lax : \MonCat^\coop \to 2\Cat;
	\end{equation}
	whose action on a strong monoidal functor $R:\cat M \to \cat N$ is defined in Proposition~\ref{prop:actegories-change-of-base} and whose action on a monoidal natural transformation $\alpha : R \twoto S$ is defined in Proposition~\ref{prop:actegories-change-of-base-nattrans}.
\end{proposition}

By effecting a bicategorical Grothendieck construction \cite[Definition 10.7.2]{johnson2021}, we obtain a 2-category of actegories over an arbitrary base that we denote by $\Actt$.

The 2-category $\Actt$ has pairs $(\cat M : \MonCat, (\cat C, \action) : \Act{\cat M}^\lax)$ as objects.
A morphism from an $\cat M$-actegory $(\cat C, \action)$ to an $\cat N$-actegory $(\cat D, \actionn)$ consists of a strong monoidal functor $R : \cat M \to \cat N$, and a lax $\cat M$-linear functor between between actegories $(\cat C, \action)$ and $(\cat D, (R \times \cat D) \comp \actionn)$.
This lax linear functor itself consists of a functor $R^\sharp : \cat C \to \cat D$ and a lineator $\ell_{m, c} : R(m) \actionn R^{\sharp}(c) \isolongto R^{\sharp}(m \action c)$.
All in all, the data of a morphism in $\Actt$ consist of a triple $(R,
R^{\sharp}, \ell)$ that can be neatly arranged in a lax commutative square:

\begin{diagram}[row sep= 4ex, column sep=4ex]
    \label{eq:act_comm_square}
    {\cat M \times \cat C} && {N \times \cat D} \\
    \\
    {\cat C} && {\cat D}
    \arrow["\action"', from=1-1, to=3-1]
    \arrow["{R^{\sharp}}"', from=3-1, to=3-3]
    \arrow["\actionn", from=1-3, to=3-3]
    \arrow["{R \times R^{\sharp}}", from=1-1, to=1-3]
    \arrow["\ell"', shorten <=18pt, shorten >=18pt, Rightarrow, from=1-3, to=3-1]
\end{diagram}

A 2-cell in $\Actt$ between $(R, R^{\sharp}, \ell)$ and $(S, S^{\sharp}, \nu)$ consists of a pair $(\alpha, \alpha^{\sharp})$, where $\alpha : R \twoto S$ is a monoidal natural transformation and $\alpha^{\sharp}$ is an $\cat M$-linear transformation filling the following lax triangle in $\Act{\cat M}^\lax$:

\begin{diagram}[row sep=2ex, column sep=2ex]
	{} &&& {R^*(\cat D, \actionn)} \\
	\\
	{(\cat C, \action)} \\
	\\
	& {} && {S^*(\cat D, \actionn)} & {}
	\arrow["{(S^{\sharp}, \nu)}"', from=3-1, to=5-4]
	\arrow["{(R^{\sharp}, \ell)}", from=3-1, to=1-4]
	\arrow["{\alpha^*_{(\cat D, \actionn)}}"', from=5-4, to=1-4]
	\arrow["{\alpha^{\sharp}}"', shorten <=58pt, shorten >=58pt, Rightarrow, from=1-1, to=5-5]
\end{diagram}

\begin{remark}
\label{rmk:cartesian-factorization}
	There is a cartesian factorization system on $\Actt$ deriving from the indexing it was born from. Its left maps are vertical maps, i.e.~lax linear functors as described in Definition~\ref{def:linear-functor}, living inside a given fibre. Right maps are cartesian maps, which turn out to be maps between different actegorical structures on a fixed category $\cat C$.
	\begin{equation}
		(R, 1_{R^*(\cat C, \action)}, {=}) : (\cat M, (\cat C, \action)) \longto (\cat N, (\cat C, \actionn)).
	\end{equation}
	Additionally, 2-cells between such morphisms are 2-cells $(\alpha, {=})$ in $\Actt$, i.e.~2-cells whose vertical part is trivial.
	We already have a name for the `vertical subcategories' of $\Actt$, namely its fibres $\Act{\cat M}^\lax$ as $\cat M$ varies.
	Correspondingly, its cartesian subcategories (the `transverse fibres') will be denoted by $\Actt^\cart(\cat C)$, as $\cat C$ varies, not to be confused with $\Act{\cat C}$, which would denote actions \emph{of} $\cat C$.
\end{remark}

\begin{example}
\label{ex:curried-action-terminal}
	In Proposition~\ref{prop:curried-action}, we have seen how every action can be curried in order to get a strong monoidal functor into $[\cat C, \cat C]$.
	We can now rephrase this observation as the fact that the evaluation action
	\begin{eqalign}
		\eval : [\cat C, \cat C] \times \cat C &\longto \cat C\\
		(F, c) &\longmapsto F(c)
	\end{eqalign}
	is (pseudo)terminal in $\Actt^\cart(\cat C)$, meaning every other action $\action$ of a monoidal category $\cat M$ factors uniquely through it:
	\begin{diagram}
		{[\cat C,\cat C] \times \cat C} & {\cat C} \\
		{\cat M \times \cat C}
		\arrow["\action"', from=2-1, to=1-2]
		\arrow["\eval", from=1-1, to=1-2]
		\arrow["{\exists!\, \curr(\action) \times \cat C}", dashed, from=2-1, to=1-1]
	\end{diagram}
\end{example}

It is a useful, if trivial, fact that every monoidal category gives an example of actegory.
Whence, we record the following:

\begin{proposition}
\label{prop:moncat-embed-act}
	The canonical fibration $\pi : \Actt^\lax \to \MonCat$ has a section $\Upsilon: \MonCat \to \Actt^\lax$ that maps each monoidal category to left self-action.
\end{proposition}
\begin{proof}
	On objects, this has been defined in Example~\ref{ex:moncat-self-actions}.
	On morphisms, a strong monoidal functor $(R, \epsilon, \mu) : (\cat M, j, \mtimes) \to (\cat N, i, \ntimes)$ is sent to the triple $(R, R, \mu)$.
	The fact that the functors between the scalars and the underlying categories of the actegories are actually the same strong monoidal functor $R$ allows us to use the laxator $\mu$ as lineator. Indeed, $(R, \mu)$ forms a lax $\cat M$-linear morphism of actegories $(\cat M, \mtimes) \to R^*(\cat N, \ntimes)$.
	Lastly, a monoidal natural transformation $\alpha : R \twoto S$ is sent to the pair $(\alpha, \alpha)$, where the $\cat M$-linearity of the second component follows from the monoidality of $\alpha$.
\end{proof}


	\newpage
	\section{Composing actegories}
\label{sec:composition}
In this section we explore some ways actegories can be combined together to give rise to new ones.
Although all of the following can be obtained by general abstract considerations, we deem interesting to explain the way we got some of these definitions, and the way we think about them, coming from operations on parametric morphisms. This will yield the cartesian and cocartesian products (Proposition~\ref{prop:cartesian-product-act} and~\ref{prop:coproduct-act}) and an `hybrid' we call \emph{external choice} (Definition~\ref{def:extch-prod-act}) by analogy with the same operation on open games and servers \cite{capucci2021translating, videla2022lenses}.
Besides those, we are going to consider other ways to combine actegories, notably their \textbf{tensor product} (Definition~\ref{def:tensor-actegories}) and the corresponding \textbf{internal hom} construction (Proposition~\ref{prop:biact-internal-hom}), which are suggested by the algebra of the situation.
Throughout this section, let $(\cat M, j, \mtimes)$ and $(\cat N, i, \ntimes)$
be two monoidal categories.

\subsection{Actegories as parametric morphisms}
\label{subsec:actegories-as-para}
An action $\action : \cat M \times \cat C \to \cat C$ can be seen as an $\cat M$-parametric morphism $\cat C \to \cat C$, and as such it can be conceptualized as a process which consumes $\cat C$, produces $\cat C$, and whose execution is commanded by a choice of $\cat M$, which we picture as chosen by an agent `guiding' the execution of the process.\footnote{This `game semantics' of parametric morphisms is discussed in detail in \cite{capucci2021towards, paratalk}.}
Therefore, actions of monoidal categories can be effectively represented as morphisms in $\Para(\Cat)$, the bicategory of parametric functors in $\Cat$ (for a definition of $\Para$, see \cite[Definition 2]{capucci2021towards}).

Since $\Cat$ is rich in structure, $\Para(\Cat)$ becomes a monoidal category in two ways:

\begin{enumerate}
	\item By \textbf{parallel product}, that is, by putting parametric morphisms in
	parallel, thereby taking the product of all three boundaries:
	\begin{figure}[H]
		\label{fig:parallel}
		\centering
		\includegraphics[width=.3\textwidth]{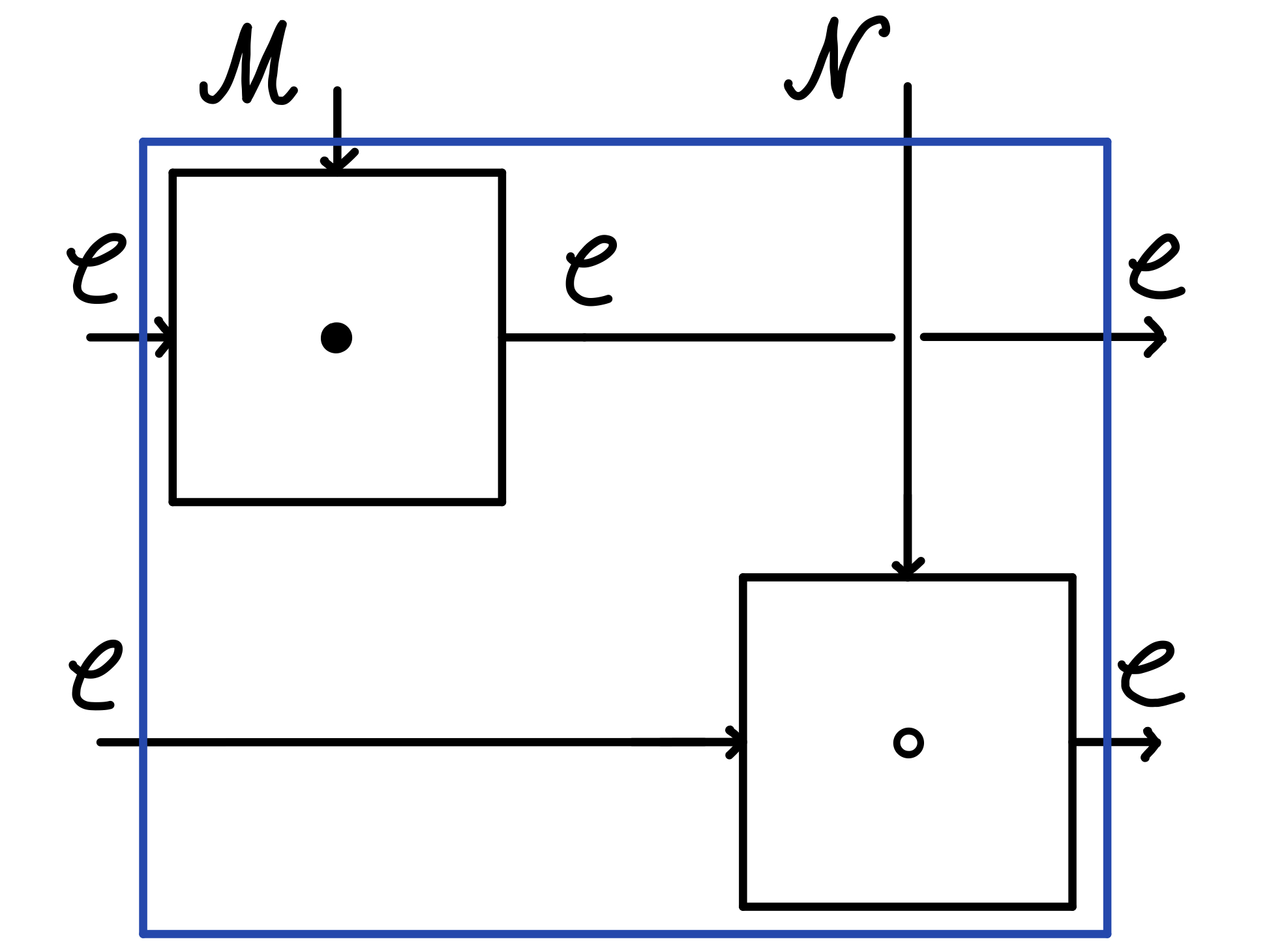}
	\end{figure}
	Concretely, we can do this thanks to the symmetry of $\times$:
	\begin{equation}
	\label{eq:parallel-product}
		(\cat M, \action) \otimes (\cat N, \actionn) :=
		\cat M \times \cat N \times \cat C \times \cat D \nlongto{\cat M \times \swap \times \cat D} \cat M \times \cat C \times \cat N \times \cat D \nlongto{\action \times \actionn} \cat C \times \cat D.
	\end{equation}
	\item By \textbf{external choice} (terminology borrowed from \cite{capucci2021translating}), hence by summing the boundaries except for the top ones which are multiplied:
    \begin{figure}[h]
		\label{fig:ext-ch}
		\centering
		\includegraphics[width=.5\textwidth]{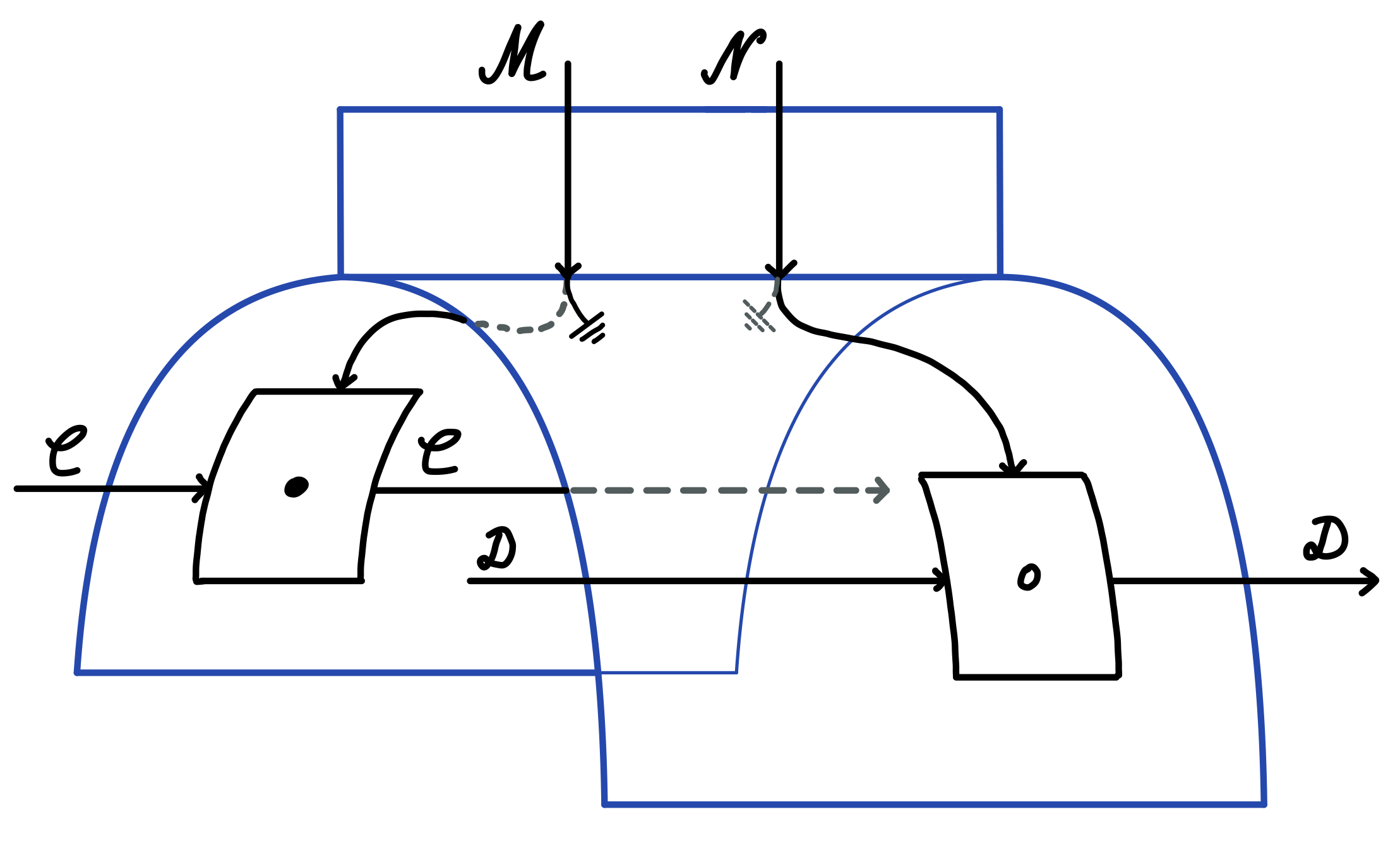}
	\end{figure}
	Concretely, we can do this thanks to the distributivity of $\times$ over $+$ in $\Cat$:
	\begin{eqalign}
	\label{eq:external-choice}
		(\cat M, \action) \extch (\cat N, \actionn) :=
		\cat M \times \cat N \times (\cat C + \cat D) &\longto \cat M \times \cat N \times \cat C + \cat M \times \cat N \times \cat D\\
		&\nlongto{\pi_{\cat M, \cat C} + \pi_{\cat N, \cat D}} \cat M \times \cat C + \cat N \times \cat D\\
		&\nlongto{\action \times \actionn} \cat C + \cat D.
	\end{eqalign}
\end{enumerate}

Moreover, $\Para(\Cat)$ is a bicategory, so we can also:

\begin{enumerate}[wide, labelwidth=!, labelindent=0pt]
	\addtocounter{enumi}{2}
	\item \textbf{Sequentially compose} two parametric morphisms:
	\begin{figure}[H]
		\label{fig:seq-comp}
		\centering
		\includegraphics[width=.3\textwidth]{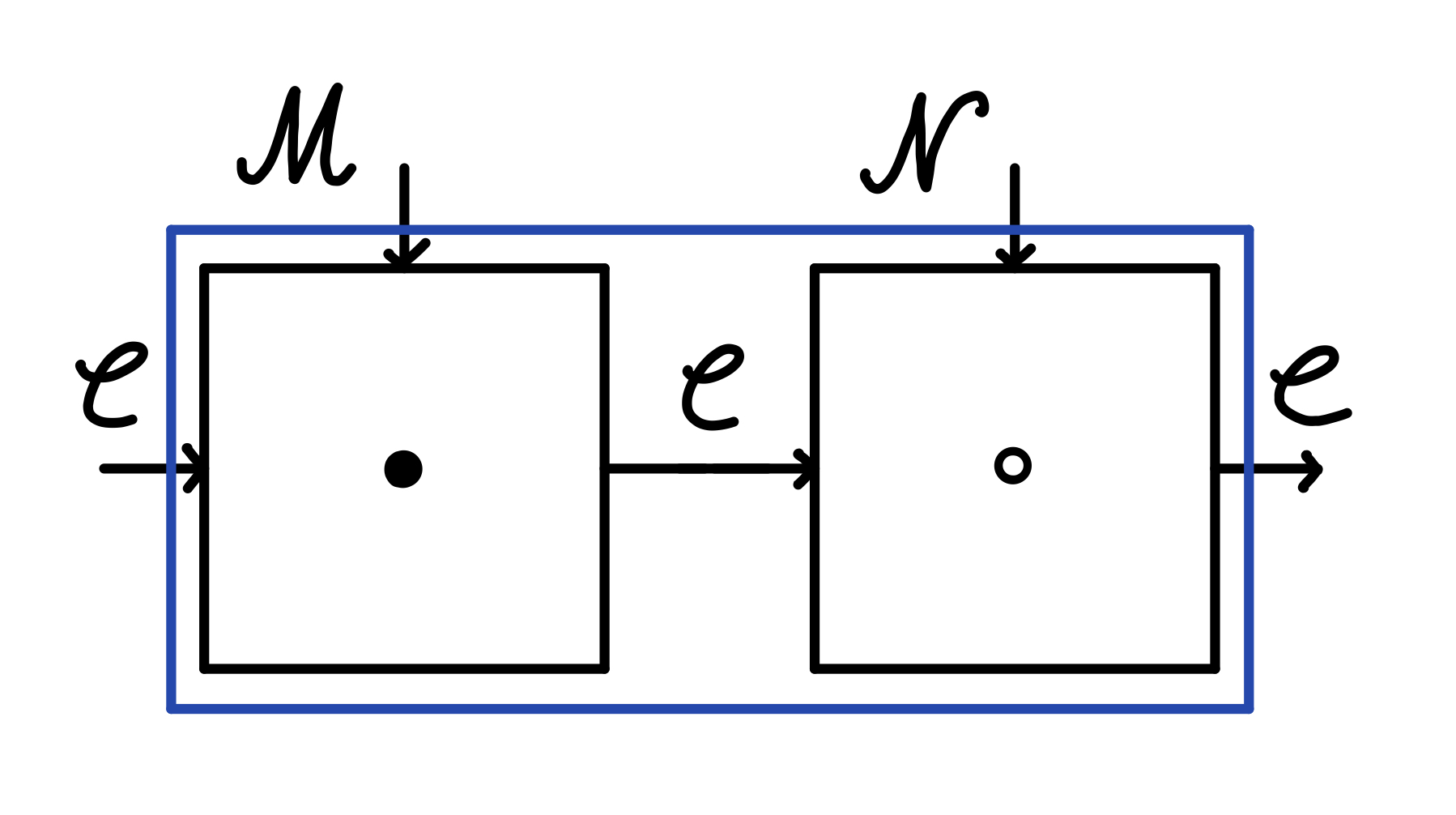}
	\end{figure}
	yielding the morphism
	\begin{equation}
	\label{eq:seq-composition}
		(\cat M, \action) \comp (\cat N, \actionn) :=
		\cat N \times \cat M \times \cat C \nlongto{\cat N \times \action} \cat N \times \cat C \nlongto{\actionn} \cat C.
	\end{equation}
	\item \textbf{Reparameterise} a given morphism:
	\begin{figure}[H]
		\label{fig:reparam}
		\centering
		\includegraphics[width=.15\textwidth]{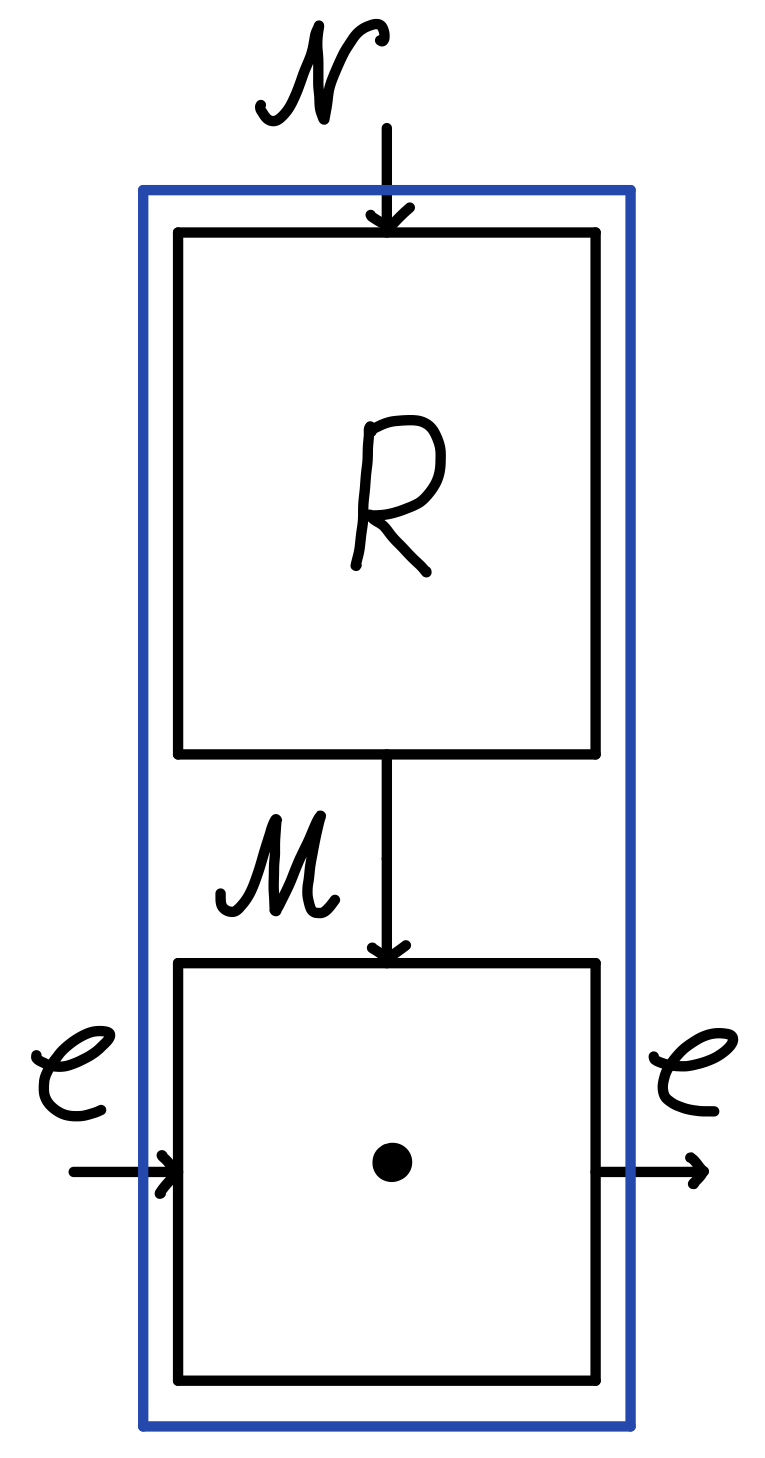}
	\end{figure}
	obtaining the morphism
	\begin{equation}
		R^*(\cat M, \action) :=
		\cat N \times \cat C \nlongto{R \times \cat C} \cat M \times \cat C \nlongto{\action} \cat C.
	\end{equation}
\end{enumerate}

When $R$ is a strong monoidal functor, this last operation can be recognized as being restriction of scalars along $R$, as defined in Proposition~\ref{prop:actegories-indexed-over-moncat} (even better, it represents a cartesian morphism in $\Actt$, as observed in Remark~\ref{rmk:cartesian-factorization}).
In particular, we know that doing so yields again an action. So actions are stable under reparameterisation along strong monoidal functors.

Can we say the same thing for operations (1)--(3)?
We anticipate the answer is \emph{yes} for parallel product and external choice.
For sequential composition, however, we'll see that to get another action we have to supply additional structure in the form of a \textbf{compatibility morphism}, i.e.~a kind of distributive law between the two actions.
This will be a wonderful excuse to introduce the two-sided analogue of actegories, \emph{biactegories}.
However, as we'll see in Section \ref{sec:monoidal-actegories}, this will not be
the only form of sequential composition.

\subsection{Product, coproduct and external choice of actegories}
In what follows, we are going to need a description of the product and
coproduct of monoidal categories (we have learned about the latter from \cite[§6.3]{roman2020profunctor}).
The cartesian product is easy: the product of $\cat M$ and $\cat N$ is supported by their product as categories and equipped with the expected componentwise monoidal structure: $(\cat M \times \cat N, (j,i), \mtimes \times \ntimes)$. Projections are given exactly as in $\Cat$.

However, their coproduct looks a bit more bizarre. Its construction can be carried out by analogy with that of (non-commutative!) monoids.
The objects of $\cat M + \cat N$ consists of finite words of objects from both $\cat M$ and $\cat N$, quotiented by the equivalences
\begin{eqalign}
	m_1m_2 \sim m_1 \mtimes m_2,& \quad j \sim e,\\
	n_1n_2 \sim n_1 \ntimes n_2,& \quad i \sim e,
\end{eqalign}%
where $e$ denotes the empty word.
An analogous construction is carried out on the morphisms.
The monoidal product on $\cat M + \cat N$ is given by juxtaposition of words, and the unit is given by $e$.
It is easy to realize that under the given equivalences, every object (and similarly for morphisms) in $\cat M + \cat N$ can be written uniquely as $m_1n_2 \cdots m_kn_k$ for $m_1, \ldots, m_k : \cat M$ and $n_1, \ldots, n_k : \cat N$.
The injection $\varepsilon_{\cat M} : \cat M \to \cat M + \cat N$ is simply defined as $m \longmapsto m$, and likewise for $\varepsilon_{\cat N}$.
Observe these are strictly monoidal because of the quotients defining $\cat M + \cat N$.

The coproduct of two monoidal categories falls short of being a biproduct because the words $mn$ and $nm$ are kept distinct.
Still, there is a universal morphism $\Gamma : \cat M + \cat N \to \cat M \times \cat N$ given by
\begin{equation}
\label{eq:gamma-morphism}
	\Gamma(m_1n_1 \cdots m_kn_k) := (m_1\mtimes \cdots \mtimes m_k,\ n_1 \ntimes \cdots \ntimes n_k)
\end{equation}
This functor is strict monoidal because of the quotient imposed on the objects of $\cat M + \cat N$.

Having defined the product and coproduct of monoidal categories, we can now
describe the product, coproduct, and external choice in $\Actt$.

\begin{proposition}[Cartesian product in $\Actt$]
\label{prop:cartesian-product-act}
	The \textbf{cartesian product} of an $\cat M$-actegory $(\cat C, \action, \eta^{\action}, \mu^{\action})$ and an $\cat N$-actegory $(\cat D, \actionn, \eta^{\actionn}, \mu^{\actionn})$ is an $\cat M \times \cat N$-actegory $(\cat C \times \cat D, \prodaction)$, where $\prodaction$ is the parallel product of $\action$ and $\actionn$ described in Equation~\eqref{eq:parallel-product}:
	\begin{equation}
		\prodaction := \cat M \times \cat N \times \cat C \times \cat D \nlongto{\cat M \times \swap \times \cat D} \cat M \times \cat C \times \cat N \times \cat D \nlongto{\action \times \actionn} \cat C \times \cat D.
	\end{equation}
\end{proposition}
\begin{proof}
	First of all, let's be explicit on the structure morphisms $\prodaction$ comes with: its unitor is given by the pair of unitors
	\begin{equation}
		\eta^{\prodaction}_{c, d} : (c, d) \nlongto{(\eta^{\action}_c, \eta^{\actionn}_d)} (j \action c, i \actionn d) = (j, i) \prodaction (c, d)
	\end{equation}
	and the multiplicator as a pair of underlying multiplicators
	\begin{eqalign}
		\mu^{\prodaction}_{(c, d)} &: (m, n) \prodaction ((m', n') \prodaction (c, d))\\
		&\equalto (m, n) \prodaction (m' \action c, n' \actionn d)\\
		&\equalto (m \action (m' \action c), n \actionn (n' \actionn d))\\
		&\nlongto{(\mu^{\action}_{c}, \mu^{\actionn}_{d})} ((m \mtimes m') \action c, (n \ntimes n') \actionn d)\\
		&\equalto ((m \mtimes m', n \ntimes n')) \prodaction (c, d)
	\end{eqalign}
	It is easy to convince oneself these make $\cat C \times \cat D$ a well-defined actegory.

	To say this is the cartesian product we need to supply projections, and these are built componentwise by the projections in $\MonCat$ and $\Cat$, and happen to be strictly linear.
	Finally, consider a  $\cat P$-actegory $(\cat E, \ast)$ and two morphisms $(F, F^\sharp, \ell) : (\cat P, (\cat E, \ast)) \to (\cat M, (\cat C, \action))$  and $(G, G^\sharp, \nu) : (\cat P, (\cat E, \ast)) \to (\cat N, (\cat D, \actionn))$. 
	Then there is a unique morphism $(\langle F, G\rangle, \langle F^\sharp, G^\sharp\rangle, \langle \ell, \nu \rangle) : (\cat P, (\cat E, \ast)) \to (\cat M \times \cat N, (\cat C \times \cat D, \prodaction))$ from the $\cat P$-actegory $\cat E$ to the product, given by universal pairing in each component.
	It is trivial to check $\langle \ell, \nu \rangle$ is indeed a well-defined lineator for $\langle F^\sharp, G^\sharp \rangle$, since coherence can be proven componentwise.
\end{proof}

We now proceed to describe the external choice in $\Actt$, which will be a
crucial ingredient in defining the coproduct in $\Actt$.

\begin{definition}[External choice product in $\Actt$]
\label{def:extch-prod-act}
	The \textbf{external choice product} of an $\cat M$-actegory $(\cat C, \action, \eta^{\action}, \mu^{\action})$ and an $\cat N$-actegory $(\cat D, \actionn, \eta^{\actionn}, \mu^{\actionn})$ is a $\cat M \times \cat N$-actegory $(\cat C + \cat D, \extchaction)$, where $\extchaction$ is defined as the external choice of $\action$ and $\actionn$ described in Equation~\eqref{eq:external-choice}:
	\begin{eqalign}
		\extchaction :=
		\cat M \times \cat N \times (\cat C + \cat D) &\longto \cat M \times \cat N \times \cat C + \cat M \times \cat N \times \cat D\\
		&\nlongto{\pi_{\cat M, \cat C} + \pi_{\cat N, \cat D}} \cat M \times \cat C + \cat N \times \cat D\\
		&\nlongto{\action \times \actionn} \cat C + \cat D.
		\end{eqalign}
	The unitor $\eta^{\extchaction}_{c} : c \isoto (j, i) \extchaction c$ and the multiplicator $\mu^{\extchaction}_c : (m, n) \extchaction ((m', n') \extchaction c) \isoto ((m \mtimes m', n \ntimes n')) \extchaction c$ are defined by cases to coincide with $\eta^\action$ and $\mu^\action$ on $\cat C$ and $\eta^\actionn$ and $\mu^\actionn$ on $\cat D$.
\end{definition}

Restricting scalars (or `reparametrising' if we want to use the language of
parametric morphisms) of the above defined external choice product along the natural morphism $\Gamma$ defined in Remark~\ref{eq:gamma-morphism}, we obtain the coproduct of actegories:

\begin{proposition}[Coproduct in $\Actt$]
\label{prop:coproduct-act}
	The \textbf{coproduct} of an $\cat M$-actegory $(\cat C, \action, \eta^{\action}, \mu^{\action})$ and an $\cat N$-actegory $(\cat D, \actionn, \eta^{\actionn}, \mu^{\actionn})$ is an $\cat M + \cat N$-actegory $\cat C + \cat D$ whose underlying functor $\coprodaction$ is defined as:
	\begin{equation}
		\coprodaction := \Gamma^*(\extchaction) = (\cat M + \cat N) \times (\cat C + \cat D) \nlongto{\Gamma \times (\cat C+ \cat D)} \cat M \times \cat N \times (\cat C + \cat D) \nlongto{\extchaction} \cat C + \cat D.
	\end{equation}
\end{proposition}
\begin{proof}
	Injections out of $(\cat M + \cat N, (\cat C+\cat D, \coprodaction))$ are built out of those of $\MonCat$ and $\Cat$, and happen to be strictly linear.
	Consider a  $\cat P$-actegory $(\cat E, \ast)$ and two morphisms $(F,F^\sharp, \ell) : (\cat M, (\cat C, \action)) \to (\cat P, (\cat E, \ast))$ and $(G, G^\sharp, \nu) : (\cat N, (\cat D, \actionn)) \to (\cat P, (\cat E, \ast))$. 
	Then there is a unique morphism $((F, G), (F^\sharp, G^\sharp), (\ell, \nu)) : (\cat M + \cat N, (\cat C + \cat D, \coprodaction)) \to (\cat P, (\cat E, \ast))$ from the coproduct to $\cat E$, given by universal copairing in each component.
	It is trivial to check $(\ell, \nu)$ is indeed a well-defined lineator for $(F^\sharp, G^\sharp)$, since coherence can be proven by cases.
\end{proof}

We can represent the coproduct of actegories with sheet diagrams of parametric
morphisms:
\begin{figure}[H]
	\label{fig:coproduct}
	\centering
	\includegraphics[width=.5\textwidth]{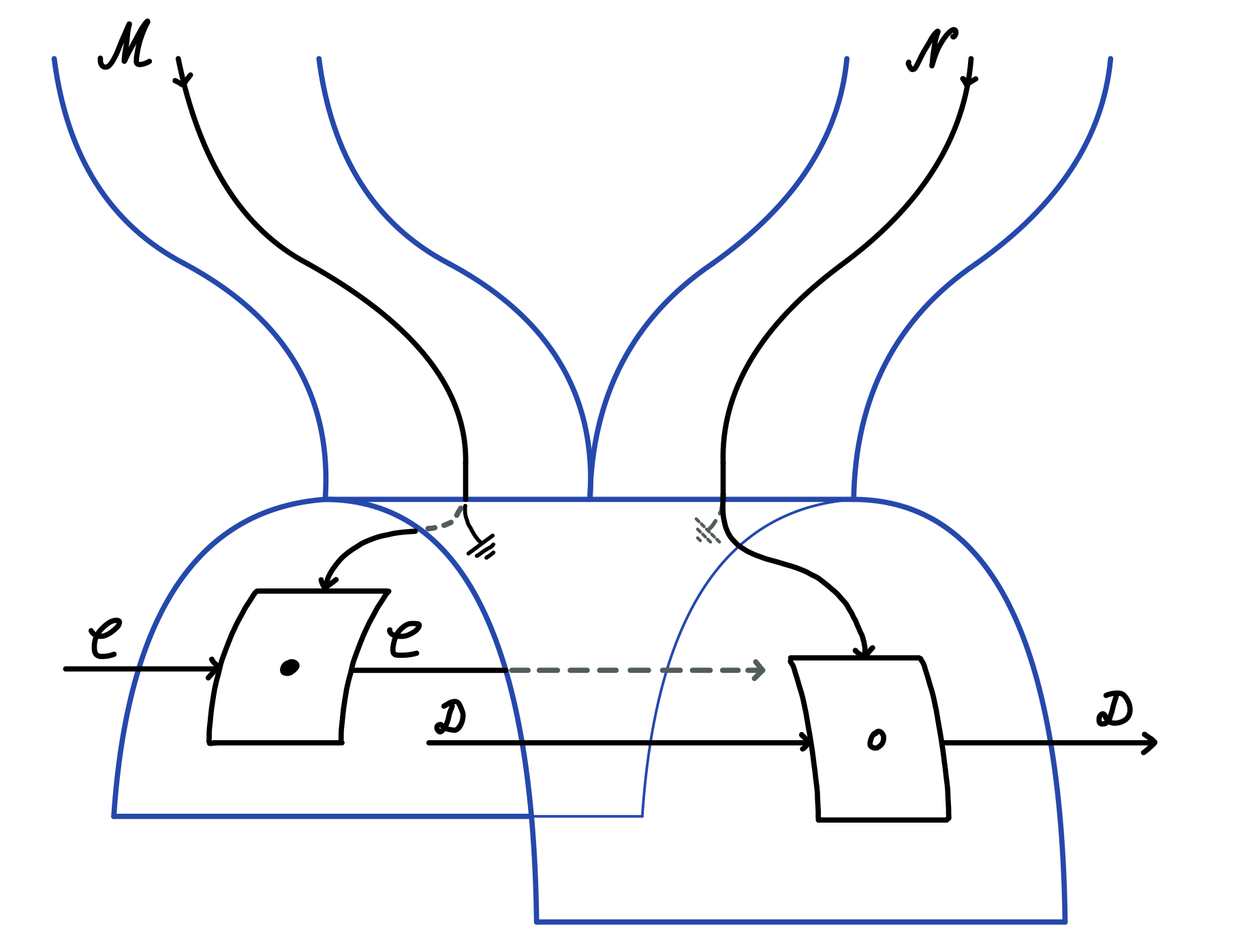}
\end{figure}

\begin{remark}
\label{rmk:act-fibrewise-cartesian}
	These two monoidal structures on $\Actt$ specialize to each of its fibres thanks to the (co)cartesian structure of $\MonCat$ \cite[Theorem 12.7]{shulman2008framed}.
	Given two $\cat M$\nobreakdash-actegories $(\cat C, \action)$ and $(\cat D, \actionn)$, we obtain another $\cat M$-actegory by first taking their product in $\Actt$ and then restricting along the copy functor ${\Delta_{\cat M} : \cat M \to \cat M \times \cat M}$ to turn the product ${\cat M \times \cat M}$-actegory into an $\cat M$-actegory.
	Concretely, we get an action supported by the functor:
	\begin{equation}
		m \prodaction_{\cat M} (c, d) := (m \action c, m \actionn d).
	\end{equation}
	The case for the coproduct is analogous.
	Notably, this restriction of the product of $\Actt$ to a fibre $\Act{\cat M}^\lax$ yields the product in the latter category. That is, $\Delta^*(\cat C \times \cat D, \action \times \actionn)$ is the product of $(\cat C, \action)$  and $(\cat D, \actionn)$ in $\Act{\cat M}^\lax$; and likewise applies to $+$.
\end{remark}

\begin{remark}
\label{rmk:prod-coprod-cartesian-subcats}
	The `transverse fibres' $\Actt^\cart(\cat C)$ (Remark~\ref{rmk:cartesian-factorization}) also have products and coproducts.
	These are easily described by deploying the fact illustrated in Example~\ref{ex:curried-action-terminal} and Proposition~\ref{prop:curried-action} that actions in $\cat C$ are exactly given by monoidal functors in $[\cat C, \cat C]$.
	The coproduct action is thus induced by universal copairing (as shown in \cite[§6.3]{roman2020profunctor}):
	\begin{diagram}
		{\cat M} & {\cat M+\cat N} & {\cat N} \\
		& {[\cat C, \cat C]}
		\arrow["\curr(\action)"', curve={height=12pt}, from=1-1, to=2-2]
		\arrow["\curr(\actionn)", curve={height=-12pt}, from=1-3, to=2-2]
		\arrow["{\varepsilon_{\cat M}}", from=1-1, to=1-2]
		\arrow["{\varepsilon_{\cat N}}"', from=1-3, to=1-2]
		\arrow["{\exists! \curr(\coprodaction_\cart)}"{description}, dashed, from=1-2, to=2-2]
	\end{diagram}
	Coproducts in $\Actt^\cart(\cat C)$ have particular relevance in the theory of optics since they provide the mathematical basis for hybrid composition of optics, that is, for composing different classes of optics (the original motivation behind the profunctor encoding \cite{pickering2017profunctor}).
	Indeed, both optics for $\action$ and for $\actionn$ can be embedded in the category of optics for $\coprodaction_\cart$, and by universal property of $+$, the latter is the smallest category of optics with this property. For instance, lenses (stemming from the self-action of a cartesian monoidal category) and prisms (stemming from the self-action of a cocartesian monoidal category) `compose' to affine traversals.
	A much more detailed discussion of this can be found in \textit{ibid.}

	We can also define a product action, by weak pullback:
	\begin{diagram}[row sep=3.5ex]
		& {\cat M \times_{[\cat C, \cat C]} \cat N} \\
		{\cat M} && {\cat N} \\
		& {[\cat C, \cat C]}
		\arrow["\curr(\action)"', curve={height=12pt}, from=2-1, to=3-2]
		\arrow["\curr(\actionn)", curve={height=-12pt}, from=2-3, to=3-2]
		\arrow["{p_{\cat M}}"', from=1-2, to=2-1]
		\arrow["{p_{\cat N}}", from=1-2, to=2-3]
		\arrow["{\exists! \curr(\prodaction_\cart)}"{description}, shift left=2, dashed, from=1-2, to=3-2]
		\arrow["\lrcorner"{anchor=center, pos=0.125, rotate=-45}, shift right=2, draw=none, from=1-2, to=3-2]
	\end{diagram}
	The monoidal category $\cat M \times_{[\cat C, \cat C]} \cat N$ contains only those pairs $(m,n)$ such that $m \action - \iso n \actionn -$. Hence optics for $\prodaction_\cart$ are the maximal category of optics which can be embedded both in optics for $\action$ and optics for $\actionn$.
	Using again lenses and prisms as an example, this category would be that of adapters, since the product of $\times$ and $+$ (as actions) is trivial (indeed, $a \times - \iso b + -$ iff $a=1$ and $b=0$).
\end{remark}

\subsection{Biactegories}
\label{sec:biactegories}
When we have a category $\cat C$ equipped with two actions, we might wonder whether these actions can be `sequentially composed'.
That is, starting with a $\cat M$-action $(\action, \eta^{\action}, \mu^{\action})$ and a $\cat N$-action $(\actionn, \eta^{\actionn}, \mu^{\actionn})$ on $\cat C$ we can form the functor
\begin{equation}
	\cat N \times \cat M \times \cat C \nlongto{\cat N \times \action} \cat N \times \cat C \nlongto{\actionn} \cat C
\end{equation}
and conjecture it is the underlying functor of an $\cat N \times \cat M$-actegory.

What kind of actegory is it?
To answer the question, it is better to approach the situation more formally, borrowing from the algebraic intuition on actegories.
We are contemplating a category receiving two actions, i.e.~carrying two pseudoalgebra structures, one for $\cat M \times -$ and one for $\cat N \times -$.
In other words, we can multiply objects (and morphisms) of $\cat C$ with scalars (and morphisms) from both $\cat M$ and $\cat N$.
It seems natural to ask for these two actions to be compatible with each other in some way.
This is even more tempting when the action of $\cat N$ is actually a right action, in which case it feels very natural to ask that
\begin{equation}
	m \action (c \actionn n) \iso (m \action c) \actionn n.
\end{equation}
This isomorphism is all that is needed to make $\cat C$ into an $\cat N \times \cat M$-actegory.
A pair of a left and a right action with a compatibility isomorphism as above is called a \emph{biactegory} \cite{skoda2006biactegories, skoda2009some}.

The `bi' in biactegory \emph{does not} refer to some additional dimension added to an actegory (as in \emph{bicategory}): if actegories are analogue to modules over monoids, biactegories are analogue to bimodules.

Even though the intuition of `sequential composition' of actegories as parametric morphisms invites to consider situations in which a category is equipped with two \emph{left} actions, following \cite{skoda2006biactegories} we define biactegories as equipped with a left and a right action.
Since left $\cat N$-actegories are equivalently right $\cat N^\rev$-actegories, this choice does not make any mathematical difference, and it is simply a matter of convenience: having a left and a right action makes it easier to work with the tensor product we are introducing later.

Coming back to our initial question, we are going to prove (Theorem~\ref{th:biact-are-bialg}) that algebras of $\cat N^\rev \times \cat M \times -$ (hence of its strictification) are equivalently given by one algebra structure for $\cat M \times -$, one for $\cat N^\rev \times -$, and the aforementioned `compatibility structure' that we call \emph{bimodulator}:
\begin{diagram}[row sep=5ex, column sep=1ex]
\label{eq:bimod-def}
	&[-7ex] {\cat M \times \cat N^\rev \times \cat C} &[1ex]&[1ex] {\cat N^\rev \times \cat M \times \cat C} &[-7ex]\\
	{\cat M \times \cat C} &&&& {\cat N^\rev \times \cat C} \\
	&& {\cat C}
	\arrow["\action"', from=2-1, to=3-3]
	\arrow["\actionn", from=2-5, to=3-3]
	\arrow["{\cat M \times \actionn}"', from=1-2, to=2-1]
	\arrow["\zeta", shift right=1, shorten <=36pt, shorten >=36pt, Rightarrow, from=2-1, to=2-5]
	\arrow["{\cat N^\rev \times \action}", from=1-4, to=2-5]
	\arrow["{\swap \times \cat C}"', from=1-4, to=1-2]
\end{diagram}
This morphism is a kind of `distributive law' between each action of a scalar in $\cat M$ and each action of a scalar in $\cat N$.
Indeed, we are going to come back to this idea in Section~\ref{sec:monoidal-actegories}.

\begin{definition}
\label{def:biactegories}
	An \textbf{$(\cat M, \cat N)$-biactegory} is a category $\cat C$ equipped with a left $\cat M$-action $(\action, \eta^\action, \mu^\action)$, a right $\cat N$-action $(\actionn, \eta^\actionn, \mu^\actionn)$ and a natural isomorphism called \emph{bimodulator}:
	\begin{equation}
		\zeta_{m,c,n} : m \action (c \actionn n) \isolongto (m \action c) \actionn n.
	\end{equation}
	The bimodulator needs to satisfy the following coherence axioms:
	\begin{diagram}[row sep=5ex, column sep=-3ex]
	\label{diag:bimod-coherence-pentag1}
		&[-7ex] {(m \mtimes n) \action (c \actionn p)} && {((m \mtimes n) \action c) \actionn p} &[-7ex]\\
		{m \action (n \action (c \actionn p)} &&&& {(m \action (n \action c)) \actionn p} \\
		&& {m \action ((n \action c) \actionn p)}
		\arrow["{\zeta_{m\mtimes n,c,p}}", from=1-2, to=1-4]
		\arrow["{\mu^\action_{m,n}}", from=2-1, to=1-2]
		\arrow["{m \action \zeta_{n,c,p}}"', from=2-1, to=3-3]
		\arrow["{\zeta_{m,n\action c,p}}"', from=3-3, to=2-5]
		\arrow["{\mu^\action_{m,n,c} \actionn p}"'{pos=0.3}, from=2-5, to=1-4]
	\end{diagram}
	\begin{diagram}[row sep=5ex, column sep=-3ex]
	\label{diag:bimod-coherence-pentag2}
		&[-6.5ex] {p \action (c \actionn (m \mtimes n))} && {(p \action c) \actionn (m \mtimes n)} &[-6.5ex]\\
		{p \action ((c \actionn m) \actionn n)} &&&& {((p \action c) \actionn m) \actionn n} \\
		&& {(p \action (c \actionn m)) \actionn n}
		\arrow["{\zeta_{p,c, m \mtimes n}}", from=1-2, to=1-4]
		\arrow["{\mu^\actionn_{p \action c, m, n}}"', from=2-5, to=1-4]
		\arrow["{\zeta_{p,c,m} \actionn n}"', from=3-3, to=2-5]
		\arrow["{\zeta_{p, c \actionn m, n}}"', from=2-1, to=3-3]
		\arrow["{p \action \mu^\actionn_{c,m,n}}"{pos=0.3}, from=2-1, to=1-2]
	\end{diagram}
	\begin{diagram}
	\label{diag:bimod-coherence-triang1}
		{j \action (c \actionn m)} && {(j \action c) \actionn m} \\
		& {c \actionn m}
		\arrow["{\eta^\action_c \actionn m}"'{pos=0.7}, from=2-2, to=1-3]
		\arrow["{\eta^\action_{c \actionn m}}"{pos=0.7}, from=2-2, to=1-1]
		\arrow["{\zeta_{j,c,m}}", from=1-1, to=1-3]
	\end{diagram}
	\begin{diagram}
	\label{diag:bimod-coherence-triang2}
		{m \action (c \actionn j)} && {(m \action c) \actionn j} \\
		& {m \action c}
		\arrow["{\eta^\actionn_{m \action c}}"', from=2-2, to=1-3]
		\arrow["{m \action \eta^\actionn_{c}}", from=2-2, to=1-1]
		\arrow["{\zeta_{m,c,j}}", from=1-1, to=1-3]
	\end{diagram}
\end{definition}

Like one-sided $\cat M$-actegories, $\cat M$-biactegories form a 2-category:

\begin{definition}
\label{def:bilinear-functor}
	A \textbf{lax $(\cat M, \cat N)$-bilinear functor} between $(\cat M, \cat N)$-biactegories $(\cat C, \action, \actionn, \zeta^{\cat C})$ and $(\cat D, \laction, \raction, \zeta^{\cat D})$ is a functor $F: \cat C \to \cat D$ equipped with a left lax $\cat M$-linear structure $\ell$ and a right lax $\cat N$-linear structure $r$, obeying the following compatibility axiom:
	\begin{diagram}
	\label{diag:bilinear-hexagon}
		&[-6.5ex] {m \laction (F(c) \raction n)} & {(m \laction F(c)) \raction n} &[-6.5ex]\\
		{m \laction F(c \actionn n)} &&& {F(m \action c) \raction n} \\
		& {F(m \action (c \actionn n))} & {F((m \action c) \actionn n)}
		\arrow["{\ell_{m,c} \action n}", from=1-3, to=2-4]
		\arrow["{F(\zeta^{\cat C}_{m,c,n})}"', from=3-2, to=3-3]
		\arrow["{\zeta^{\cat D}_{m,F(c),n}}", from=1-2, to=1-3]
		\arrow["{m \action r_{c,n}}"', from=1-2, to=2-1]
		\arrow["{\ell_{m,c \actionn n}}"', from=2-1, to=3-2]
		\arrow["{r_{m \action c, n}}", from=2-4, to=3-3]
	\end{diagram}
\end{definition}

\begin{remark}
	Observe bilinear functors are not maps which are `linear in each argument separately'--those are balanced maps, the subject of the next section (Definition~\ref{def:balanced}).
	The two are sometimes confused since it is often assumed the left and right structures of a biactegory (or a bimodule) `coincide' (see Proposition~\ref{prop:mirroring-biact}), in which case balance and bilinearity are related.
\end{remark}

\begin{definition}
\label{def:bilinear-transformation}
	A \textbf{bilinear transformation} $\xi : (F, \ell, r) \twoto (G, \nu, \varpi): (\cat C, \action, \actionn, \zeta^{\cat C}) \to (\cat D, \laction, \raction, \zeta^{\cat D})$ is a natural transformation $F \twoto G$ which is both a linear transformation (Definition~\ref{def:linear-trans}) $(F, \ell) \twoto (G, \nu)$ (`on the left') and $(F, r) \twoto (G, \varpi)$ (`on the right').
\end{definition}

\begin{definition}
	The 2-category of $(\cat M, \cat N)$-biactegories, lax bilinear functors and bilinear natural transformation is denoted by $\Biact{\cat N}^\lax$.
\end{definition}

Consider the composite pseudomonad of left $\cat M$- and right $\cat N$-actegories.
This pseudomonad exists because there is a distributive law of pseudomonads (as defined in \cite[§4]{marmolejo1999distributive})
\begin{equation}
	\swap \times - : \cat N^\rev \times \cat M \times - \longtwoto \cat M \times \cat N^\rev \times -
\end{equation}
given by symmetry of $\times$ in $\Cat$.
The four structure morphisms such a distributive law should be endowed with \cite[§4]{marmolejo1999distributive} are all trivial in this simple case.
Moreover, we can strictify $\cat M$ and $\cat N$ to make the coherence laws hold strictly \cite[(coh 1)--(coh 9)]{marmolejo1999distributive}, thus making the pseudomonad structure of $\cat N^\rev \times \cat M \times -$ \cite[§5]{marmolejo1999distributive} more manageable.

Ultimately, as anticipated, we have the following result which equates the definition of biactegories as pseudoalgebras of this composite monad to the definition we gave in terms of pairs of compatible actions:

\subsubsection{Theorem.}
\label{th:biact-are-bialg}
\begin{restheorem}{biactegoriesarebialgebras}
	There is an equivalence of 2-categories
	\begin{equation}
		\Biact[\cat M]{\cat N}^\lax \equi \Act{(\cat N^\rev \times \cat M)}^\lax.
	\end{equation}
	\vspace{-5ex}
\end{restheorem}
\begin{proof}
	See Appendix~\ref{appendix:proofs}.
\end{proof}

A very important fact for the following section (and also for the `zoology' of biactegories) is that left or right $\cat M$-actegories can be canonically promoted to $\cat M$-biactegories every time $\cat M$ is braided, since in that case there is an isomorphism $\cat M^\rev \iso \cat M$.
This can also be interpreted as saying that actegories for braided categories are biactegories in the sense of `sequential composition', that is, for two left actions (which in this case happen to coincide).

\subsubsection{Proposition.}
\label{prop:mirroring-biact}
\begin{resproposition}[Mirroring]{leftmodofsymisbi}
	Let $(\cat C, \action)$ be a left $\cat M$-actegory, where $(\cat M, j, \mtimes, \beta)$ is a braided monoidal category.
	Then there is a canonical right $\cat M$-actegorical structure on $\cat C$ and a canonical $\cat M$-biactegorical structure induced by the braiding.
	We call the resulting biactegory the \emph{mirroring} of $(\cat C, \action)$.
\end{resproposition}
\begin{proof}
	Consider a braiding $\beta$ on $\cat M$ as a strong monoidal structure on the identity functor $\cat M^\rev \to \cat M$.
	We can pullback $\action$ along this morphism (Proposition~\ref{prop:actegories-indexed-over-moncat}) to obtain a left $\cat M^\rev$-action $\revrel{\action}$ on $\cat C$.
	Concretely, this is given by the same functor $- \action {=} : \cat M^\rev \times \cat C \to \cat C$, now equipped with
	\begin{equation}
		\eta^\rev_c := c \revrel{\action} j \equalto j \action c \nlongto{\eta_c} c,
	\end{equation}
	\begin{eqalign}
		\mu^\rev_{c,m,n} :=\ (c \revrel{\action} n) \revrel{\action} m &\equalto m \action (n \action c)\\
		&\nlongto{\mu_{n,n,c}} (m \mtimes n) \action c\\
		&\nlongto{\beta \action c} (n \mtimes m) \action c\\
		&\equalto (m \revrel{\mtimes} n) \action c.
	\end{eqalign}
	To make $(\cat C, \action, \revrel{\action})$ into an $\cat M$-biactegory, hence a natural isomorphism
	\begin{equation}
		\zeta_{m,c,n} : m \action (c \revrel{\action} n) \equalto m \action (n \action c) \isolongto n \action (m \action c) \equalto (m \action c) \revrel{\action} n
	\end{equation}
	that we obtain by conjugating $\beta$ (or better, the action thereof) with $\mu$:
	\begin{equation}
	\label{eq:bimodulator-for-symm}
		m \action (n \action c) \nlongto{\mu^{-1}_{m,n,c}} (m \mtimes n) \action c \nlongto{\beta_{m,n} \action c} (n \mtimes m) \action c \nlongto{\mu_{n,m,c}} n \action (m \action c).
	\end{equation}
	To prove this is a well-defined bimodulator, we appeal to strictification (Lemma~\ref{lemma:strictification}) and move to a setting where both the monoidal structure of $\cat M$ and the structure of $\action$ are strict.
	Note this doesn't apply to the braiding of $\cat M$, which can't be strictified.
	Anyway, in this setting we notice immediately that~\eqref{diag:bimod-coherence-pentag1} and~\eqref{diag:bimod-coherence-pentag2} collapse to the two hexagonal axioms for a braiding \cite[Definition 1.2.34]{johnson2021}.
	The two triangular axioms~\eqref{diag:bimod-coherence-triang1}--\eqref{diag:bimod-coherence-triang2} hold for analogous reasons.
\end{proof}

\begin{corollary}
	Since $(\cat M^\rev)^\rev = \cat M$, any right $\cat M$-actegory, for $\cat M$ braided, can be canonically turned into an $\cat M$-biactegory.
\end{corollary}

\begin{remark}
	The above proves that $\Act{\cat M}^\lax$ embeds into $\Biact{\cat M}^\lax$, but this embedding is not wide, even when $\cat M$ is braided.
	There still are biactegories whose left and right action are genuinely distinct and not simply the symmetrization of each other.
	Despite this, often we abuse notation denoting the left and right structure of $\cat M$-biactegory with the same symbol.
\end{remark}

\begin{example}[Monoidal categories]
\label{ex:moncat-self-biaction}
	As observed in Example~\ref{ex:moncat-self-actions}, every monoidal category $(\cat C, i, \ctimes)$ is canonically a left and right self-actegory (both supported by $\ctimes$).
	It is trivial to observe that these two actions  are compatible in the sense of biactegories:
	\begin{equation}
		\zeta_{a,b,c} := \alpha^{-1}_{a,b,c} : a \ctimes (b \ctimes c) \isolongto (a \ctimes b) \ctimes c.
	\end{equation}
	We call $(\cat C, \ctimes, \ctimes)$ the \emph{canonical self-biactegory} associated to $\cat C$.
\end{example}

\begin{example}[Braided monoidal categories]
\label{ex:braidcat-mirror-biaction}
	When $(\cat C, i, \ctimes)$ comes with a braiding $\beta$, there is an additional way to make it into a self-biactegory, given by the mirroring construction described above (Proposition~\ref{prop:mirroring-biact}) applied to the left action.
	This yields the bicategory $(\cat C, \ctimes, \revrel{\ctimes})$, where the bimodulator is given by braiding:
	\begin{eqalign}
		m \ctimes (c \revrel{\ctimes} n) &\equalto m \ctimes (n \ctimes c)\\
		&\nlongto{\alpha^{-1}_{m,n,c}} (m \ctimes n) \ctimes c\\
		&\nlongto{\beta_{m,n} \ctimes c} (n \ctimes m) \ctimes c\\
		&\nlongto{\alpha_{n,m,c}} c \ctimes (m \ctimes c)\\
		&\equalto (m \ctimes c) \revrel{\ctimes} n.
	\end{eqalign}
	The identity map $1_{\cat C} : \cat C \to \cat C$ is a bilinear functor from $(\cat C, \ctimes, \ctimes)$ (the canonical self-biactegory) to $(\cat C, \ctimes, \revrel{\ctimes})$, equipped with the linear structures
	\begin{eqalign}
		\ell_{c,d} &: c \ctimes d \equalto c \ctimes d,\\
		r_{c,d} &: c \ctimes d \nlongto{\beta_{c,d}} d \ctimes c \equalto c \revrel{\ctimes} d.
	\end{eqalign}
	The hexagonal coherence for these then reproduces the first hexagonal identity for $\beta$.
	We recover the second one from repeating the same construction for $\cat C^\rev$.
\end{example}


	\subsection{Tensor product of actegories}
\label{sec:tensor-product}
Like any category of `modules', actegories possess a natural notion of monoidal product, namely that of tensor product.
This notion is indeed totally analogous to the eponymous product on vector spaces, modules, and the closely related cousins of actegories, module categories \cite{ostrik2003module}.
Also, let us remind the reader we are not the first to observe the existence of such a product, nor to define it explicitly. See, for instance, \cite{skoda2006biactegories,skoda2009some}.

The tensor operation is defined between actegories of different handedness: the tensor $\cat D \tensor \cat C$ requires $\cat D$ to be a \emph{right} $\cat M$-actegory and $\cat C$ to be a \emph{left} one.
As usual, these actions are then `used up' and are not available on $\cat D \tensor \cat C$ anymore.
However, if further actegorical structures are present, then these are not affected. Therefore if $\cat D$ is a left $\cat N$-actegory or $\cat C$ is a right $\cat P$-actegory, so will be $\cat D \tensor \cat C$.

Here we define the tensor product of actegories using its universal properties of representing \emph{balanced} functors, i.e.~functors $F: \cat D \times \cat C \to \cat E$ such that $F(d \actionn m, c) \iso F(d, m \action c)$, which is of course the categorification of the notion of balanced map of modules, where the isomorphism is simply an equality.
That isomorphism is now structure obeying its own set of coherence laws:

\begin{definition}
\label{def:balanced}
	Let $(\cat C, (\action, \eta^\action, \mu^\action))$ be a left $\cat M$-actegory, $(\cat D, (\actionn, \eta^\actionn, \mu^\actionn))$ a right $\cat M$-actegory and $\cat E$ a category.
	An \textbf{$\cat M$-balanced} functor is a functor $F:\cat D \times \cat C \to \cat E$ together with a natural isomorphism
	\begin{equation}
		\varepsilon_{d,m,c} : F(d \actionn m, c) \isolongto F(d, m \action c)
	\end{equation}
	called \textbf{equilibrator}, obeying the following coherence laws:
	\begin{diagram}[row sep=5ex, column sep=-3ex]
		\label{diag:bal-coherence-pentag}
		& {F(d \actionn (m \mtimes n), c)} && {F(d, (m \mtimes n) \action c)} \\
		{F((d \actionn m) \actionn n, c)} &&&& {F(d, m \action (n \action c))} \\
		&& {F(d \actionn m, n \action c)}
		\arrow["{\varepsilon_{d, m\mtimes n, c}}", from=1-2, to=1-4]
		\arrow["{F(\mu^\actionn_{d,m,n},c)}", from=2-1, to=1-2]
		\arrow["{\varepsilon_{d \actionn m,n,c}}"', from=2-1, to=3-3]
		\arrow["{\varepsilon_{d,m,n\action c}}"', from=3-3, to=2-5]
		\arrow["{F(d, \mu^\action_{m,n,c})}"', from=2-5, to=1-4]
	\end{diagram}
	\begin{diagram}
	\label{diag:bal-coherence-triang}
		{F(d \actionn j, c)} && {F(d, j \action c)} \\
		& {F(d,c)}
		\arrow["{F(\eta^\actionn, c)}", from=2-2, to=1-1]
		\arrow["{F(d, \eta^\action)}"', from=2-2, to=1-3]
		\arrow["{\varepsilon_{d,j,c}}", from=1-1, to=1-3]
	\end{diagram}
\end{definition}

\begin{definition}
	Let $(F, \varepsilon), (G, \vartheta) : \cat D \times \cat C \to \cat E$ be $\cat M$-balanced functors.
	An \textbf{$\cat M$-balanced transformation} is a natural transformation $\xi : F \twoto G$ for which all the following squares commute:
	\begin{diagram}
	\label{diag:balanced-nat-trans}
		{F(d \actionn m, c)} & {F(d, m \action c)} \\
		{G(d \actionn m, c)} & {G(d, m \action c)}
		\arrow["{\xi_{d \actionn m, c}}"', from=1-1, to=2-1]
		\arrow["{\xi_{d, m \action c}}", from=1-2, to=2-2]
		\arrow["{\varepsilon_{d,m,c}}", from=1-1, to=1-2]
		\arrow["{\vartheta_{d,m,c}}"', from=2-1, to=2-2]
	\end{diagram}
\end{definition}

Evidently, balanced functors $\cat D \times \cat C \to \cat E$ and balanced transformations thereof gather in a category $\Bal(\cat D \times \cat C, \cat E)$, and this construction is 2-functorial in $\cat E$.

\begin{definition}
\label{def:tensor-actegories}
	Let $(\cat C, \action)$ be a left $\cat M$-actegory and $(\cat D, \actionn)$ a right $\cat M$-actegory.
	Their \textbf{tensor product} is the corepresenting object $\cat D \tensor \cat C$ of the 2-functor $\Bal(\cat D \times \cat C, -)$.
	That is, it is a category $\cat D \tensor \cat C$ together with the structure of an isomorphism
	\begin{equation}
		\Bal(\cat D \times \cat C, \cat E) \;\iso\; \Cat(\cat D \tensor \cat C, \cat E), \qquad \text{for all $\cat E : \Cat$}.
	\end{equation}
\end{definition}

One can `read off' a concrete construction of the tensor product from this universal property.
Saying every functor out of $\cat D \tensor \cat C$ is automatically balanced implies such a category must have the same objects as $\cat D \times \cat C$, but more isomorphisms, namely all those of the form $(d \actionn m, c) \iso (d, m \action c)$.

\begin{proposition}
\label{prop:tens-is-pseudocoeq}
	Let $(\cat C, \action)$ be a left $\cat M$-actegory and $(\cat D, \actionn)$ a right $\cat M$-actegory.
	Their tensor product $\cat D \tensor \cat C$ is given by the following isocoinserter (in $\Cat$):
	\begin{equation}
		\begin{tikzcd}[ampersand replacement=\&, sep=scriptsize]
			{\cat D \times \cat M \times \cat C} \&\& {\cat D \times \cat C} \& {\cat D \tensor \cat C}
			\arrow["{\actionn \times \cat C}", shift left, from=1-1, to=1-3]
			\arrow["{\cat D \times \action}"', shift right, from=1-1, to=1-3]
			\arrow["Q", from=1-3, to=1-4]
		\end{tikzcd}
	\end{equation}
	meaning it comes equipped with a universal natural isomorphism $\tau$ which coequalizes, up to iso, the two actions of $\cat M$:
	\begin{equation}
		\begin{tikzcd}[ampersand replacement=\&, sep=scriptsize, row sep=tiny]
			\&\& {\cat D \times \cat C} \\
			{\cat D \times \cat M \times \cat C} \&\&\& {\cat D \tensor \cat C} \\
			\&\& {\cat D \times \cat C}
			\arrow["{\actionn \times \cat C}", shift left, curve={height=-5pt}, from=2-1, to=1-3]
			\arrow["Q", curve={height=-8pt}, from=1-3, to=2-4]
			\arrow["{\cat D \times \action}"', shift right, curve={height=5pt}, from=2-1, to=3-3]
			\arrow["Q"', curve={height=8pt}, from=3-3, to=2-4]
			\arrow["\tau", "\wr"', shift right=2, shorten <=6pt, shorten >=6pt, Rightarrow, from=1-3, to=3-3]
		\end{tikzcd}
	\end{equation}
\end{proposition}
\begin{proof}
	It suffices to prove that if $F: \cat D \times \cat C \to \cat E$ is balanced, then it yields a unique functor out of $\cat D \tensor \cat C$ (which here denotes the pseudocoequalizer) and \emph{vice versa}.
	But this is immediate: being balanced means exactly coequalizing up to iso the two maps that $Q$ is a pseudocoequalizer of, thus by universal property there is an isomorphism as in Definition~\ref{def:balanced}, or, \emph{vice versa}, such an isomorphism provides the required universal map.
\end{proof}

We can describe such a isocoinserter as the category obtained by adjoining isomorphisms and equations to $\cat D \times \cat C$:

\begin{proposition}
\label{prop:tensor-actegory-desc}
	Let $(\cat C, (\action, \eta^\action, \mu^\action))$ be a left $\cat M$-actegory and $(\cat D, (\actionn, \eta^\actionn, \mu^\actionn))$ a right $\cat M$-actegory.
	Their tensor product is the category $\cat D \tensor \cat C$ has the same objects and morphisms of $\cat D \times \cat C$, together with a family of isomorphisms
	\begin{equation}
		\tau_{d,m,c} : (d \actionn m, c) \isolongto (d, m \action c), \qquad \text{for all $m : \cat M$, $d : \cat D$, $c : \cat C$}.
	\end{equation}
	satisfying the following relations:
	\begin{enumerate}[label=\Roman*)]
		\item For every morphism $f: d \to d'$ in $\cat D$, $g:c \to c'$ in $\cat C$ and $\alpha : m \to n$ in $\cat M$, the following square must commute:
		\begin{diagram}
		\label{diag:tensor-relations-nat-square}
			{(d \actionn m, c)} & {(d'\actionn n, c')} \\
			{(d, m \action c)} & {(d',n \action c')}
			\arrow["{(f \actionn \alpha, g)}", from=1-1, to=1-2]
			\arrow["{(f, \alpha \action g)}"', from=2-1, to=2-2]
			\arrow["{\tau_{d',n,c'}}", from=1-2, to=2-2]
			\arrow["{\tau_{d,m,c}}"', from=1-1, to=2-1]
		\end{diagram}
		\item For every $d: \cat D$, $m,n : \cat M$, $c:\cat C$, the following diagrams must commute:
		\begin{diagram}[row sep=5ex, column sep=-3ex]
		\label{diag:tensor-relations-pentag}
			& {(d \actionn (m \mtimes n), c)} && {(d, (m \mtimes n) \action c)} \\
			{((d \actionn m) \actionn n, c)} &&&& {(d, m \action (n \action c))} \\
			&& {(d \actionn m, n \action c)}
			\arrow["{\tau_{d, m\mtimes n, c}}", from=1-2, to=1-4]
			\arrow["{(\mu^\actionn_{d,m,n},c)}", from=2-1, to=1-2]
			\arrow["{\tau_{d \actionn m,n,c}}"', from=2-1, to=3-3]
			\arrow["{\tau_{d,m,n\action c}}"', from=3-3, to=2-5]
			\arrow["{(d, \mu^\action_{m,n,c})}"', from=2-5, to=1-4]
		\end{diagram}
		\begin{diagram}
		\label{diag:tensor-relations-triang}
			{(d \actionn j, c)} && {(d, j \action c)} \\
			& {(d,c)}
			\arrow["{(\eta^\actionn, c)}", from=2-2, to=1-1]
			\arrow["{(d, \eta^\action)}"', from=2-2, to=1-3]
			\arrow["{\tau_{d,j,c}}", from=1-1, to=1-3]
		\end{diagram}
	\end{enumerate}
\end{proposition}

\begin{proposition}
	Let $(\cat C, \action, \action)$ be an $(\cat N, \cat M)$-biactegory and $(\cat D, \actionn, \actionn)$ be an $(\cat M, \cat P)$-biactegory.
	Then $\cat D \tensor \cat C$ is an $(\cat N, \cat P)$-biactegory, inheriting the left action from $\cat D$ and the right action from $\cat C$.
\end{proposition}
\begin{proof}
	On $\cat D \tensor \cat C$, we define
	\begin{equation}
		m \actionn (d,c) := (m \actionn d, c)
	\end{equation}
	and keep $\eta$ and $\mu$ of the left action $\actionn$ of $\cat D$.
	The same applies on the right, and the bimodulator for $\cat D \tensor \cat C$ is trivial:
	\begin{equation}
		\zeta_{m,(d,c), n} : m \actionn ((d,c) \action n) \equalto (m \action d, c \action n) \equalto
		(m \actionn (d,c)) \action n.
	\end{equation}
\end{proof}

The fact the tensor of actegories is a more natural choice of monoidal product between actegories can also be motivated by the following---which sometimes is used as definition:

\subsubsection{Proposition.}
\label{prop:biact-internal-hom}
\begin{resproposition}{biactinternalhom}
	Let $\cat C$ be an $(\cat N, \cat M)$-biactegory, $\cat D$ be a $(\cat M, \cat P)$-biactegory and $\cat E$ be a $(\cat N, \cat P)$-biactegory, let $\action$ denote all actions by abuse of notation.
	Then the category of (right) $\cat P$-linear maps $\cat D \to \cat E$, which we denote by $[\cat D, \cat E]_{\cat P}$ has a canonical $(\cat N, \cat M)$-biactegory structure, and is such that:
	\begin{equation}
	\label{eq:internal-hom-adj}
		\Biact[\cat N]{\cat P}(\cat C \tensor \cat D, \cat E) \iso \Biact[\cat N]{\cat M}(\cat C, [\cat D, \cat E]_{\cat P}).
	\end{equation}
	We refer to this bracket as the \textbf{internal hom $(\cat N, \cat M)$-biactegory} of $\cat D$ and $\cat E$.
\end{resproposition}
\begin{proof}
	See Appendix~\ref{appendix:proofs}.
\end{proof}

\begin{remark}
	The above doesn't hold if we take biactegories with \emph{lax} linear functors between them---that would instead require $\tensor$ to be replaced by an asymmetric version in which $\tau$ from Proposition~\ref{prop:tens-is-pseudocoeq} is not asked to be invertible (thus defining a coinserter instead).%
	\footnote{We made this mistake in a previous version of this work. We thank Subhajit Das for bringing this problem to our attention.}
\end{remark}

By duality, we also have (note the switch between $\cat C$ and $\cat D$ during transposition):

\begin{corollary}
\label{cor:dual-tensor-hom}
	Let $\cat C$ be an $(\cat M, \cat N)$-biactegory, $\cat D$ be a $(\cat P, \cat M)$-biactegory and $\cat E$ be a $(\cat P, \cat N)$-biactegory.
	Then the category of (left) $\cat P$-linear maps $\cat D \to \cat E$, which we denote by $[\cat D, \cat E]_{\cat P}$ has a canonical $(\cat M, \cat N)$-biactegory structure, and is such that:
	\begin{equation}
	\label{eq:internal-hom-dual-adj}
		\Biact[\cat P]{\cat N}(\cat D \tensor \cat C, \cat E) \iso \Biact[\cat M]{\cat N}(\cat C, [\cat D, \cat E]_{\cat P}).
	\end{equation}
\end{corollary}


\subsubsection{Corollary.}
\label{cor:biactegories-are-moncat}
\begin{rescorollary}{biactaremoncat}
	The tensor product of $\cat M$-actegories $\tensor$ can be equipped with the structure of a monoidal product on $\Biact{\cat M}^\lax$.
	Its restriction to $\Biact{\cat M}$ is pseudoclosed, with internal hom $[-,=]_{\cat M}$.
\end{rescorollary}
\begin{proof}
	The extra structure required to obtain a monoidal structure on the 2-category $\Biact{\cat M}^\lax$ is described in Appendix~\ref{appendix:proofs}.
	The fact $[-,=]_{\cat M}$ is a pseudoclosed structure has been proven in Proposition~\ref{prop:biact-internal-hom}.
\end{proof}



We can now prove the following:

\begin{proposition}
\label{prop:extension-of-scalars}
	Let $R : \cat M \to \cat N$ be a strong monoidal functor, and let $R^* : \Act{\cat N} \to \Act{\cat M}$ be the associated restriction of scalars, as defined in Proposition~\ref{prop:actegories-change-of-base}.
	Then $R^*$ has a left adjoint $R_!$, called \textbf{extension of scalars}, given by
	\begin{equation}
		R_!(\cat C, \action) \iso \cat N_R \tensor \cat C.
	\end{equation}
	where $\cat N_R$ is the canonical biactegory associated to $\cat N$ except the right action has been restricted by $R$.
\end{proposition}
\begin{proof}
	Notice if $\cat D$ is a left $\cat N$-actegory, we have $R^*(\cat D) \iso{} [\cat N_R, \cat D]_{\cat N}$, since $[\cat N, \cat D]_{\cat N} \iso \cat D$ as categories, and the right $\cat M$-action on $\cat N_R$ permeates through the hom (Proposition~\ref{prop:biact-internal-hom}) to give the same left $\cat M$-action as $R^*(\cat D)$.
	Therefore, using the tensor-hom adjunction described in Corollary~\ref{cor:dual-tensor-hom}, we obtain the desired transposition isomorphism:
	\begin{equation}
		\Act{\cat M}(\cat C, R^*(\cat D))
		\iso
		\Act{\cat M}(\cat C, [\cat N_R, \cat D]_{\cat N})
		\iso
		\Act{\cat N}(\cat N_R \tensor \cat C, \cat D).
	\end{equation}
\end{proof}

\begin{proposition}
\label{prop:free-act-is-monoidal}
	Let $\cat C$ and $\cat D$ be categories, then
	\begin{equation}
		\cat M[\cat C] \tensor \cat M[\cat D] \equi \cat M[\cat C \times \cat D].
	\end{equation}
\end{proposition}
\begin{proof}
	Recall from Proposition~\ref{prop:free-act-is-copow} that $\cat M[-] \equi - \copow \cat M$, where $\copow$ denotes the $\Cat$-copower operation (as described, for instance, in Example~\ref{ex:copower} or \cite[§3.7]{kelly1982basic}) on $\Act{\cat M}$.
	We proved $\tensor$ is left adjoint (on both sides), so it preserves colimits. Hence we have
	\begin{equation}
		\cat M[\cat C] \tensor \cat M[\cat D] \iso (\cat C \copow \cat M) \tensor (\cat D \copow \cat M) \iso \cat C \copow (\cat D \copow \cat M).
	\end{equation}
	If we can prove $\cat C \copow (\cat D \copow \cat M) \iso (\cat C \times \cat D) \copow \cat M$, we are done. Indeed, we can prove the first satisfies the universal property of the latter:
	\begin{eqalign}
		\Act{\cat M}(\cat C \copow (\cat D \copow \cat M), -) &\iso \Cat(\cat C, \Act{\cat M}(\cat D \copow \cat M, -))\\
		&\iso \Cat(\cat C, \Cat(\cat D, \Act{\cat M}(\cat M, -)))\\
		&\iso \Cat(\cat C \times \cat D, \Act{\cat M}(\cat M, -))\\
		&\iso \Act{\cat M}((\cat C \times \cat D) \copow \cat M, -).
	\end{eqalign}
\end{proof}


	\newpage
	\section{Interaction of monoidal and actegorical structures}
\label{sec:monoidal-actegories}
For the purposes of categorical cybernetics, it is almost impossible to avoid pondering situations in which an $\cat M$-actegory $\cat C$ is also equipped with a monoidal structure $(i,\ctimes)$.
When this happens, it is natural to ask whether the two structures are compatible in any way.
In practice, this compatibility turns out to be a requirement for making some constructions run smoothly, chiefly the $\Para$ construction \cite{capucci2021towards}.
In fact for the latter construction to be monoidal,\footnote{In the bicategorical sense, see \cite[Explanation 12.1.3]{johnson2021}.} we need exactly the data of an $\cat M$-actegory $\cat C$ equipped with a monoidal structure \emph{and} a natural morphism
\begin{equation}
\label{eq:interchanger}
	\iota_{m,n,c,d} : (m \mtimes n) \action (c \ctimes d) \longto (m \action c) \ctimes (n \action d)
\end{equation}
In \cite[§2.1]{capucci2021towards} this was called \emph{mixed interchanger}, and it was observed it can be decomposed in strength-like morphisms
\begin{eqalign}
\label{eq:strengths}
	\kappa_{m,c,d} : m \action (c \ctimes d) \longto c \ctimes (m \action d),\\
	\chi_{m,c,d} : m \action (c \ctimes d) \longto (m \action c) \ctimes d.
\end{eqalign}
The three were related by making heavy use of symmetric structure on $\cat M$, and it wasn't clear which properties of $\cat M$, $\cat C$ and $\action$ were actually relevant.

In this section, we expound the situation by framing those sparse observations in the context of the theory of actegories we have built so far.
We are going to investigate multiple ways to combine actegorical and monoidal structure, as summarized later in Table~\ref{table:mon-act-distributive-laws}. In Section~\ref{subsec:monoidal-actegories} we define `monoidal actegories' as pseudoactions in $\MonCat$ and show these are those for which the $\Para$ and $\Optic$ constructions are monoidal. Later, in Section~\ref{subsec:algebroidal-actegories}, we study pseudomonoids in $\Act{\cat M}$, taken with respect to either its tensor or its cartesian monoidal structures. We show these turn out to be equivalent to monoidal actegories in the first case, and explore some connections with hybrid optics in the second.
In Section~\ref{subsec:braided-monact} we extend the previous definitions to the braided and symmetric cases.
Finally, in Section~\ref{subsec:classifying} we prove three Cayley-type classification theorems for monoidal, braided and symmetric actegories.

\subsection{Monoidal actegories}
\label{subsec:monoidal-actegories}
The easiest approach to describe actions on monoidal categories is to just have the action of $\cat M$ to take place in $\MonCat$.
To even talk about `actions' of $\cat M$, $\cat M$ has to be a (pseudo)monoid in $\MonCat$, i.e.~a braided monoidal category.
Hence, from now on, unless explicitly specified, we will assume $\cat M$ is equipped with a braiding $\beta_{m,n}: m \mtimes n \isoto n \mtimes m$.

\subsubsection{Definition.}
\label{def:monoidal-actegory}
\begin{resdefinition}{monact}
	An \textbf{oplax monoidal left $\cat M$-actegory} is a monoidal category $(\cat C, i, \ctimes)$ equipped with an oplax monoidal functor $\action : \cat M \times \cat C \to \cat C$ and two monoidal natural transformations $\mu$ and $\eta$ defined analogously to Definition~\ref{def:left-actegory}.
	When $\action$ is strong, we call it simply \textbf{monoidal actegory}.
\end{resdefinition}

\begin{remark}
\label{rmk:upsilon-is-redundant}
	Explicitly, the monoidal structure on $\action$ is given by natural morphisms
	\begin{eqalign}
		\upsilon &: i \longto j \action i,\\
		\iota_{m,c,n,d} &: (m \mtimes n) \action (c \ctimes d) \longto (m \action c) \ctimes (n \action d).
	\end{eqalign}
	where we still call $\iota$ the \emph{mixed interchanger}.
	As $\eta$ is now a \emph{monoidal} natural transformation, it comes equipped with two additional compatibility conditions, one of which unpacks to the following diagram
	\begin{diagram}[ampersand replacement=\&, sep=4ex, row sep=2ex]
		\& i \\
		i \&\& {j \action i}
		\arrow[Rightarrow, no head, from=1-2, to=2-1]
		\arrow["{\eta_i}", from=1-2, to=2-3]
		\arrow["\upsilon"', from=2-1, to=2-3]
	\end{diagram}
	This diagram tells us that when defining an oplax monoidal structure on an actegory, all we need to define is the mixed interchanger naturally, since $\upsilon$ can always be defined as $\eta_i^{-1}$.
	We have drawn the remaining coherence laws satisfied by $\mu$, $\eta$ and $\iota$ in Appendix~\ref{appendix:defs}.
\end{remark}

\begin{remark}
	Diagram~\eqref{diag:mu-monoidal} (monoidality of $\mu$) shows why $\cat M$ needs to be braided to even spell this definition out: we need to be able to talk about the monoidal structure of $\mtimes$ itself, which is indeed given by the braiding.
\end{remark}

\begin{remark}
\label{rmk:monoidal-actegory-strengths}
	From an interchanger, both $\kappa$ and $\chi$ (as in~\eqref{eq:strengths}) can be derived:
	\begin{eqalign}
		\kappa_{m,c,d} := m \action (c \ctimes d) &\nlongto{\lambda^{-1}_m \action (c \ctimes d)} (j \mtimes m) \action (c \ctimes d)\\
		&\nlongto{\iota_{j,c,m,d}} (j \action c) \ctimes (m \action d)\\
		&\nlongto{\eta_c \ctimes (m \action d)} c \ctimes (m \action d),\\[1.5ex]
		\chi_{m,c,d} := m \action (c \ctimes d) &\nlongto{\rho^{-1}_m \action (c \ctimes d)} (m \mtimes j) \action (c \ctimes d)\\
		&\nlongto{\iota_{m,j,c,d}} (m \action c) \ctimes (j \action d)\\
		&\nlongto{(m \action c) \ctimes \eta_d} (m \action c) \ctimes d.
	\end{eqalign}
	Vice versa, $\kappa$ and $\chi$ determine $\iota$:
	\begin{eqalign}
		\iota_{m,n,c,d} := (m \mtimes n) \action (c \ctimes d) &\nlongto{\mu^{-1}_{m,n,c\ctimes d}} m \action (n \action (c \ctimes d))\\
		&\nlongto{m \action \kappa_{n,c,d}} m \action (c \ctimes (n \action d))\\
		&\nlongto{\chi_{n,c,d}} (m \action c) \ctimes (n \action d)).
	\end{eqalign}
\end{remark}

\begin{remark}
\label{rmk:mon-acts-are-biacts}
	Observe that monoidal actegories could be presented as particular biactegories. In fact a monoidal $\cat M$-actegory is equivalently a $(\cat M, \cat C)$-biactegory and a $(\cat M, \cat C^\rev)$-biactegory: the structure morphisms $\kappa$ and $\chi$ are bimodulators for $\action$ and $\boxtimes$, where the latter is used both as a left action (hence $\kappa$) and a right action (hence $\chi$), as in Example~\ref{ex:moncat-self-biaction}.
\end{remark}

It is evident the definition of lax linear functors and linear natural transformation can be adapted too to this case, by asking the underlying functors and natural transformations, respectively, to be monoidal. Therefore there is an indexed 2-category $(-)\dash\Actt_\MonCat^\lax$, analogous to the one described in Proposition~\ref{prop:actegories-indexed-over-moncat}, giving rise to a 2-category $\Actt_\MonCat^\lax$.
The same applies to all the other kinds of actegories we are going to define in this section, so we won't bother the reader with a similar remark again and we will tacitly consider analogous 2-categories to have been defined.

\begin{example}
	The most trivial example of monoidal actegory is given by the self-action (Example~\ref{ex:moncat-self-actions}) of any braided monoidal category.
\end{example}

\begin{example}
	In \cite{cockett2009logic}, the central structure is that of \emph{linear actegory}.
	Using our terminology, it is an $\cat M$-diactegory (Remark~\ref{rmk:diactegories}) supported by a category $\cat C$ equipped with two monoidal products $\boxtimes$, $\boxplus$ and isomorphisms $a \boxtimes (b \boxplus c) \iso (a \boxtimes b) \boxplus c \iso b \boxplus (a \boxtimes c)$.
	Then $\cat C$ is a monoidal $\cat M$-actegory since the $\cat M$-action $\action$ and the $\cat M^\op$-action $\actionn$ are required to be monoidal with respect to both $\boxtimes$ and $\boxplus$. Moreover, $m \action - \adj m \actionn -$ for every $m : \cat M$.
\end{example}

\begin{example}
\label{ex:cart-cocart-oplax-monact}
	Similarly as above, suppose $\cat C$ is cartesian and cocartesian monoidal. Then there always is an oplax monoidal structure on the self-action induced by $\times$ on $(\cat C, 0, +)$:
	\begin{equation*}
		(a \times b) \times (c + d) \nlongto{\exists!} a \times b \times c + a \times b \times d \nlongto{\pi_{a,c} + \pi_{b,d}} a \times c + b \times d.
	\end{equation*}
\end{example}

As anticipated, we investigated the notion of monoidal actegory in order to capture the structure needed to make the results of $\Copara$ and $\Para$ constructions monoidal bicategories.
For the sake of clarity, we remind the reader of \cite[Definition 2]{capucci2021towards}.
Given an $\cat M$-actegory $(\cat C, \action)$, the bicategory $\Para_\action(\cat C)$ has the same objects of $\cat C$ but a morphism $(m,f):c \to d$ is now given by a choice of scalar $m : \cat M$ and a morphism $f:m \action c \to d$. Composition and identities are built from the multiplicator and unitor of $\action$.

Suppose now we additionally equip $\cat C$ with the monoidal structure $(i, \ctimes)$. The object part of $\ctimes$ trivially extends to $\Para_\action(\cat C)$, but in order to extend it to morphisms, we need to assume $\action$ is an oplax monoidal $\cat M$-action. In fact if $(m,f): c \to d$ and $(n, g) : c' \to d'$ are morphisms, we can define (denoting by $\ctimes_\action$ the candidate monoidal product on $\Para_\action(\cat C)$):
\vspace{-1.1ex}
\begin{equation}
\label{eq:para-product}
	f \ctimes_\action g := (m \mtimes n, \ (m \mtimes n) \action (c \ctimes c') \nlongto{\iota_{m,n,c,c'}} (m \action c) \ctimes (n \action c') \nlongto{f \ctimes g} d \ctimes d').
\end{equation}
Ultimately, bifunctoriality follows from the braidedness of $\cat M$, the bifunctoriality of $\ctimes$, and the monoidality of $\mu^\action$ (see Diagram~\ref{diag:mu-monoidal}).
Conversely, given a monoidal product on $\Para_\action(\cat C)$, we can recover $\iota_{m, n, c, c'}$ by setting $f= (m, 1_{m \action c})$ and $g=(n, 1_{n \action c'})$.
The same can be shown for $\Copara_\action(\cat C)$.

\begin{remark}
\label{rmk:mon-of-optics}
	Let $(\cat C, \action)$ and $(\cat D, \actionn)$ be $\cat M$-actegories.
	In \cite{braithwaite2021fibre} it is shown how $\Optic_{\action,\actionn}$ arises as the (local discretization of) the pullback of
	\begin{diagram}[sep=4ex]
		\Para_\action(\cat C)^\coop \arrow{r}{P} & (\deloop \cat M)^\op & \arrow[swap]{l}{C} \Copara_\actionn(\cat D),
	\end{diagram}
	where $P$ and $C$ respectively project out the parameter and coparameter of a given (co)parametric morphism.
	Therefore, in case $\cat C$ and $\cat D$ are actually oplax monoidal actegories, the induced monoidal structure on $\Para_\action(\cat C)$ and $\Copara_\actionn(\cat D)$ would automatically yield a monoidal structure on $\Optic_{\action,\actionn}$, defined `componentwise' on the pullback.
\end{remark}

\begin{remark}
	Observe the product defined in Equation~\ref{eq:para-product} yields both the parallel product and external choice mentioned in Section~\ref{subsec:actegories-as-para}. Indeed, when $\cat C$ is in the situation outlined in Example~\ref{ex:cart-cocart-oplax-monact}, we get parallel product as $\times_\times$ and external choice as $+_\times$.
\end{remark}




\subsection{Algebroidal actegories}
\label{subsec:algebroidal-actegories}
Monoidal actegories are thus actions in $\MonCat$, but nothing forbids going the other way: instead of expressing the $\cat M$-action on $\cat C$ internally to $\MonCat$, we might express the monoidal structure internally to $\Act{\cat M}$.

Monoidal objects in categories of modules have been traditionally called `algebras', thus we are going to call their categorified version `algebroidal actegories'.%
\footnote{Quantum algebra literature seems to use `algebra' anyway for this kind of objects, like in~\cite{ben2010integral}.}
In order to talk about these `algebroidal structures', we have to choose a monoidal structure for in $\Act{\cat M}$.
As proven in Remark~\ref{rmk:act-fibrewise-cartesian}, this category is cartesian monoidal, but that is not the only useful monoidal product around.
We have proven in Proposition~\ref{prop:mirroring-biact} that when $\cat M$ is braided, every left $\cat M$-actegory is canonically a $\cat M$-biactegory therefore we can restrict to $\Act{\cat M}$ the monoidal structure of $\Biact{\cat M}$ given by tensor product.

We start by analyzing the notion of algebroidal actegory we get from the tensor product of actegories:

\subsubsection{Definition.}
\label{def:balanced-algebroidal-cat}
\begin{resdefinition}[Balanced algebroidal $\cat M$-actegory]{balalgebroidalact}
	A \textbf{balanced algebroidal $\cat M$-actegory} is a pseudomonoid in $(\Act{\cat M}, \cat M, \tensor)$, amounting to a left $\cat M$-actegory $(\cat C, \action)$ together with $\cat M$-linear functors
	\begin{eqalign}
		I &: \cat M \longto \cat C,\\
		\ctimes &: \cat C \tensor \cat C \longto \cat C
	\end{eqalign}
	and $\cat M$-linear natural transformations $\lambda, \rho, \alpha$ satisfying the usual axioms for unitors and associators in pseudomonoids.
\end{resdefinition}

By appealing to Lemma~\ref{lemma:strictification-functor}, most of the data pertaining $I$ (i.e.~its left and right lineators) can be safely discarded by agreeing that $I(m) = m \action I(j)$, leaving just the choice of $I(j) : \cat C$ as data, that from now on we simply denote by $i$.

By Proposition~\ref{prop:tens-is-pseudocoeq} and Corollary~\ref{cor:biactegories-are-moncat}, the product map is equivalently given by a \emph{balanced} (and still linear) functor out of the cartesian product:
\begin{diagram}
	{\cat C \times \cat C} \\
	{\cat C \tensor \cat C} & {\cat C}
	\arrow["{\exists !\,\ctimes}", dashed, from=1-1, to=2-2]
	\arrow["Q"', from=1-1, to=2-1]
	\arrow["\ctimes"', from=2-1, to=2-2]
\end{diagram}
Therefore a balanced $\cat M$-algebroidal structure amounts to a monoidal structure $(i, \ctimes, \lambda, \rho, \alpha)$ on $\cat C$, plus an equilibrator (Definition~\ref{def:balanced}) for $\ctimes$:
\begin{equation}
	\varepsilon_{m,c,d} : (m \action c) \ctimes d \isolongto c \ctimes (m \action d);
\end{equation}
and a left lineator (Definition~\ref{def:linear-functor}) for $\ctimes$:
\begin{eqalign}
	\ell_{m,c,d} &: m \action (c \ctimes d) \isolongto (m \action c) \ctimes d.
\end{eqalign}
We spelled out the coherence laws obeyed by these in Appendix~\ref{appendix:defs}.

\begin{proposition}
\label{prop:balance-algebra-is-monoidal-actegory}
	For a braided monoidal category $\cat M$, every balanced algebroidal $\cat M$-actegory can be presented as a (strong) monoidal $\cat M$-actegory and vice versa.
\end{proposition}
\begin{proof}
	Given a balanced algebroidal $\cat M$-actegory, we get the data needed to define a monoidal $\cat M$-actegory by setting $\chi := \ell$ and $\kappa := \ell \comp \varepsilon$.
	Vice versa, we set $\ell := \chi$ and $\varepsilon = \chi^{-1} \comp \kappa$.
	Regarding coherence, it is a bit hard to eyeball exactly the correspondence between the two set of diagrams, since those for monoidal actegories are naturally spelled in terms of $\iota$ instead of $\kappa$ and $\chi$.
	In broad terms, Diagram~\eqref{diag:monact-comp-ass} corresponds to Diagram~\eqref{diag:alpha-left-lin} and Diagram~\eqref{diag:alpha-balanced}, Diagram~\eqref{diag:monact-comp-unit} to Diagram~\eqref{diag:lambda-rho-lin}, Diagram~\eqref{diag:eta-monoidal} to Diagram~\eqref{diag:bal-algebr-left-lin-unit} and Diagram~\eqref{diag:equilib-unit}, and finally Diagram~\eqref{diag:mu-monoidal} to Diagram~\eqref{diag:bal-algebr-left-lin-mult} and Diagram~\eqref{diag:equilib-ass}.
\end{proof}

\begin{remark}
	Even though balanced $\cat M$-algebroidal actegories `trivialize' to monoidal actegories, if we let Definition~\ref{def:balanced-algebroidal-cat} take place in \emph{bona fide} $\cat M$-biactegories with their tensor product, we obtain something genuinely different.
	Indeed, the equivalence we just sketched hinges crucially on $\cat C$ being used as a `mirrored' biactegory, i.e.~one for which left and right action only differ by a braiding twist in the multiplicator (Proposition~\ref{prop:mirroring-biact}).
	When $\cat C$ doesn't have this property, a morphism like $\kappa$, that \emph{changes the order of symbols} in an expression like $m \action (c \ctimes d)$, does not emerge anymore.
	Therefore a generic pseudomonoid object in $(\Biact{\cat M}, \cat M, \tensor)$, even when $\cat M$ is braided, is not `just' a `monoidal $\cat M$-biactegory' (i.e.~Definition~\ref{def:biactegories} internalized to $\Biact{\cat M}$, as we did for Definition~\ref{def:left-actegory} to get Definition~\ref{def:monoidal-actegory}), whereas the converse is true: every monoidal $\cat M$-biactegory defines an $\cat M$-algebroidal biactegory (by forgetting $\kappa$).
\end{remark}

We now turn to the second possible definition of algebroidal actegory, using the cartesian monoidal structure on $\Act{\cat M}$.

\subsubsection{Definition.}
\label{def:distributive-algebroidal-act}
\begin{resdefinition}{distalgebract}[Distributive algebroidal $\cat M$-actegory]
	A \textbf{distributive algebroidal $\cat M$-actegory} is a pseudomonoid in $(\Act{\cat M}, 1, \times)$, i.e.~a left $\cat M$-actegory $(\cat C, \action)$ together with $\cat M$-linear functors:
	\begin{eqalign}
		o &: 1 \longto \cat C,\\
		\cplus &: \cat C \times \cat C \longto \cat C
	\end{eqalign}
and $\cat M$-linear natural transformations $\lambda, \rho, \alpha$ satisfying the usual axioms for unitors and associators in pseudomonoids described in Appendix~\ref{appendix:defs}.
\end{resdefinition}

\begin{remark}
  Explicitly, the lineators of $o$ and $\cplus$ are called the \emph{absorber} and \emph{distributor}, respectively, componentwise unpacking to
	\begin{eqalign}
		\gamma_m &: m \action o \isolongto o,\\
		\delta_{m,c,d} &: m \action (c \cplus d) \isolongto (m \action c) \cplus (m \action d)
	\end{eqalign}
  We note that Diagram~\eqref{diag:absorber-laws} shows $\gamma_j = \eta_o^{-1}$.
\end{remark}

We are not aware of this definition being in the literature already, except for the case in which $\action$ descends from a monoidal structure $(i, \ctimes)$ already present on $\cat C$, i.e.~$\cat C$ is an actegory of the kind of Example~\ref{ex:moncat-self-actions}.
In that case, the above definition captures a fragment of that of Laplaza~\cite{laplaza1972coherence}, including coherences.

\begin{example}
	The canonical $\N$-action on a monoidal category $\cat M$, defined in Example~\ref{ex:natural-numbers-act}, is distributive algebroidal whenever $\cat M$ is symmetric.
	In fact, in that circumstance, $j^{\otimes n} \iso j$ and $(m \otimes m')^{\otimes n} \iso m^{\otimes n} \otimes {m'}^{\otimes n}$ for any $m: \cat M$ and $n : \N$.
\end{example}

\begin{example}
\label{ex:cocartesian-dist-actegory}
	In \cite{cockett2009logic,alcolei2014concurrent}, they consider the structure of a category being acted upon in two adjoint ways, that is, a category $\cat C$ receiving an $\cat M$-action $\actionn$ and an $\cat M^\op$-action $\action$ such that for each $m:\cat M$, $m \actionn - \adj m \action -$.
	This happens, for example, when $\cat C = \cat M$ is cartesian closed in which case $\actionn = \otimes$ and $\action = [-,=]$.
	Then if $\cat C$ has coproducts, $m \actionn -$ preserves them by virtue of being left adjoint, therefore $(\cat C, \action, 0, +)$ is a distributive algebroidal $\cat M$-actegory.
\end{example}

\begin{example}
  Consider the category $\Vec_K$ of finite-dimensional vector spaces over a field $K$.
  This is a distributive algebroidal $\Vec_K$-actegory where the acting monoidal structure is given by the tensor product $(\otimes, K)$, and the underlying product structure is given by the direct product $(\oplus, K^0)$.
  Unpacked concretely, the absorber and the distributor reduce to the well-known isomorphisms in vector spaces
  \begin{align*}
               U \otimes K &\isolongto U,\\
    U \otimes (V \oplus W) &\isolongto (U \otimes V) \oplus (U \otimes W)
  \end{align*}
\end{example}

\begin{proposition}[From distributive algebroidal to monoidal]
  Let $\cat M$ be a symmetric monoidal category suitably equipped with a commutative comonoid structure on every object (also known as a CD-category, see \cite[Def. 2.2.]{Cho_2019}).
  Then any distributive algebroidal actegory $(\cat C, \action)$ (with $\cat M$-linear functors $o$ and $\boxtimes$) can be turned into an oplax monoidal actegory whose interchanger is defined as the following composite:

  \begin{align*}
    & (m \otimes m') \action (a \boxtimes a')\\
    \xrightarrow{\delta_{m \otimes m', a, a'}} & ((m \otimes m') \action a) \boxtimes ((m \otimes m') \action a')\\
    \xrightarrow{((m \otimes \epsilon_{m'}) \action a) \boxtimes ((\epsilon_{m} \otimes m') \action a')} & ((m \otimes j) \action a) \boxtimes ((j \otimes m') \action a')\\
    \xrightarrow{(\rho_{m} \action a) \boxtimes (\lambda_{m'} \action a')} & (m \action a) \boxtimes (m' \action a')
  \end{align*}
  and where $\epsilon_{m}$ and $\epsilon_{m'}$ denote the counit morphism on $m$ and $m'$ of $\cat M$, respectively.
  It's routine to show that this is natural.
\end{proposition}


Most importantly for our interests, this kind of distributivity is used in the optics literature to define a class of mixed optics called `affine traversals' \cite{roman2020profunctor, categoricalupdate}.

In fact, suppose $\cat C$ is a distributive algebroidal $\cat M$-actegory. We can form (at least) two kinds of optics from this data.
The first kind is `homogeneous' optics using up the data of the $\cat M$-action on $\cat C$, namely $\Optic_{\action, \action}$. The second kind is obtained using the monoidal structure on $\cat C$ instead, yielding $\Optic_{\cplus, \cplus}$. More often than not, $\cplus$ is a cocartesian monoidal structure, so these latter optics can be considered `prism-like'.

Affine traversals come up when one tries to compose optics from $\Optic_{\action, \action}$ with the prism-like optics $\Optic_{\cplus, \cplus}$. Such `chimeric' optics are usually obtained by combining the actegorical data used for each of the parts making up the chimera to yield a `minimal' composite actegory (we have seen this in Remark~\ref{rmk:prod-coprod-cartesian-subcats}).
In the case at hand, this means combining $\action$ and $\cplus$ in some way, and if they form a distributive algebroidal $\cat M$-actegory then we have the data required to do so:

\subsubsection{Proposition.}
\label{prop:waff-prod}
\begin{resproposition}[Aff product]{waffproduct}
	Let $(\cat C, \action, o, \cplus)$ be a distributive algebroidal $\cat M$-actegory.
	Then the product category $\cat C \times \cat M$ can be equipped with a monoidal product $\waff$, whose unit is $(o,j)$ and whose product law is so defined:
	\begin{equation}
		(c, m) \waff (d,n) := (c \cplus (m \action d), m \mtimes n)
	\end{equation}
	Following the terminology in \cite[Prop. 3.25.]{categoricalupdate}) we call this the \emph{Aff product}, and we denote such monoidal category by $\Waff(\cat C, \cat M)$.
\end{resproposition}
\begin{proof}
	See Appendix~\ref{appendix:proofs}.
\end{proof}

\begin{remark}
\label{rmk:waff-is-aff-comp}
	This product might seem puzzling, but it is actually very natural.
	One has to think about pairs $(c, m)$ as standing for the \emph{affine endofunctors} $c \cplus (m \action -)$ on $\cat C$.
	Then $\waff$ is abstracting the composition product of such functors, meaning that $(c, m) \waff (d, n)$ are the coefficients of $c \cplus (m \action -)\ \comp\ d \cplus (n \action -)$.
	Notice, however, that the inclusion $(c,m) \mapsto c \cplus (m \action -)$ is \emph{not} full, that is, we can't say in general that every natural transformation $c \cplus (m \action -)\ \twoto\ d \cplus (n \action -)$ arises from a pair of morphisms $c \to d$ in $\cat C$ and $m \to n$ in $\cat M$.
\end{remark}

\begin{example}
	A dual situation to that of Example~\ref{ex:cocartesian-dist-actegory} arises when $\cat C$ is cartesian closed, in which case $(\cat C, [-,=], 1, \times)$ is a distributive algebroidal $\cat C^\op$-actegory, where $[-,=] : \cat C^\op \times \cat C \to \cat C$ denotes the internal hom.
	The Aff product in this case is given by
	\begin{equation*}
		(c, m) \waff (d, n) := (c \times [m, d], m \times n)
	\end{equation*}
	and the associated class of optics is known as \emph{glasses} \cite[§3.3.2]{categoricalupdate}.
\end{example}

Naturally, affine endofunctors on $\cat C$ act on $\cat C$:
\begin{equation}
	(c, m) \ostar d = c \cplus (m \action d).
\end{equation}
The unitor and multiplicator of this action are routine to define (see Corollary~\ref{cor:waff-action}).

\begin{remark}
	One might notice we already have the means to obtain affine traversals (or, in general, of hybrid optics) using the coproduct of actions described in Remark~\ref{rmk:prod-coprod-cartesian-subcats}.
	Indeed, there is a `projection functor' $\pi : \cat C + \cat M \to \Waff(\cat C, \cat M)$ defined as
	\begin{equation}
		c_1m_1\cdots c_km_k \mapsto (c_1 \cplus (m_1 \action c_2) \cplus \cdots \cplus (m_1 \mtimes \cdots \mtimes m_{k-1} \action c_k),\, m_1 \mtimes \cdots \mtimes m_k).
	\end{equation}
	This is a morphism in $\Actt^\cart(\cat C)$, meaning
	\begin{equation}
		\pi(c_1m_1\cdots c_km_k) \ostar -\iso (c_1m_1\cdots c_km_k) \, \coprodaction_\cart -.
	\end{equation}
	Combining this observation with Remark~\ref{rmk:waff-is-aff-comp} above, we realize \emph{$\Waff(\cat C, \cat M)$ is a direct presentation of the replete image\footnotemark\ of $\curr(\coprodaction_\cart)$ in $[\cat C, \cat C]$}; that is, it captures the essential way $\cat C+\cat M$ acts on $\cat C$.%
	\footnotetext{The \emph{replete image} of a functor $F: \cat A \to \cat B$ is the universal pseudomonic functor $F$ factors through. In other words, it is the smallest subcategory of $\cat B$ closed under isomorphisms that contains the image of $F$.}
\end{remark}

The difference between the action of $\Waff(\cat C, \cat M)$ and that of $\cat C+\cat M$ is a guiding example in \cite[§6.4]{roman2020profunctor}, where \emph{clear optics} are studied, i.e.~optics associated to actions whose currying is pseudomonic.
R\'oman notes that the $\Optic$ construction is invariant under repletion, that is, forming categories of optics only depends on the data contained in the replete image of a given action:
\begin{diagram}[column sep=5ex]
	{\Actt_\cart(\cat C)} & {\Actt_\cart(\cat C)} \\
	& \Cat
	\arrow["\repl", from=1-1, to=1-2]
	\arrow["\Optic"', from=1-1, to=2-2]
	\arrow["\Optic", dashed, from=1-2, to=2-2]
\end{diagram}
Therefore any given class of optics can be presented by multiple, inequivalent actions (indeed, $\ostar \not\iso \coprodaction_\cart$), as long as they all have the same replete image. In practice, one would like to use the `least redundant' among these actions, that is, the one for which the map $\cat M \to [\cat C, \cat C]$ is pseudomonic, which is indeed what the repletion gives.
Thus comparing the Aff action and the coproduct action, we see that the albeit the latter is more general, the former directly yields such maximally efficient action, by exploiting the structure of the situation.

This also happens when the structure of biactegory is around. Suppose in fact $\cat C$ is an $(\cat M, \cat N)$-actegory. Then the repletion of $\curr(\coprodaction_\cart)$ coincides with $\curr\langle\action, \actionn\rangle$, since the bimodulator provides a way to commute the actions of scalars in $\cat M$ past the actions of scalars in $\cat N$. In other words, the functor $\Gamma : \cat M \times \cat N \to \cat M + \cat N$ described in Equation~\eqref{eq:gamma-morphism} becomes a morphism in $\Actt^\cart(\cat C)$, specifically the repletion projection..

The common theme here is that both distributive algebroidal actegories and biactegories describe structure to `sequentially compose' (recall the beginning of Section~\ref{sec:composition}) two actegories.
One might then speculate that such distributive laws provide general ways to reduce the coproduct action: distributive algebroidal structure allows to reduce $\cat C+ \cat M$ to $\Waff(\cat C, \cat M)$ whereas biactegorical structure reduces $\cat N+\cat M$ to to $\cat N \times \cat M$.

\subsection{Distributive laws}
As remarked in Section~\ref{sec:biactegories}, bimodulators are akin to distributive laws for the endofunctors induced by scalars in $\cat M$ and $\cat N$ respectively.
Proposition~\ref{prop:waff-prod} shows, in some cases, different kind of distributive laws might be around, thus yielding a different `sequential composition' of the actions.

All in all, we have the following classification:

\def\arraystretch{2}
\begin{table}[H]
	\centering
	\resizebox{\columnwidth}{!}{
		\begin{tabularx}{1.2\textwidth}{c|>{\centering\arraybackslash}X|>{\centering\arraybackslash}X}
			&
			monoidal structure $(o, \cplus)$
			&
			$\cat N$-actegorical structure $\actionn$
			\\[1.5ex]
			\hline
			\begin{tabular}{@{}c@{}}
				$(\cat C, i, \ctimes)$ in\\[-2.5ex]
				$(\MonCat, 1, \times)$
			\end{tabular}
			&
			\begin{tabular}{@{}c@{}}
				\textbf{braided monoidal category}\\[-2ex]
				$(a \cplus c) \ctimes (b \cplus d) \iso (a \ctimes b) \cplus (c \ctimes d)$
			\end{tabular}
			&
			\begin{tabular}{@{}c@{}}
				\textbf{monoidal $\cat N$-actegory}\\[-2ex]
				$\cat N$ braided\\[-2ex]
				$(n \actionn c) \ctimes (n' \actionn d) \iso (n \ntimes n') \actionn (c \ctimes d)$
			\end{tabular}
			\\[1.5ex]
			\hline
			\begin{tabular}{@{}c@{}}
				$(\cat C, \action)$ in\\[-2.5ex]
			  $(\Act{\cat M}, 1, \times)$
			\end{tabular}
			&
			\begin{tabular}{@{}c@{}}
				\textbf{distributive algebroidal $\cat M$-actegory}\\[-2ex]
				$m \action o \iso o$\\[-2ex]
				$m \action (c \cplus d) \iso (m \action c) \cplus (m \action d)$
			\end{tabular}
			&
			\begin{tabular}{@{}c@{}}
				\textbf{$(\cat M, \cat N^{\rev})$-biactegory}\\[-2ex]
				$m \action (n \actionn c) \iso n \actionn (m \action c)$
			\end{tabular}
			\\[1.5ex]
			\hline
			\begin{tabular}{@{}c@{}}
				$(\cat C, \action)$ in\\[-2.5ex]
				$(\Act{\cat M}, \cat M, \tensor)$\\
			\end{tabular}
			&
			\begin{tabular}{@{}c@{}}
				\textbf{balanced algebroidal $\cat M$-actegory}\\[-2ex]
				$\cat M$ braided\\[-2ex]
				$(m \action c) \cplus (m' \action d) \iso (m \mtimes m') \action (c \cplus d)$
			\end{tabular}
			&
			\begin{tabular}{@{}c@{}}
				\textbf{$(\cat M, \cat N^{\rev})$-biactegory}\\[-2ex]
				$\cat M$ braided\\[-2ex]
				$m \action (n \actionn c) \iso n \actionn (m \action c)$
			\end{tabular}
		\end{tabularx}
	}
	\caption{Each cell contains the result of equipping the object in the corresponding row with the structure of the corresponding column. In the above, $(\cat M, j, \mtimes)$ and $(\cat N, e, \ntimes)$ are monoidal categories.}
	\label{table:mon-act-distributive-laws}
\end{table}

From this point of view we recover all the `hybrid' structures we contemplated so far, including an additional cameo of braided categories whose equivalence to `monoidal monoidal categories' is well-known \cite{joyal1993braided}.

Perhaps the only entries of the table which are not completely obvious are the ones pertaining biactegories.
In fact we didn't present $(\cat M, \cat N)$-biactegories as `$\cat N$-actions in $\Act{\cat M}$', but as bialgebras for the $\cat M \times -$ and $\cat N \times -$ (for sake of simplicity, let's temporarily ignore the fact we should technically consider $\cat N$ acting on the right).

This gives us the opportunity to brush up something we briefly mentioned at the beginning of the paper, in Definition~\ref{def:action-in-actegory}, where we defined the structure of `action of a monoid' in any category $\cat C$ which is itself subject to an action of a monoidal category $\cat M$.

It is easy to convince oneself that this statement stays true if we categorify one more time.
Hence to talk about `$\cat N$-actions in $\Act{\cat M}$' it is enough for $\Cat$ to act on $\Act{\cat M}$, so that $\cat N$ (which is indeed a monoid in $\Cat$) can act on objects of $\Act{\cat M}$.
One can prove (the characterization theorems in the next section are useful in this regard) that there are two such actions:
\begin{eqalign}
\label{eq:cat-actions-on-Mact}
	\cat D \ltimes (\cat C, \action) &:= (\cat D \times \cat C, \cat D \times \action),\\
	\cat D \lhd (\cat C, \action) &:= \cat M[\cat D] \tensor \cat C.
\end{eqalign}
Therefore an $\cat N$-action in $\Act{\cat M}$ can be interpreted in two different ways:
\begin{enumerate}
	\item Using $\ltimes$, we get an $\cat M$-actegory $(\cat C, \action)$ together with an $\cat M$-linear morphism $\actionn : \cat N \times \cat C \to \cat C$, and it is indeed this linear structure that gives rise to the bimodulator morphism $n \actionn (m \action c) \iso m \action (n \actionn c)$.
	\item Using $\lhd$, we get an $\cat M$-actegory $(\cat C, \action)$ together with an $\cat M$-linear morphism $\actionn : \cat M[\cat N] \tensor \cat C \to \cat C$. But we can easily prove that $ \cat M[\cat N] \tensor \cat C \equi \cat N \times \cat C$, so this structure coincides with the one above.
\end{enumerate}

\subsection{Braided and symmetric monoidal actegories}
\label{subsec:braided-monact}
So far, we have been focusing on the case where a braided monoidal category $(\cat M, j, \mtimes, \beta)$ acts on a monoidal category $(\cat C, i, \ctimes)$.
What happens if instead $\cat M$ is symmetric? And what if $\cat C$ is braided or symmetric?

We frame this question in the same way we did before, i.e.~by formulating the definition of action not in $\MonCat$ anymore, but internally to $\BrMonCat$ and $\SymMonCat$.
Conversely we formulate the notion of braided and symmetric pseudomonoid inside $\Act{\cat M}$, with either choice of monoidal product.

In the first case, recall that $\BrMonCat \iso \PsdMon(\MonCat)$, $\SymMonCat \iso \PsdMon(\BrMonCat)$ (and $\PsdMon(\SymMonCat) \iso \SymMonCat$), defining a column in the periodic table of higher categories \cite{baezdolan1995}.
This means that actions on braided and symmetric monoidal categories, by definition, come from symmetric monoidal categories.
Thus we have:

\begin{definition}[Braided monoidal actegory]
\label{def:braided-monoidal-actegory}
	Let $(\cat M, j, \mtimes, \sigma)$ be a symmetric monoidal category.
	A \emph{braided monoidal left $\cat M$-actegory} is a braided monoidal category $(\cat C, i, \ctimes, \beta)$ equipped with a braided monoidal functor $\action : \cat M \times \cat C \to \cat C$ and two monoidal natural transformations $\mu$ and $\eta$ defined analogously to Definition~\ref{def:left-actegory}.
	The strong monoidal structure on $\action$ is given as in Definition~\ref{def:monoidal-actegory}, and additionally the following axiom is satisfied:
	\begin{diagram}
		{(m \mtimes n) \action (c \ctimes d)} & {(m \action c) \ctimes (n \action d)} \\
		{(n \mtimes m) \action (d \ctimes c)} & {(n \action d) \ctimes (m \action c)}
		\arrow["{\sigma_{m,n} \action \beta_{c,d}}"', from=1-1, to=2-1]
		\arrow["{\iota_{m,n,c,d}}", from=1-1, to=1-2]
		\arrow["{\iota_{n,m,d,c}}"', from=2-1, to=2-2]
		\arrow["{\beta_{m \action c, n \action d}}", from=1-2, to=2-2]
	\end{diagram}
\end{definition}

\begin{remark}
\label{rmk:braided-actegory-strengths}
	In terms of $\kappa$ and $\chi$ (Remark~\ref{rmk:monoidal-actegory-strengths}), the last axiom splits in two:
	\begin{diagram}
	\label{diag:braided-actegory-axioms}
		{m \action (c \ctimes d)} & {(m \action c) \ctimes d} &[-3ex]&[-3ex] {m \action (c \ctimes d)} & {c \ctimes (m \action d)} \\
		{m \action (d \ctimes c)} & {d \ctimes (m \action c)} && {m \action (d \ctimes c)} & {(m \action d) \ctimes c}
		\arrow["{m \action \beta_{c,d}}"', from=1-1, to=2-1]
		\arrow["{\chi_{m,c,d}}", from=1-1, to=1-2]
		\arrow["{\kappa_{m,d,c}}"', from=2-1, to=2-2]
		\arrow["{\beta_{m \action c, d}}", from=1-2, to=2-2]
		\arrow["{\beta_{c, m\action d}}", from=1-5, to=2-5]
		\arrow["{m \action \beta_{c,d}}"', from=1-4, to=2-4]
		\arrow["{\kappa_{m,c,d}}", from=1-4, to=1-5]
		\arrow["{\chi_{m,d,c}}"', from=2-4, to=2-5]
	\end{diagram}
\end{remark}

\begin{remark}[Symmetric monoidal actegory]
\label{rmk:symmetric-mnoidal-actegory}
	Since `symmetry' is a \emph{property} of braided structures, and braided monoidal functors automatically preserve it, `\emph{symmetric monoidal actegories}' amount to braided monoidal actegories such that $\beta$ (the braiding on $\cat C$) happens to be symmetric.
\end{remark}

We can still wonder what happens if we take an algebroidal actegory and ask the pseudomonoidal structure to be braided or symmetric (as defined in \cite[§3]{mccrudden2000balanced}).%
\footnotetext{To even talk about this, one should prove the monoidal 2-category where the pseudomonoid leaves is itself symmetric. In our case, this is either $(\Act{\cat M}, 1, \times)$ (which is obviously symmetric) and $(\Act{\cat M}, \cat M, \tensor)$ (which is symmetric when $\cat M$ is, although less obviously so).}

For a balanced algebroidal actegory $(\cat C, \action, i, \ctimes)$ (Definition~\ref{def:balanced-algebroidal-cat}), a braiding in this sense amounts to a braiding $\beta$ on the underlying monoidal category $(\cat C, i, \ctimes)$ which happens to also be a linear and balanced transformation in $\Act{\cat M}$:
\begin{diagram}[sep=4ex]
	{\cat C \times \cat C} & {\cat C \times \cat C} && {\cat C \times \cat C} & {\cat C \times \cat C} \\
	{\cat C} & {\cat C} && {\cat C} & {\cat C}
	\arrow["\ctimes"', from=1-1, to=2-1]
	\arrow["\ctimes", from=1-2, to=2-2]
	\arrow["\swap", from=1-1, to=1-2]
	\arrow[Rightarrow, no head, from=2-1, to=2-2]
	\arrow["\beta", shift right=2, shorten <=16pt, shorten >=16pt, Rightarrow, from=2-1, to=1-2]
	\arrow["\swap"', from=1-5, to=1-4]
	\arrow[Rightarrow, no head, from=2-4, to=2-5]
	\arrow["\ctimes", from=1-5, to=2-5]
	\arrow["\ctimes"', from=1-4, to=2-4]
	\arrow["\beta", shift right=2, shorten <=15pt, shorten >=15pt, Rightarrow, from=2-4, to=1-5]
\end{diagram}
An interesting fact to observe is that $\varepsilon$, the equilibrator of $\ctimes$, becomes part of the left linear structure of $\swap \comp \ctimes$, so that linearity for $\beta$ looks like this:
\begin{diagram}
	{m \action (c \ctimes d)} && {(m \action c) \ctimes d} \\
	{m \action (d \ctimes c)} & {(m \action d) \ctimes c} & {d \ctimes (m \action c)} \\
	{m \action (c \ctimes d)} & {(m \action c) \ctimes d} & {c \ctimes (m \action d)} \\
	{m \action (d \ctimes c)} && {(m \action d) \ctimes c}
	\arrow["{\ell_{m,c,d}}", from=1-1, to=1-3]
	\arrow["{m \action \beta_{c,d}}"', from=1-1, to=2-1]
	\arrow["{\ell_{m,d,c}}"', from=2-1, to=2-2]
	\arrow["{\varepsilon_{m,d,c}}"', from=2-2, to=2-3]
	\arrow["{\beta_{m\action c, d}}", from=1-3, to=2-3]
	\arrow["{\varepsilon_{m,c,d}}", from=3-2, to=3-3]
	\arrow["{\ell_{m,c,d}}", from=3-1, to=3-2]
	\arrow["{\beta_{c, m \action d}}", from=3-3, to=4-3]
	\arrow["{\ell_{m,d,c}}"', from=4-1, to=4-3]
	\arrow["{m \action \beta_{c,d}}"', from=3-1, to=4-1]
\end{diagram}
Recall now the proof of Proposition~\ref{prop:balance-algebra-is-monoidal-actegory}, where we have shown monoidal actegories and balanced algebroidal categories coincide. We set $\chi = \ell$ and $\kappa = \ell \comp\varepsilon$, a substitution under which Diagram~\eqref{diag:braided-actegory-axioms} is transmuted to the one just drawn.
Vice versa, given $\chi$ and $\kappa$ we set $\ell = \chi$ and $\varepsilon = \chi^{-1} \comp \kappa$, and the transformation of diagrams reverses.
Therefore, we conclude \textbf{the equivalence of balanced algebroidal and monoidal actegories extends to the braided (and thus, the symmetric) case}.

When it comes to distributive algebroidal actegories (Definition~\ref{def:distributive-algebroidal-act}), a braiding for $(\cat C, \action, o, \cplus)$ is still a braiding $\beta$ for the underlying monoidal category satisfying an additional linearity axiom, namely
\begin{diagram}
	{m \action (c \cplus d)} & {m \action (d \cplus c)} \\
	{(m \action c) \cplus (m \action d)} & {(m \action d) \cplus (m \action c)}
	\arrow["{m \action \beta_{c,d}}", from=1-1, to=1-2]
	\arrow["{\beta_{m \action c, m\action d}}", from=2-1, to=2-2]
	\arrow["{\delta_{m,d,c}}", from=1-2, to=2-2]
	\arrow["{\delta_{m,c,d}}"', from=1-1, to=2-1]
\end{diagram}


	\subsection{The classifying objects of actions}
\label{subsec:classifying}
We now have a pretty clear picture of the possible ways monoidal and actegorical structures can interact.
In particular, we have seen how this interaction amounts to specific `distributive laws' (Table~\ref{table:mon-act-distributive-laws}). One way to think of those laws is as constraints on the effect of the action on products.
Take for example a monoidal $\cat M$-actegory $(\cat C, i, \ctimes, \action)$.
We know that in that case
\begin{equation}
	m \action (c \ctimes d) \iso c \ctimes (m \action d).
\end{equation}
However, \emph{every element of $\cat C$ can be written as a monoidal product}, in at least two ways:
\begin{equation}
	c \ctimes i \iso c \iso i \ctimes c.
\end{equation}
As a consequence, $\action$ is completely determined by its effect on the monoidal unit $i$:
\begin{equation}
	m \action c \iso m \action (c \ctimes i) \iso c \ctimes (m \action i).
\end{equation}
In particular, $m \mapsto m \action i$ defines functor $- \action i : \cat M \to \cat C$ which is strong monoidal:
\begin{eqalign}
	j \action i &\iso i,\\
	(m \mtimes m') \action i &\iso (m \mtimes m') \action (i \ctimes i) \iso (m \action i) \ctimes (m' \action i).
\end{eqalign}
Thus one might speculate an equivalence between action on $\cat C$ and strong monoidal functors into $\cat C$.
After all, we already stumbled upon a similar fact in Example~\ref{ex:curried-action-terminal}, where we observed how actions on $\cat C$ are classified by $[\cat C, \cat C]$.

There is a problem though: while every action on $\cat C$ does indeed produce a monoidal functor into $\cat C$, the converse is not true. In fact given ${F: \cat M \to \cat C}$, and defining $m \action c := F(m)\ctimes c$, we do get a well defined action but when it comes to defining the mixed interchanger (i.e.~the monoidal structure of $\action$), we are stuck:
\definecolor{redwcontrast}{rgb}{1, 0.03, 0.0}
\begin{eqalign}
\label{eq:we-got-stuck}
	(m \mtimes m') \action (c \ctimes d) &= F(m \mtimes m') \ctimes c \ctimes d\\
	&\iso F(m) \ctimes \textcolor{redwcontrast}{F(m') \ctimes c} \ctimes d\\
	\textcolor{redwcontrast}{\textbf{?}} &\iso F(m) \ctimes \textcolor{redwcontrast}{c \ctimes F(m')} \ctimes d\\
	&= (m \action c) \ctimes (m' \action d).
\end{eqalign}
Note the missing isomorphism is something of the form $F(m') \ctimes c \isoto c \ctimes F(m')$.

Thus, the takeaway is that monoidal actions on $\cat C$ \emph{encode more information} than a monoidal functor into $\cat C$ can provide, namely that of a choice of `braiding' $F(m) \ctimes - \isoto - \ctimes F(m)$ for each scalar $m$.

To fix this problem, we are going to introduce two objects, the \emph{Drinfeld center} $\drinfeld(\cat C)$ (Definition~\ref{def:center}) and the \emph{symmetric center} $\Sigma(\cat C)$ (Definition~\ref{def:symm-center}), that play the role of $[\cat C, \cat C]$ for, respectively, monoidal (Definition~\ref{def:monoidal-actegory}) and braided actions (Definition~\ref{def:braided-monoidal-actegory}).

The main results of this section are classification results: they tell us that monoidal actions on a monoidal category $\cat C$ correspond to strong monoidal functors into $\drinfeld(\cat C)$ and braided actions on a braided category $\cat C$ correspond to braided functors into $\Sigma(\cat C)$, formalizing the idea we illustrated above.
When $\cat C$ is symmetric, such a correspondence has the simplest form: braided $\cat M$-actions on $\cat C$ correspond to braided functors $\cat M \to \cat C$.

\subsubsection{Classifying monoidal actions}
The \emph{Drinfeld center} construction is a `proof-relevant' version of the usual center construction $Z(M)$ of a monoid $M$.
In fact, the latter is the monoid of those elements $m \in M$ \emph{with the property that} $\forall n \in M, m\cdot n = n \cdot m$, while the first is a monoidal category of those elements $c : \cat C$ equipped \emph{with the structure of} a natural isomorphism $c \ctimes - \isotwoto - \ctimes c$.

Inasmuch as a monoid is commutative if the inclusion $Z(M) \into M$ admits a section, i.e.~if it is surjective; a braiding on a monoidal category $(\cat C, i, \ctimes)$ correspond to a braided section of the evident forgetful functor $\drinfeld(\cat C) \to \cat C$ \cite[Corollary XIII.4.4]{kassel1995quantum}.

The exact definition of $\drinfeld(\cat C)$ goes as follows:

\begin{definition}[{\cite[Definition 3]{joyalstreet1991}}]
	\label{def:center}
	Let $(\cat C, i, \ctimes)$ be a monoidal category.
	Its \emph{Drinfeld center} $\drinfeld(\cat C)$ is the braided monoidal category defined as follows.
	Objects are pairs ${(c : \cat C, \upsilon : c \ctimes - \isotwoto - \ctimes c)}$ such that for every $d, e : \cat C$ (up to associators):
	\begin{diagram}[sep=2ex]
		\label{eq:upsilon-coherence}
		{c \ctimes d \ctimes e} &&&& {d \ctimes e \ctimes c} \\
		\\
		&& {d \ctimes c \ctimes e}
		\arrow["{\upsilon_{d} \ctimes e}"', from=1-1, to=3-3]
		\arrow["{d \ctimes \upsilon_{e}}"', from=3-3, to=1-5]
		\arrow["{\upsilon_{d \ctimes e}}", from=1-1, to=1-5]
	\end{diagram}
	Morphisms $(c, \upsilon) \to (d, \tau)$ are arrows $f:c \to d$ in $\cat C$ such that for every $b : \cat C$:
	\begin{diagram}[sep=4ex]
	\label{diag:drinfeld-center-mors}
		{c \ctimes b} & {d \ctimes b} \\
		{b \ctimes c} & {b \ctimes y}
		\arrow["{f \ctimes b}", from=1-1, to=1-2]
		\arrow["{\tau_b}", from=1-2, to=2-2]
		\arrow["{\upsilon_b}"', from=1-1, to=2-1]
		\arrow["{b \ctimes f}"', from=2-1, to=2-2]
	\end{diagram}
	The monoidal product is given by
	\begin{equation}
		(c, \upsilon) \ctimes (d, \tau) := (c \ctimes d, (\upsilon \ctimes d) \comp (c \ctimes \tau)),
	\end{equation}
	with unit $(i, \lambda \comp \rho^{-1})$.
	Finally, the braiding is defined as $\beta_{(c,\upsilon),(d,\tau)} := \upsilon_d$.
\end{definition}

To give a functor $\cat M \to \drinfeld(\cat C)$, therefore, is to specify a functor in $\cat C$ together with a way to commute each $F(m)$ past any other object in $\cat C$, exactly what we need to get unstuck in Equation~\eqref{eq:we-got-stuck}.
In fact, we have the following:

\begin{theorem}[Monoidal actions classifier]
\label{th:classification-monoidal-act}
	$\drinfeld(\cat C)$ \emph{classifies monoidal actions on $\cat C$}, meaning $\drinfeld(\cat C)$ acts monoidally on $\cat C$ and every other monoidal action $\action$ of a braided monoidal category $\cat M$ factors uniquely through it:
	\begin{diagram}
		{\drinfeld(\cat C) \times \cat C} & {\cat C} \\
		{\cat M \times \cat C}
		\arrow["\action"', from=2-1, to=1-2]
		\arrow["\ctimes_1", from=1-1, to=1-2]
		\arrow["{\exists! F_\action \times \cat C}", dashed, from=2-1, to=1-1]
	\end{diagram}
	In other words, $\ctimes_1$ is terminal in the category $\Actt_{\MonCat}^\cart(\cat C)$ of monoidal actions on $\cat C$.
\end{theorem}

\begin{remark}
	In the above, $\Actt_{\MonCat}$ is the (2-Grothendieck construction of) indexed 2-category of monoidal actions, i.e.~the functor $\BrMonCat^\coop \to 2\Cat$ assigning to each braided monoidal category the 2-category of monoidal actegories over it.
	This is defined analogously to that in Proposition~\ref{prop:actegories-indexed-over-moncat}, included the notation $\Actt_{\MonCat}^\cart(\cat C)$ to denote the cartesian subcategory of $\Actt_{\MonCat}$ at $\cat C$, as defined in Remark~\ref{rmk:cartesian-factorization}.
\end{remark}

To swiftly prove Theorem~\ref{th:classification-monoidal-act}, we'll rely on the following:

\subsubsection{Lemma.}
\label{lemma:center-charact}
\begin{reslemma}[Characterization of the Drinfeld center]{centercharact}
	Consider the the canonical $\cat C$-biactegory $(\cat C, \ctimes, \ctimes)$ (as defined in Example~\ref{ex:moncat-self-biaction}) associated to $\cat C$, that we still denote by $\cat C$.
	Then
	\begin{equation}
		\drinfeld(\cat C) \equi [\cat C, \cat C]_{\Biact[\cat C]{\cat C}}
	\end{equation}
	where on the right we have the category of maps $\varphi : \cat C \to \cat C$ equipped with $\cat C$-linear structures
	\begin{eqalign}
		\label{eq:center-strengths}
		\ell_{c, d} &: \varphi(c) \ctimes d \isolongto \varphi(c \ctimes d),\\
		r_{c, d} &: c \ctimes \varphi(d) \isolongto \varphi(c \ctimes d).
	\end{eqalign}
	satisfying the laws stated in Definition~\ref{def:bilinear-functor}.
\end{reslemma}
\begin{proof}
	See Appendix~\ref{appendix:proofs}.
\end{proof}

The object $[\cat C, \cat C]_{\Biact[\cat C]{\cat C}}$ is indeed called \emph{center} in \cite[Definition 5.1]{ben2010integral}.

\begin{proof}[of Theorem~\ref{th:classification-monoidal-act}]
	By Lemma~\ref{lemma:center-charact}, we can replace $\drinfeld(\cat C)$ with $[\cat C, \cat C]_{\Biact[\cat C]{\cat C}}$.
	The latter acts on $\cat C$ by evaluation, so we denote the action as $\eval$.
	Moreover, to give an $\cat M$-action on $\cat C$ in $\MonCat$ is to give isomorphisms (see Remark~\ref{rmk:monoidal-actegory-strengths})
	\begin{eqalign}
		\chi^\action_{m,c,d} &: m \action (c \ctimes d) \isolongto (m \action c) \ctimes d,\\
		\kappa^\action_{m,c,d} &: m \action (c \ctimes d) \isolongto c \ctimes (m \action d).
	\end{eqalign}
	Hence $\eval$ factors $\action$ through the functor
	\begin{equation}
		F_\action(m) = (m \action -, {\chi^\action}^{-1}, {\kappa^\action}^{-1}).
	\end{equation}
	Such a factorization is manifestly unique.
\end{proof}

\begin{corollary}
	There is an equivalence
	\begin{equation}
		\BrMonCat \!\underset{\MonCat}{\downarrow}\! \drinfeld(\cat C) \equi \Actt_{\MonCat}^\cart(\cat C)
	\end{equation}
	where the 2-category on the left amounts to the subcategory of $\MonCat/\drinfeld(\cat C)$ given by braided monoidal categories whose braiding has been forgotten.
\end{corollary}
\begin{proof}
	Observe that the action $\eval$ of $[\cat C, \cat C]_{\Biact[\cat C]{\cat C}}$ translates back to the action $\ctimes_1$ of $\drinfeld(\cat C)$, which is simply monoidal product after forgetting the braiding:
	\begin{equation}
		(c, \upsilon) \ctimes_1 d := c \ctimes d.
	\end{equation}
	Running the proof of Lemma~\ref{lemma:center-charact} backwards, we thus get that the correspondence is given as follows.
	Given a monoidal functor $F : \cat M \to \drinfeld(\cat C)$, where $\cat M$ is braided, we get the $\cat M$-action
	\begin{equation}
		m \action^F c := F_1(m) \ctimes c,
	\end{equation}
	while given a monoidal $\cat M$-action $\action$ on $\cat C$, we get the monoidal functor $\cat M \to \drinfeld(\cat C)$ defined as
	\begin{equation}
		F_\action(m) := (m \action i, \chi^\action_{m,i}).
	\end{equation}
	That these are well-defined objects follows from the main theorem.
\end{proof}

\subsubsection{Classifying braided actions}
Now the same can be done for actions of symmetric monoidal categories on braided and symmetric monoidal categories but, again, the classifying object will be different.

\begin{definition}
\label{def:symm-center}
	Let $(\cat{C}, i, \ctimes, \beta)$ be a braided monoidal category.
	Its \emph{symmetric center} $\Sigma(\cat C)$ is the full subcategory of $\cat C$ determined by those objects ${c : \cat C}$ so that, for all ${d : \cat C}$, the following commutes:
	\begin{diagram}[sep=3ex]
		& {d \ctimes c} \\
		{c \ctimes d} && {c \ctimes d}
		\arrow["{\beta_{c,d}}", from=2-1, to=1-2]
		\arrow["{\beta_{d,c}}", from=1-2, to=2-3]
		\arrow[Rightarrow, no head, from=2-1, to=2-3]
	\end{diagram}
\end{definition}

\begin{theorem}[Braided actions classifier]
\label{th:classification-braided-act}
	Let $(\cat{C}, i, \ctimes, \beta)$ be a braided monoidal category.
	$\Sigma(\cat C)$ \emph{classifies braided actions on $\cat C$}, meaning $\Sigma(\cat C)$ acts braidely on $\cat C$ and every other braided action $\action$ of a symmetric monoidal category $\cat M$ factors uniquely through it:
	\begin{diagram}
		{\Sigma(\cat C) \times \cat C} & {\cat C} \\
		{\cat M \times \cat C}
		\arrow["\action"', from=2-1, to=1-2]
		\arrow["\ctimes", from=1-1, to=1-2]
		\arrow["{\exists! F_\action \times \cat C}", dashed, from=2-1, to=1-1]
	\end{diagram}
	In other words, $\ctimes$ is terminal in the category $\Actt_{\BrMonCat}^\cart(\cat C)$ of braided actions on $\cat C$.
\end{theorem}

As before, in order to prove this we rely on the following characterization:

\begin{lemma}[Characterization of the symmetric center]
\label{lemma:symm-center-charact}
	The symmetric center of $(\cat{C}, i, \ctimes, \beta)$ is equivalent to the full subcategory of $\drinfeld(\cat C)$ spanned by those $\varphi$ for which the following diagrams commute:
	\begin{diagram}
	\label{diag:braided-bilinear}
		{\varphi(c) \ctimes d} & {\varphi(c \ctimes d)} &[-2ex]&[-2ex] {c \ctimes \varphi(d)} & {\varphi(c \ctimes d)} \\
		{d \ctimes \varphi(c)} & {\varphi(d \ctimes c)} && {\varphi(d) \ctimes c} & {\varphi(d \ctimes c)}
		\arrow["{\beta_{\varphi(c),d}}"', from=1-1, to=2-1]
		\arrow["{\varphi(\beta_{c,d})}", from=1-2, to=2-2]
		\arrow["{r_{c,d}}", from=1-1, to=1-2]
		\arrow["{\ell_{d,c}}"', from=2-1, to=2-2]
		\arrow["{\ell_{c,d}}", from=1-4, to=1-5]
		\arrow["{r_{d,c}}"', from=2-4, to=2-5]
		\arrow["{\beta_{c, \varphi(d)}}"', from=1-4, to=2-4]
		\arrow["{\varphi(\beta_{c,d})}", from=1-5, to=2-5]
	\end{diagram}
	We denote this subcategory by $[\cat C, \cat C]_{\Biact[\cat C]{\cat C}}^\braided$.
\end{lemma}
\begin{proof}
	The proof amounts to checking that bilinear endomorphisms which satisfy~\eqref{diag:braided-bilinear} are given, up to isomorphism, by tensoring with a symmetric object.
	Assume $\cat C$ is strict monoidal, and let $(\varphi, r)$ be a bilinear endomorphism strict on the left (we use the same trick employed in Lemma~\ref{lemma:center-charact}).
	Drawing~\eqref{diag:braided-bilinear} for ${c=j}$ proves that $\beta_{d,\varphi(j)}$ and $\beta_{\varphi(j),d}$ are both identities.
	In fact, when $\cat C$ is strict, $r_{j,d} = 1$ as observed in Remark~\ref{rmk:coherence-as-inductive-def}.
	Vice versa, given a symmetric object, we readily prove tensoring with it gives a functor satisfying~\eqref{diag:braided-bilinear}.
\end{proof}

\begin{proof}[of Theorem~\ref{th:classification-braided-act}]
	By Lemma~\ref{lemma:symm-center-charact}, we can replace $\Sigma(\cat C)$ with $[\cat C, \cat C]_{\Biact[\cat C]{\cat C}}^\braided$.
	With this setup, everything works very much like the proof of Theorem~\ref{th:classification-monoidal-act}. One only has to be sure that the universal morphism $F_\action$ actually lands in $[\cat C, \cat C]_{\Biact[\cat C]{\cat C}}^\braided$ when $\action$ is braided, and vice versa.
	But Diagram~\eqref{diag:braided-actegory-axioms} (sealing $\action$ braidedness) looks exactly like Diagram~\eqref{diag:braided-bilinear}, concluding the proof.
\end{proof}

\begin{corollary}
	There is an equivalence
	\begin{equation}
		\SymMonCat \!\underset{\BrMonCat}{\downarrow}\! \Sigma(\cat C) \equi \Actt_{\BrMonCat}^\cart(\cat C)
	\end{equation}
	where the 2-category on the left amounts to the subcategory of $\BrMonCat/\Sigma(\cat C)$ given by symmetric monoidal categories.
\end{corollary}
\begin{proof}
	To spell explicitly the correspondence on objects, given a braided functor $F : \cat M \to \Sigma(\cat C)$, where $\cat M$ is symmetric, we get the $\cat M$-action
	\begin{equation}
		m \action^F c := F(m) \ctimes c,
	\end{equation}
	while given a braided $\cat M$-action $\action$ on $\cat C$, we get the braided functor $\cat M \to \Sigma(\cat C)$ defined as
	\begin{equation}
		F_\action(m) := m \action i.
	\end{equation}
	The well-definedness of these objects follows from the main theorem.
\end{proof}

Now suppose $\cat C$ is symmetric. In that case, $\Sigma(\cat C) = \cat C$, therefore we automatically have the following:

\begin{corollary}[Symmetric actions classifier]
\label{cor:classification-symmetric-act}
	Let $(\cat{C}, i, \ctimes, \sigma)$ be a symmetric monoidal category.
	$\cat C$ \emph{classifies symmetric actions on $\cat C$}, meaning $\cat C$ acts symmetrically on $\cat C$ and every other symmetric action $\action$ of a symmetric monoidal category $\cat M$ factors uniquely through it:
	\begin{diagram}
		{\cat C \times \cat C} & {\cat C} \\
		{\cat M \times \cat C}
		\arrow["\action"', from=2-1, to=1-2]
		\arrow["\ctimes", from=1-1, to=1-2]
		\arrow["{\exists! F_\action \times \cat C}", dashed, from=2-1, to=1-1]
	\end{diagram}
	In other words, $\ctimes$ is terminal in the category $\Actt_{\SymMonCat}^\cart(\cat C)$ of braided actions on $\cat C$, thus inducing an equivalence
	\begin{equation}
		\SymMonCat/ \cat C \equi \Actt_{\SymMonCat}^\cart(\cat C).
	\end{equation}
\end{corollary}

\begin{example}
\label{ex:prob-on-msbl}
	Let $(\Msbl, \action)$ be the $\Prob$-actegory of Example~\ref{ex:stochastic-maps}.
	This is a monoidal actegory, since the action of a probability space $(\Omega, \mathcal F, \mathbb P)$ corresponds to the forgetful functor $\Prob \to \Msbl$ which is evidently braided monoidal, like $(\Msbl,1,\times)$.
\end{example}

\begin{example}
	In Equation~\eqref{eq:cat-actions-on-Mact}, we have exhibited an action $\lhd$ of $\Cat$ on\break ${(\Act{\cat M}, \cat M, \tensor)}$, given by $\cat M[-] \tensor {=}$.
	We can conclude this action is symmetric by noting the functor $\cat M[-] \tensor \cat M \iso \cat M[-]$ is indeed braided monoidal. In fact, we proved in Proposition~\ref{prop:free-act-is-monoidal} that $\cat M[\cat C] \tensor \cat M[\cat D] \equi \cat M[\cat C \times \cat D]$ and this makes $\cat M[-]$ evidently monoidal and braided.
\end{example}

\begin{remark}
	In Remark~\ref{rmk:prod-coprod-cartesian-subcats}, we observed how the classification result of Example~\ref{ex:curried-action-terminal} could be used to construct coproducts and products in $\Actt^\cart(\cat C)$.
	We observe that totally analogous definitions can be repeated for monoidal, braided and symmetric actions.
	This has repercussions on the theory of hybrid composition of optics by providing us with guarantees on the monoidal structure of the resulting category of hybrid optics (remember optics generated by monoidal actions are monoidal, as shown in Remark~\ref{rmk:mon-of-optics}).
\end{remark}


	\newpage
	\section{Conclusion}
\label{sec:conclusion}
In the past pages, we have been  on a journey through `actegory theory'.
We charted a territory sparsely surveyed before, and managed to uncover some new corners of this vast subject.
Specifically, we contemplated the role of monoidal structure in the different ways it can present itself, and described the links between these phenomena and those in the theory of optics.

Large swaths of actegory theory didn't make it in our maps.
Perhaps the largest and most relevant to our eyes is Tambara theory \cite{tambara, pastrostreet,roman2020profunctor} and the study of the proarrow equipment of actegories and Tambara modules.

There are also dark corners, still to be visited, like the exact nature of the relation between braiding and linear structures.
It seems it has never been contemplated the idea of a structure $\rho_{c,m} : c \actionn m \isoto m \action c$ on a biactegory $(\cat C, \action, \actionn)$, linking the left and right actions. It is, in some sense, the actegorical equivalent of a braiding, and in fact a braiding on the scalars induces one canonically (we have seen this in Proposition~\ref{prop:mirroring-biact}). Indeed, a `reflection structure' such as $\rho$ seems to arise anytime there is a strong monoidal functor $\cat M \to \cat M^\rev$.
A shadow of these structures can be observed in Lemma~\ref{lemma:symm-center-charact}, where we characterized the symmetric center of a monoidal category in terms of endomorphisms of $(\cat C, \ctimes, \ctimes)$ that somehow commute with $\beta$. In terms of reflection structures, this constitutes a natural notion of morphism.

In conclusion, let us mention a major perspective we didn't cover here---that of actegories as presheaves.
The data of a left $\cat M$-action, in fact, can be encoded `externally' as a functor $\twocat B\cat M \to \Cat$, whereas a right actegory would be given by $\twocat B\cat M^\op \to \Cat$ since $\twocat B(\cat M^\rev) = \twocat B\cat M^\op$.
The upshot is, actegories can be seen as very special indexed categories, opening up the extension of all the theory we sketched here to actions of categories, bicategories and ultimately double categories \cite{bakovic2008bigroupoid}.
This extension seems to be warranted if we aim to extend the theory of optics to `dependent' or `indexed optics', i.e.~if we desire to deploy dependent types in the current constructions.
The idea that actions of bicategories are instrumental to this objective has been advanced in \cite{braithwaite2021fibre} and then later in \cite{milewski2022compound}.

	\bibliographystyle{alpha}
	\bibliography{bibliography}

\newcommand{\etalchar}[1]{$^{#1}$}
\begin{thebibliography}{CGHR21}

\bibitem[Alc14]{alcolei2014concurrent}
Aurore Alcolei.
\newblock Concurrent program as proofs: compacting message passing logic.
\newblock Unpublished manuscript available at \url{https://aalcolei.epheme.re/report_m1.pdf}, 2014.

\bibitem[Alv20]{ChangeActionsThesis}
Mario Alvarez{-}Picallo.
\newblock Change actions: from incremental computation to discrete derivatives.
\newblock {\em CoRR}, abs/2002.05256, 2020.

\bibitem[APO19]{ChangeActionModels}
Mario Alvarez-Picallo and C.-H.~Luke Ong.
\newblock Change actions: Models of generalised differentiation.
\newblock In Miko{\l}aj Boja{\'{n}}czyk and Alex Simpson, editors, {\em Foundations of Software Science and Computation Structures}, pages 45--61, Cham, 2019. Springer International Publishing.

\bibitem[AU16]{ahman2016directed}
Danel Ahman and Tarmo Uustalu.
\newblock Directed containers as categories.
\newblock arXiv preprint available at \url{https://arxiv.org/abs/1604.01187}, 2016.

\bibitem[Bak08]{bakovic2008bigroupoid}
Igor Bakovi{\'c}.
\newblock {\em Bigroupoid 2-torsors}.
\newblock PhD thesis, lmu, 2008.

\bibitem[BCG{\etalchar{+}}21]{braithwaite2021fibre}
Dylan Braithwaite, Matteo Capucci, Bruno Gavranovi{\'c}, Jules Hedges, and Eigil~Fjeldgren Rischel.
\newblock Fibre optics.
\newblock arXiv preprint available at \url{https://arxiv.org/abs/2112.11145}, 2021.

\bibitem[BD95]{baezdolan1995}
John~C. Baez and James Dolan.
\newblock Higher‐dimensional algebra and topological quantum field theory.
\newblock {\em Journal of Mathematical Physics}, 36(11):6073--6105, 1995.

\bibitem[B{\'e}n67]{benabou1967introduction}
Jean B{\'e}nabou.
\newblock Introduction to bicategories.
\newblock In {\em Reports of the Midwest Category Seminar}, pages 1--77. Springer, 1967.

\bibitem[BKP89]{blackwell1989two}
Robert Blackwell, Gregory~M. Kelly, and John~A. Power.
\newblock Two-dimensional monad theory.
\newblock {\em Journal of pure and applied algebra}, 59(1):1--41, 1989.

\bibitem[Boi20]{boisseau}
Guillaume Boisseau.
\newblock String diagrams for optics.
\newblock In {\em 5th International Conference on Formal Structures for Computation and Deduction (FSCD 2020)}. Schloss Dagstuhl-Leibniz-Zentrum f{\"u}r Informatik, 2020.

\bibitem[BW03]{brzezinski2003london}
T~Brzezinski and R~Wisbauer.
\newblock {\em Corings and Comodules}, volume 309.
\newblock Cambridge University Press, 2003.

\bibitem[BZFN10]{ben2010integral}
David Ben-Zvi, John Francis, and David Nadler.
\newblock Integral transforms and drinfeld centers in derived algebraic geometry.
\newblock {\em Journal of the American Mathematical Society}, 23(4):909--966, 2010.

\bibitem[Cap21]{paratalk}
Matteo Capucci.
\newblock Parametrised categories and categories by proxy.
\newblock Talk at CT2021, 2021.

\bibitem[CEG{\etalchar{+}}20]{categoricalupdate}
Bryce Clarke, Derek Elkins, Jeremy Gibbons, Fosco Loregian, Bartosz Milewski, Emily Pillmore, and Mario Rom\'an.
\newblock Profunctor optics: {A} categorical update.
\newblock {\em NWPT 2019}, page~47, 2020.

\bibitem[CGHR21]{capucci2021towards}
Matteo Capucci, Bruno Gavranovi{\'c}, Jules Hedges, and Eigil~Fjeldgren Rischel.
\newblock Towards foundations of categorical cybernetics.
\newblock In {\em Proceedings of Applied Category Theory 2021}. EPTCS, 2021.
\newblock arXiv preprint available at \url{https://arxiv.org/2105.06763}.

\bibitem[CGLF21]{capucci2021translating}
Matteo Capucci, Neil Ghani, J{\'e}r{\'e}my Ledent, and Fredrik~Nordvall Forsberg.
\newblock Translating extensive form games to open games with agency.
\newblock In {\em Proceedings of Applied Category Theory 2021}. EPTCS, 2021.
\newblock arXiv preprint available at \url{https://arxiv.org/abs/2105.06763}.

\bibitem[CJ19]{Cho_2019}
Kenta Cho and Bart Jacobs.
\newblock Disintegration and bayesian inversion via string diagrams.
\newblock {\em Mathematical Structures in Computer Science}, 29(7):938--971, mar 2019.

\bibitem[Cla20]{clarke2020internal}
Bryce Clarke.
\newblock Internal lenses as functors and cofunctors.
\newblock {\em Electronic Proceedings in Theoretical Computer Science, EPTCS}, 323:183--195, 2020.

\bibitem[Coc13]{cockettACCATtalk}
J.~Robin~B. Cockett.
\newblock Concurrency and the categorical proof theory of message passing.
\newblock Talk at ACCAT Rome, slides available at \url{https://citeseerx.ist.psu.edu/viewdoc/download?doi=10.1.1.674.7533&rep=rep1&type=pdf}, 2013.

\bibitem[CP09]{cockett2009logic}
J.~Robin~B. Cockett and Craig Pastro.
\newblock The logic of message-passing.
\newblock {\em Science of Computer Programming}, 74(8):498--533, 2009.

\bibitem[Day70]{day1970closed}
Brian Day.
\newblock On closed categories of functors.
\newblock In {\em Reports of the Midwest Category Seminar {IV}}, pages 1--38. Springer, 1970.

\bibitem[DD09]{deza_encyclopedia_2009}
E.~Deza and M.~M. Deza.
\newblock {\em Encyclopedia of {Distances}}.
\newblock Springer-Verlag Berlin, 2009.

\bibitem[ENO05]{etingof2005fusion}
Pavel Etingof, Dmitri Nikshych, and Viktor Ostrik.
\newblock On fusion categories.
\newblock {\em Annals of Mathematics}, 162(2):581--642, 2005.

\bibitem[{Fuj}19]{GradedMonads}
Soichiro {Fujii}.
\newblock {A 2-Categorical Study of Graded and Indexed Monads}.
\newblock {\em arXiv e-prints}, page arXiv:1904.08083, April 2019.

\bibitem[Gar18]{garner2018embedding}
Richard Garner.
\newblock An embedding theorem for tangent categories.
\newblock {\em Advances in Mathematics}, 323:668--687, 2018.

\bibitem[Gra76]{grayson1976higher}
Daniel Grayson.
\newblock Higher algebraic k-theory: {II}.
\newblock In {\em Algebraic K-theory}, pages 217--240. Springer, 1976.

\bibitem[Gre10a]{greenough2010bimodule}
Justin Greenough.
\newblock {\em Bimodule categories and monoidal 2-structure}.
\newblock PhD thesis, University of New Hampshire, 2010.

\bibitem[Gre10b]{greenough2010monoidal}
Justin Greenough.
\newblock Monoidal 2-structure of bimodule categories.
\newblock {\em Journal of Algebra}, 324(8):1818--1859, 2010.

\bibitem[JK01]{janelidze2001note}
George Janelidze and Gregory~M. Kelly.
\newblock A note on actions of a monoidal category.
\newblock {\em Theory and Applications of Categories}, 9(4):61--91, 2001.

\bibitem[JS91]{joyalstreet1991}
André Joyal and Ross Street.
\newblock Tortile yang-baxter operators in tensor categories.
\newblock {\em Journal of Pure and Applied Algebra}, 71(1):43--51, 1991.

\bibitem[JS93]{joyal1993braided}
A~Joyal and R~Street.
\newblock Braided tensor categories.
\newblock {\em Advances in Mathematics}, 102(1):20--78, 1993.

\bibitem[JY21]{johnson2021}
Niles Johnson and Donald Yau.
\newblock {\em 2-dimensional Categories}.
\newblock Oxford University Press, USA, 2021.

\bibitem[Kas95]{kassel1995quantum}
Christian Kassel.
\newblock {\em Quantum Groups}.
\newblock Graduate Texts in Mathematics. Springer New York, 1st edition edition, 1995.

\bibitem[Kel64]{kelly1964maclane}
Gregory~M. Kelly.
\newblock On maclane's conditions for coherence of natural associativities, commutativities, etc.
\newblock {\em Journal of Algebra}, 1(4):397--402, 1964.

\bibitem[Kel82]{kelly1982basic}
Gregory~M. Kelly.
\newblock {\em Basic concepts of enriched category theory}, volume~64.
\newblock CUP Archive, 1982.

\bibitem[Koc72]{kock1972strong}
Anders Kock.
\newblock Strong functors and monoidal monads.
\newblock {\em Archiv der Mathematik}, 23(1):113--120, 1972.

\bibitem[Lac00]{lack_coherent_2000}
Stephen Lack.
\newblock A coherent approach to pseudomonads.
\newblock {\em Advances in Mathematics}, 152(2):179--202, 2000.
\newblock Publisher: Elsevier.

\bibitem[Lap72]{laplaza1972coherence}
Miguel~L Laplaza.
\newblock Coherence for distributivity.
\newblock In {\em Coherence in categories}, pages 29--65. Springer, 1972.

\bibitem[Lor21]{loregian_coend_2021}
Fosco Loregian.
\newblock {\em ({Co})end {Calculus}}.
\newblock London {Mathematical} {Society} {Lecture} {Note} {Series}. Cambridge University Press, Cambridge, 2021.

\bibitem[Mac71]{WorkingMathematician}
Saunders MacLane.
\newblock {\em Categories for the Working Mathematician}.
\newblock Springer-Verlag, New York, 1971.
\newblock Graduate Texts in Mathematics, Vol. 5.

\bibitem[Mar99]{marmolejo1999distributive}
Francisco Marmolejo.
\newblock Distributive laws for pseudomonads.
\newblock {\em Theory and Applications of Categories}, 5(5):91--147, 1999.

\bibitem[McC00a]{mccrudden2000balanced}
Paddy McCrudden.
\newblock Balanced coalgebroids.
\newblock {\em Theory and Applications of Categories}, 7(6):71--147, 2000.

\bibitem[McC00b]{mccrudden2000categories}
Paddy McCrudden.
\newblock Categories of representations of coalgebroids.
\newblock {\em Advances in Mathematics}, 154(2):299--332, 2000.

\bibitem[Mil22]{milewski2022compound}
Bartosz Milewski.
\newblock Compound optics.
\newblock arXiv preprint available at \url{https://arxiv.org/abs/2203.12022}, 2022.

\bibitem[NA20]{nishiwaki2020logic}
Yuichi Nishiwaki and Toshiya Asai.
\newblock Logic of computational semi-effects and categorical gluing for equivariant functors.
\newblock arXiv preprint available at \url{https://arxiv.org/abs/2007.04621}, 2020.

\bibitem[noa]{noauthor_hausdorff_nodate}
Hausdorff measure - {Encyclopedia} of {Mathematics}.

\bibitem[Ost03]{ostrik2003module}
Victor Ostrik.
\newblock Module categories, weak hopf algebras and modular invariants.
\newblock {\em Transformation groups}, 8(2):177--206, 2003.

\bibitem[Par77]{pareigis1977non}
Bodo Pareigis.
\newblock Non-additive ring and module theory.
\newblock {\em Publicationes mathematicae}, pages 189--204, 1977.

\bibitem[PS08]{pastrostreet}
Craig Pastro and Ross Street.
\newblock Doubles for monoidal categories.
\newblock {\em Theory and Applications of Categories [electronic only]}, 21:61--75, 2008.

\bibitem[PWG17]{pickering2017profunctor}
Matthew~T Pickering, Nicolas Wu, and Jeremy Gibbons.
\newblock Profunctor optics: Modular data accessors.
\newblock {\em Art, Science, and Engineering of Programming}, 1(2):7, 2017.

\bibitem[Ril18]{riley2018categories}
Mitchell Riley.
\newblock Categories of optics.
\newblock arXiv preprint available at \url{https://arxiv.org/abs/1809.00738}, 2018.

\bibitem[Rom20]{roman2020profunctor}
Mario Rom{\'a}n.
\newblock Profunctor optics and traversals, 2020.
\newblock arXiv preprint available at \url{https://arxiv.org/abs/2001.08045}.

\bibitem[Shu08]{shulman2008framed}
Michael Shulman.
\newblock Framed bicategories and monoidal fibrations.
\newblock {\em Theory and Applications of Categories}, 20:650--738, 2008.

\bibitem[{\v{S}}ko06]{skoda2006biactegories}
Zoran {\v{S}}koda.
\newblock Bi-actegories.
\newblock Unpublished manuscript available at \url{https://www2.irb.hr/korisnici/zskoda/biact.pdf}, 2006.

\bibitem[{\v{S}}ko09]{skoda2009some}
Zoran {\v{S}}koda.
\newblock Some equivariant constructions in noncommutative algebraic geometry.
\newblock {\em Georgian Mathematical Journal}, 16(1):183--202, 2009.

\bibitem[Ste18]{stefanou2018dynamics}
Anastasios Stefanou.
\newblock {\em Dynamics on categories and applications}.
\newblock PhD thesis, State University of New York at Albany, 2018.
\newblock Available at \url{https://sites.google.com/site/anastasiostefanou/home/thesis}.

\bibitem[Str74]{street1974elementary}
Ross Street.
\newblock Elementary cosmoi {I}.
\newblock In {\em Category Seminar}, pages 134--180. Springer, 1974.

\bibitem[Str20]{street2020}
Ross Street.
\newblock Polynomials as span.
\newblock {\em Cahiers de Topologie et G\'eom\'etrie Diff\'erentielle Cat\'egoriques}, LXI(2):113--153, 2020.

\bibitem[Szl09]{szlachanyi2009fiber}
Korn{\'e}l Szlach{\'a}nyi.
\newblock Fiber functors, monoidal sites and tannaka duality for bialgebroids.
\newblock arXiv preprint available at \url{https://arxiv.org/abs/0907.1578}, 2009.

\bibitem[Tam06]{tambara}
Deisuke Tambara.
\newblock Distributors on a tensor category.
\newblock {\em Hokkaido Mathematical Journal}, 35:379--425, 2006.

\bibitem[VC22]{videla2022lenses}
Andre Videla and Matteo Capucci.
\newblock Lenses for composable servers.
\newblock arXiv preprint, available at \url{https://arxiv.org/abs/2203.15633}, 2022.

\bibitem[Woo76]{wood1976}
Richard~J. Wood.
\newblock {\em Indicial methods for relative categories}.
\newblock PhD thesis, Dalhousie University, 1976.

\bibitem[Yea12]{yeasin2012linear}
Masuka Yeasin.
\newblock Linear functors and their fixed points.
\newblock Master's thesis, Graduate Studies, 2012.

\end{thebibliography}

	\appendix
	\section{Definitions}
\label{appendix:defs}

\subsubsection*{\ref*{def:pseudomonad}. Definition.}
\defpseudomonad*%
satisfying the following coherence laws:
\begin{enumerate}[wide, labelwidth=!, labelindent=0pt, label=\Roman*)]
	\item Coherence of the associator:
	\begin{diagram}[sep=3.5ex]
	\label{diag:pseudomonad-cube-law}
		TTTT && TTT &&[-5ex]& TTTT && TTT \\
		&&& TT &&& TTT && TT \\
		TTT && TT &&& TTT \\
		& TT && T &&& TT && T
		\arrow["Tm"{description}, Rightarrow, from=3-1, to=3-3]
		\arrow["{m T}"{description}, Rightarrow, from=1-3, to=3-3]
		\arrow["{m T}"{description}, Rightarrow, from=3-1, to=4-2]
		\arrow["Tm"{description, pos=0.4}, Rightarrow, from=1-3, to=2-4]
		\arrow["\alpha"{description}, triple, curve={height=-6pt}, from=4-2, to=3-3]
		\arrow["\alpha"{description, pos=0.4}, triple, curve={height=6pt}, from=3-3, to=2-4]
		\arrow["m"{description}, Rightarrow, from=4-2, to=4-4]
		\arrow["m"{description}, Rightarrow, from=3-3, to=4-4]
		\arrow[""{name=0, anchor=center, inner sep=0}, "m"{description}, Rightarrow, from=2-4, to=4-4]
		\arrow["{m TT}"{description, pos=0.4}, Rightarrow, from=1-1, to=3-1]
		\arrow["TTm"{description}, Rightarrow, from=1-1, to=1-3]
		\arrow["{Tm T}"{description}, Rightarrow, from=1-6, to=2-7]
		\arrow["{m T}"{description, pos=0.3}, Rightarrow, from=2-7, to=4-7]
		\arrow["{\alpha T}"{description}, triple, curve={height=-6pt}, from=2-7, to=3-6]
		\arrow["Tm"{description, pos=0.3}, Rightarrow, from=2-7, to=2-9]
		\arrow["T\alpha"{description}, triple, curve={height=-6pt}, from=2-7, to=1-8]
		\arrow["m"{description}, Rightarrow, from=4-7, to=4-9]
		\arrow["m"{description}, Rightarrow, from=2-9, to=4-9]
		\arrow["TTm"{description}, Rightarrow, from=1-6, to=1-8]
		\arrow["Tm"{description}, Rightarrow, from=1-8, to=2-9]
		\arrow["{m TT}"{description}, Rightarrow, from=1-6, to=3-6]
		\arrow["{m T}"{description}, Rightarrow, from=3-6, to=4-7]
		\arrow["\alpha"{description}, triple, curve={height=6pt}, from=4-7, to=2-9]
		\arrow[shorten <=13pt, shorten >= 3pt, Rightarrow, no head, from=0, to=3-6]
	\end{diagram}
	\item Coherence of the unitors:
	\begin{diagram}[sep=4.5ex]
	\label{diag:pseudomonad-cone-law}
		TT \\
		& TTT && TT && TT \\
		\\
		& TT && T && T
		\arrow["m"{description}, Rightarrow, from=4-2, to=4-4]
		\arrow[""{name=0, anchor=center, inner sep=0}, "m"{description}, Rightarrow, from=2-4, to=4-4]
		\arrow["\alpha"{description}, triple, from=4-2, to=2-4]
		\arrow["{T m}"{description}, Rightarrow, from=2-2, to=2-4]
		\arrow[""{name=1, anchor=center, inner sep=0}, curve={height=18pt}, Rightarrow, no head, from=1-1, to=4-2]
		\arrow["{m T}"{description}, Rightarrow, from=2-2, to=4-2]
		\arrow["{Tj T}"{description}, Rightarrow, from=1-1, to=2-2]
		\arrow[""{name=2, anchor=center, inner sep=0}, curve={height=18pt}, Rightarrow, no head, from=2-4, to=1-1]
		\arrow[""{name=3, anchor=center, inner sep=0}, "m"{description}, Rightarrow, from=2-6, to=4-6]
		\arrow["{\rho T}"{pos=0.6}', triple, shorten <=4pt, shorten >=2pt, to=2-2, from=1]
		\arrow["T\lambda"{pos=0.6}, triple, shorten <=4pt, shorten >=2pt, to=2-2, from=2]
		\arrow[shorten <=28pt, shorten >=28pt, Rightarrow, no head, from=0, to=3]
	\end{diagram}
\end{enumerate}

\begin{proposition}
\label{prop:m-times-is-pseudomonad}
	The endofunctor $\cat M \times - : \Cat \to \Cat$ admits a pseudomonad structure given by $(j \times -, \mtimes \times -, \lambda \times -, \rho \times -, \alpha \times -)$.
\end{proposition}
\begin{proof}
	Notice Equation~\eqref{diag:pseudomonad-cube-law} depicts two halves of a cube.
	The fillings of the faces allow to map between the six different ways to move from the top left vertex to the bottom right one.
	The two sides of the equation correspond to the two ways to map from the lower outer edge of the cube to the upper outer edge.
	Notice these maps are themselves obtained by applying the modifications $A$ in succession.
	There are exactly five such modifications appearing in the axiom, thus forming a pentagon.
	Such a pentagon has a composite of length two, given by the left hand side of~\eqref{diag:pseudomonad-cube-law}, and a composite of length three, given by the righ hand side.
	Therefore asking~\eqref{diag:pseudomonad-cube-law} to hold for $\cat M \times -$ is equivalent to the pentagonal coherence of the associator $\alpha$ of $\cat M$ \cite[Equation 1.2.5]{johnson2021}.

	A similar reasoning applies to the Equation~\eqref{diag:pseudomonad-cone-law}, which is equivalent to the triangular coherence involving the unitors $\lambda$, $\rho$ and the associator $\alpha$ \cite[Equation 1.2.4]{johnson2021}.
\end{proof}

\subsubsection*{\ref*{def:pseudoalgebra}. Definition.}
\defpseudoalgebra*%
satisfying the following coherence axioms:
\begin{enumerate}[wide, labelwidth=!, labelindent=0pt, label=\Roman*)]
	\item Compatibility with the unit:
	\begin{diagram}[sep=4ex]
		TTTx && TTx &&[-5ex] &[-5ex] TTTx && TTx \\
		& TTx && Tx &&&&& Tx \\
		TTx &&&& {=} & TTx && Tx \\
		& Tx && x &&& Tx && x
		\arrow["TTt", from=1-1, to=1-3]
		\arrow[""{name=0, anchor=center, inner sep=0}, "{Tm_x}"{description}, from=1-1, to=2-2]
		\arrow["Tt"', from=2-2, to=2-4]
		\arrow[""{name=1, anchor=center, inner sep=0}, "Tt", from=1-3, to=2-4]
		\arrow[""{name=2, anchor=center, inner sep=0}, "{m_{Tx}}"', from=1-1, to=3-1]
		\arrow["{m_x}"', from=3-1, to=4-2]
		\arrow[""{name=3, anchor=center, inner sep=0}, "{m_x}"{description}, from=2-2, to=4-2]
		\arrow["t", from=4-2, to=4-4]
		\arrow[""{name=4, anchor=center, inner sep=0}, "t", from=2-4, to=4-4]
		\arrow[""{name=5, anchor=center, inner sep=0}, "{m_{Tx}}"', from=1-6, to=3-6]
		\arrow["TTt", from=1-6, to=1-8]
		\arrow[""{name=6, anchor=center, inner sep=0}, "{m_x}"', from=3-6, to=4-7]
		\arrow["t"', from=4-7, to=4-9]
		\arrow["Tt", from=1-8, to=2-9]
		\arrow[""{name=7, anchor=center, inner sep=0}, "t", from=2-9, to=4-9]
		\arrow["Tt", from=3-6, to=3-8]
		\arrow[""{name=8, anchor=center, inner sep=0}, "{m_X}"{description}, from=1-8, to=3-8]
		\arrow[""{name=9, anchor=center, inner sep=0}, "t", from=3-8, to=4-9]
		\arrow["{T\mu_x}", shorten <=13pt, shorten >=13pt, Rightarrow, from=1, to=0]
		\arrow[shorten <=7pt, shorten >=7pt, Rightarrow, no head, from=3, to=2]
		\arrow["{\mu_x}"', shorten <=13pt, shorten >=13pt, Rightarrow, from=4, to=3]
		\arrow[shorten <=13pt, shorten >=13pt, Rightarrow, no head, from=8, to=5]
		\arrow["{\mu_x}"', shorten <=13pt, shorten >=13pt, Rightarrow, from=9, to=6]
		\arrow["{\mu_x}"'{pos=0.6}, shorten <=7pt, shorten >=7pt, Rightarrow, from=7, to=8]
	\end{diagram}
	\item Compatibility with the multiplication:
	\begin{diagram}[sep=4ex]
		&& Tx &&&& TTx \\
		Tx & TTx && x & {=} & Tx && Tx & x \\
		&& Tx &&&& TTx
		\arrow["{Tj_x}", from=2-1, to=2-2]
		\arrow["Tt", from=2-2, to=1-3]
		\arrow["{m_x}"', from=2-2, to=3-3]
		\arrow["t", from=1-3, to=2-4]
		\arrow["t"', from=3-3, to=2-4]
		\arrow["{\mu_x}", shorten <=6pt, shorten >=6pt, Rightarrow, from=1-3, to=3-3]
		\arrow["t", from=2-8, to=2-9]
		\arrow["{Tj_x}", from=2-6, to=1-7]
		\arrow["{Tj_x}"', from=2-6, to=3-7]
		\arrow["Tt", from=1-7, to=2-8]
		\arrow["{m_x}"', from=3-7, to=2-8]
		\arrow[""{name=0, anchor=center, inner sep=0}, from=2-6, to=2-8]
		\arrow["{T\eta_x}", shorten >=3pt, Rightarrow, from=1-7, to=0]
		\arrow["{\rho_x}", shorten <=3pt, Rightarrow, from=0, to=3-7]
	\end{diagram}
\end{enumerate}

\subsubsection*{\ref*{def:morphism-of-pseudoalgebras}. Definition.}
\deflaxmorphism*%
obeying the following coherence laws:
\begin{enumerate}[wide, labelwidth=!, labelindent=0pt, label=\Roman*)]
	\item Compatibility with the unitor:
	\begin{diagram}[sep=5.5ex]
		x & y &&&& x \\
		Tx & Ty &&&& Tx \\
		x & y &&&& x & y
		\arrow["f", from=1-1, to=1-2]
		\arrow["{j_y}", from=1-2, to=2-2]
		\arrow["{j_x}"', from=1-1, to=2-1]
		\arrow["Tf", from=2-1, to=2-2]
		\arrow["{t'}", from=2-2, to=3-2]
		\arrow["t"', from=2-1, to=3-1]
		\arrow[""{name=0, anchor=center, inner sep=0}, curve={height=-50pt}, Rightarrow, no head, from=1-2, to=3-2]
		\arrow["f"', from=3-1, to=3-2]
		\arrow["\ell"', shift left=1, shorten <=4pt, shorten >=6pt, Rightarrow, from=2-2, to=3-1]
		\arrow[""{name=1, anchor=center, inner sep=0}, curve={height=50pt}, Rightarrow, no head, from=1-6, to=3-6]
		\arrow["{j_x}"', from=1-6, to=2-6]
		\arrow["t"', from=2-6, to=3-6]
		\arrow["f"', from=3-6, to=3-7]
		\arrow["{\eta_y}", shorten <=4pt, Rightarrow, from=0, to=2-2]
		\arrow["{\eta_x}", shorten <=4pt, Rightarrow, from=1, to=2-6]
		\arrow[shorten <=23pt, shorten >=23pt, Rightarrow, no head, from=0, to=1]
	\end{diagram}
	\item Compatibility with the multiplicator:
	\begin{diagram}[sep=4ex]
		TTx && TTy &&& TTx && TTy & {} \\
		& Tx && Ty &&&&& Ty \\
		Tx &&&&& Tx && Ty \\
		& x && y &&& x && y
		\arrow["{m_x}"', from=1-1, to=3-1]
		\arrow["TTf", from=1-1, to=1-3]
		\arrow["Tt"', from=1-1, to=2-2]
		\arrow["t", from=3-1, to=4-2]
		\arrow["t"', from=2-2, to=4-2]
		\arrow["Tt'", from=1-3, to=2-4]
		\arrow["Tf", from=2-2, to=2-4]
		\arrow["t'", from=2-4, to=4-4]
		\arrow["f", from=4-2, to=4-4]
		\arrow["\ell"', shorten <=11pt, shorten >=11pt, Rightarrow, from=2-4, to=4-2]
		\arrow["T\ell"', shorten <=4pt, shorten >=4pt, Rightarrow, from=1-3, to=2-2]
		\arrow["{\mu_x}"', shorten <=4pt, shorten >=4pt, Rightarrow, from=2-2, to=3-1]
		\arrow["TTf", from=1-6, to=1-8]
		\arrow[""{name=0, anchor=center, inner sep=0}, "{m_y}"', from=1-6, to=3-6]
		\arrow["Tf", from=3-6, to=3-8]
		\arrow["{m_y}"', from=1-8, to=3-8]
		\arrow["Tt'", from=1-8, to=2-9]
		\arrow["t'", from=3-8, to=4-9]
		\arrow["t'", from=2-9, to=4-9]
		\arrow["t"', from=3-6, to=4-7]
		\arrow["f"', from=4-7, to=4-9]
		\arrow["{\mu_y}"', shorten <=4pt, shorten >=4pt, Rightarrow, from=2-9, to=3-8]
		\arrow["\ell"', shorten <=10pt, shorten >=10pt, Rightarrow, from=3-8, to=4-7]
		\arrow[shorten <=17pt, shorten >=28pt, Rightarrow, no head, from=2-4, to=0]
	\end{diagram}
\end{enumerate}

\subsubsection*{\ref*{def:monoidal-actegory}. Definition.}
\monact*

As observed in Remark~\ref{rmk:upsilon-is-redundant}, the only extra structure needed is the mixed interchanger:
\begin{equation}
	\iota_{m,c,n,d} : (m \mtimes n) \action (c \ctimes d) \isolongto (m \action c) \ctimes (n \action d).
\end{equation}
All in all, this structure needs to satisfy, on top of the two laws of actegories (Definition~\ref{def:left-actegory}), the following coherence laws regarding $\action$, $\mu$ and $\eta$ monoidality.

\begin{enumerate}[wide, labelwidth=!, labelindent=0pt, label=\Roman*)]
	\item Compatibility with the associator, for every $a, b, c : \cat C$ and $m, n, p : \cat M$:
	\begin{diagram}[column sep=12ex]
		\label{diag:monact-comp-ass}
		{((m \action a) \ctimes (n \action b)) \ctimes (p \action c)} & {(m \action a) \ctimes ((n \action b) \ctimes (p \action c))} \\
		{((m \mtimes n) \action (a \ctimes b)) \ctimes (p \action c)} & {(m \action a) \ctimes ((n \mtimes p) \action (b \ctimes c))} \\
		{((m \mtimes n) \mtimes p) \action ((a \ctimes b)\ctimes c)} & {(m \mtimes (n \mtimes p)) \action (a \ctimes (b \ctimes c))}
		\arrow["{\alpha^{\cat C}_{m \action a, n \action b, p\action c}}", from=1-1, to=1-2]
		\arrow["{\alpha^{\cat M}_{m,n,p} \action \alpha^{\cat C}_{a,b,c}}"', from=3-1, to=3-2]
		\arrow["{(m \action a) \ctimes \iota_{n,p,b,c}}"', to=1-2, from=2-2]
		\arrow["{\iota_{m,n\mtimes p, a, b \ctimes c}}"', to=2-2, from=3-2]
		\arrow["{\iota_{m,n,a,b} \ctimes (p \action c)}", to=1-1, from=2-1]
		\arrow["{\iota_{m \mtimes n, p, a \ctimes b,c}}", to=2-1, from=3-1]
	\end{diagram}
	\item Compatibility with unitors, for every $c : \cat C$ and $m : \cat M$:
	\begin{diagram}[column sep=6ex]
	\label{diag:monact-comp-unit}
		{i \ctimes (m \action x)} & {(j \action i) \ctimes (m \action x)} &[-5ex]&[-5ex] {(m \action x) \ctimes i} & {(m \action x) \ctimes (j \action i)} \\
		{m \action x} & {(j \mtimes m) \action (i \ctimes x)} && {m \action x} & {(m \mtimes j) \action (x \ctimes i)}
		\arrow["{\lambda^{\cat C}_{m \action x}}"', from=1-1, to=2-1]
		\arrow["{\eta_i \ctimes (m \action x)}"', from=1-2, to=1-1]
		\arrow["{\iota_{j,m,i,x}}"', to=1-2, from=2-2]
		\arrow["{\lambda^{\cat M}_m \action \lambda^{\cat C}_x}", from=2-2, to=2-1]
		\arrow["{\iota_{m,j,x,i}}"', to=1-5, from=2-5]
		\arrow["{(m \action x) \ctimes \eta_i}"', from=1-5, to=1-4]
		\arrow["{\rho^{\cat C}_{m\action x}}"', from=1-4, to=2-4]
		\arrow["{\rho^{\cat M}_m \action \rho^{\cat C}_x}", from=2-5, to=2-4]
	\end{diagram}
	\item Monoidality of $\eta$, for every $c, d : \cat C$:
	\begin{diagram}
	\label{diag:eta-monoidal}
		{(j \action c) \ctimes (j \action d)} & {c \ctimes d} \\
		{(j \mtimes j) \action (c \ctimes d)} & {j \action (c \ctimes d)}
		\arrow["{\eta_c \ctimes \eta_d}", from=1-1, to=1-2]
		\arrow["{\iota_{j,j,c,d}}", to=1-1, from=2-1]
		\arrow["{\lambda^{\cat M}_j}"', from=2-1, to=2-2]
		\arrow["{\eta_{c \ctimes d}}"', from=2-2, to=1-2]
	\end{diagram}
	\item Monoidality for $\mu$, for every $c, d : \cat C$ and $m, n, m', n' : \cat M$:
	\begin{diagram}[column sep=12ex]
	\label{diag:mu-monoidal}
		{(m \action (m' \action c)) \ctimes (n \action (n' \action d))} & {((m \mtimes m') \action c) \ctimes ((n \mtimes n') \action d)} \\
		{(m \mtimes n) \action ((m' \action c) \ctimes (n' \action d))} & {((m \mtimes m') \mtimes (n \mtimes n')) \action (c \ctimes d)} \\
		{(m \mtimes n) \action ((m' \mtimes n') \action (c \ctimes d))} & {((m \mtimes n) \mtimes (m' \mtimes n')) \action (c \ctimes d)}
		\arrow["{\mu_{m,m',c} \ctimes \mu_{n,n',d}}", from=1-1, to=1-2]
		\arrow["{\mu_{m \mtimes m', n \mtimes n', c \ctimes d}}"', from=3-1, to=3-2]
		\arrow["{(m \mtimes n) \action \iota_{m',n', c, d}}", to=2-1, from=3-1]
		\arrow["{\iota_{m, n, m' \action c, n'\action d}}", to=1-1, from=2-1]
		\arrow["{\iota_{m \mtimes m', n \mtimes n', c, d}}"', to=1-2, from=2-2]
		\arrow["{(m \mtimes \beta_{n, m'} \mtimes n') \action (c \ctimes d)}"', to=2-2, from=3-2]
	\end{diagram}
\end{enumerate}

\subsubsection*{\ref*{def:balanced-algebroidal-cat}. Definition.}
\balalgebroidalact*%
The extra coherence laws satisfied by $\lambda$, $\rho$, $\alpha$, $\ell$ and $\varepsilon$ are the following:

\begin{enumerate}[wide, labelwidth=!, labelindent=0pt, label=\Roman*)]
	\item Coherence for the left lineator $\ell$:
	\begin{diagram}[row sep=5ex, column sep=-3ex]
		\label{diag:bal-algebr-left-lin-mult}
		&[-2ex] {m \action (n \action (c \ctimes d))} && {m \action ((n \action c) \ctimes d)} &[-2ex]\\
		{(m \mtimes n) \action (c \ctimes d)} &&&& {(m \action (n \action c)) \ctimes d} \\
		&& {((m \mtimes n) \action c) \ctimes d}
		\arrow["{m \action \ell_{n, c, d}}", from=1-2, to=1-4]
		\arrow["{\mu_{m, n, c \ctimes d}}"', from=1-2, to=2-1]
		\arrow["{\ell_{m \mtimes n, c, d}}"', from=2-1, to=3-3]
		\arrow["{\mu_{m,n,c} \ctimes d}", from=2-5, to=3-3]
		\arrow["{\ell_{m, n \action c, d}}", from=1-4, to=2-5]
	\end{diagram}
	\begin{diagram}
		\label{diag:bal-algebr-left-lin-unit}
		{j \action (c \ctimes d)} && {(j \action c) \ctimes d} \\
		& {c \ctimes d}
		\arrow["{\eta_{c \ctimes d}}", from=2-2, to=1-1]
		\arrow["{\ell_{j, c, d}}", from=1-1, to=1-3]
		\arrow["{\eta_c \ctimes d}"', from=2-2, to=1-3]
	\end{diagram}
	\item Coherence for the equilibrator $\varepsilon$:
	\begin{diagram}[row sep=5ex, column sep=-3ex]
	\label{diag:equilib-ass}
		& {(n \action (m \action c)) \ctimes d} \\
		{(m \action c) \ctimes (n \action d)} && {((n \mtimes m) \action c) \ctimes d} \\
		{c \ctimes (m \action (n \action d))} && {((m \mtimes n) \action c) \ctimes d} \\
		& {c \ctimes ((m \mtimes n) \action d)}
		\arrow["{\varepsilon_{m \mtimes n, c, d}}", from=3-3, to=4-2]
		\arrow["{\varepsilon_{n,m \action c, c}}"', from=1-2, to=2-1]
		\arrow["{\varepsilon_{m,c,n\action d}}"', from=2-1, to=3-1]
		\arrow["{c \ctimes \mu^\action_{m,n,d}}"', from=3-1, to=4-2]
		\arrow["{\mu^\action_{n,m,c} \ctimes d}", from=1-2, to=2-3]
		\arrow["{(\beta_{n,m} \action c) \ctimes d}", from=2-3, to=3-3]
	\end{diagram}
	\begin{diagram}
	\label{diag:equilib-unit}
		{(j \action c) \ctimes d} && {c \ctimes (j \action d)} \\
		& {c \ctimes d}
		\arrow["{\eta^\action_c \ctimes d}", from=2-2, to=1-1]
		\arrow["{c \ctimes \eta^\action_d}"', from=2-2, to=1-3]
		\arrow["{\varepsilon_{j,c,d}}", from=1-1, to=1-3]
	\end{diagram}
	\item Linearity of $\lambda$ and $\rho$:
	\begin{diagram}
	\label{diag:lambda-rho-lin}
		{m \action (i \ctimes c)} & {m \action c} &[-5ex]&[-5ex] {m \action (c \ctimes i)} & {m \action c} \\
		{(m \action i) \ctimes c} & {i \ctimes (m \action c)} && {(m \action c) \ctimes i}
		\arrow["{m \action \lambda_c}", from=1-1, to=1-2]
		\arrow["{\lambda_{m \action c}}"', from=2-2, to=1-2]
		\arrow["{m \action \rho_c}", from=1-4, to=1-5]
		\arrow["{\ell_{m,c,i}}"', from=1-4, to=2-4]
		\arrow["{\rho_{m \action c}}"', from=2-4, to=1-5]
		\arrow["{\varepsilon_{m, i, c}}"', from=2-1, to=2-2]
		\arrow["{\ell_{m,i,c}}"', from=1-1, to=2-1]
	\end{diagram}
	\item Linearity of $\alpha$:
	\begin{diagram}[row sep=5ex, column sep=-3ex]
	\label{diag:alpha-left-lin}
		&[-2ex] {m \action ((a \ctimes b) \ctimes c)} && {m \action (a \ctimes (b \ctimes c))} &[-2ex]\\
		{(m \action (a \ctimes b)) \ctimes c} &&&& {(m \action a) \ctimes (b \ctimes c)} \\
		&& {((m \action a) \ctimes b) \ctimes c}
		\arrow["{m \action \alpha_{a,b,c}}", from=1-2, to=1-4]
		\arrow["{\ell_{m,a \ctimes b, c}}"', from=1-2, to=2-1]
		\arrow["{\ell_{m,a,b} \ctimes c}"', from=2-1, to=3-3]
		\arrow["{\alpha_{m \action a, b, c}}"', from=3-3, to=2-5]
		\arrow["{\ell_{m, a, b \ctimes c}}", from=1-4, to=2-5]
	\end{diagram}
	\item Balance of $\alpha$:
	\begin{diagram}
	\label{diag:alpha-balanced}
		{(m \action (a \ctimes b)) \ctimes c} && {(a \ctimes b) \ctimes (m \action c)} \\
		{((m \action a) \ctimes b)) \ctimes c} \\
		{(m \action a) \ctimes (b \ctimes c)} & {a \ctimes (m \action (b \ctimes c))} & {a \ctimes (b \ctimes (m \action c))}
		\arrow["{\alpha_{a,b,m\action c}}", from=1-3, to=3-3]
		\arrow["{\alpha_{m \action a, b,c}}"', from=2-1, to=3-1]
		\arrow["{\ell_{m,a,b} \ctimes c}"', from=1-1, to=2-1]
		\arrow["{\varepsilon_{m,a\ctimes b,c}}", from=1-1, to=1-3]
		\arrow["{a \ctimes \ell_{m,b,c}}"', from=3-2, to=3-3]
		\arrow["{\varepsilon_{m,a, b \ctimes c}}"', from=3-1, to=3-2]
	\end{diagram}
\end{enumerate}

\subsubsection*{\ref*{def:distributive-algebroidal-act}. Definition.}
\distalgebract*%
The extra coherence laws satisfied by $\lambda$, $\rho$, $\alpha$, $\ell$ and $\varepsilon$ are the following:
\begin{enumerate}[wide, labelwidth=!, labelindent=0pt, label=\Roman*)]
	\item Coherence for the absorber $\gamma$:
	\begin{diagram}
	\label{diag:absorber-laws}
		{m \action (n \action o)} && {m \action o} &[-2ex]&[-2ex] {j \action o} && o \\
		{(m \mtimes n) \action o} && o &&& o
		\arrow["{m \action \gamma_n}", from=1-1, to=1-3]
		\arrow["{\mu_{m,n,o}}"', from=1-1, to=2-1]
		\arrow["{\gamma_{m \mtimes n}}"', from=2-1, to=2-3]
		\arrow["{\gamma_m}", from=1-3, to=2-3]
		\arrow[Rightarrow, no head, from=2-6, to=1-7]
		\arrow["{\gamma_j}", from=1-5, to=1-7]
		\arrow["{\eta_o}", from=2-6, to=1-5]
	\end{diagram}
	\item Coherence for the distributor $\delta$:
	\begin{diagram}[row sep=5ex, column sep=-3ex]
		&[-3ex] {m \action (n \action (c \cplus d))} &[-10ex]&[-10ex] {m \action ((n \action c)\cplus(n \action d))} &[-12ex]\\
		{(m \mtimes n) \action (c \cplus d)} &&&& {(m \action (n \action c))\cplus (m \action (n \action d)))} \\
		&& {((m \mtimes n) \action c) \cplus ((m \mtimes n) \action d)}
		\arrow["{m \action \delta_{n,c,d}}", from=1-2, to=1-4]
		\arrow["{\delta_{m, n \action c, n \action d}}", from=1-4, to=2-5]
		\arrow["{\mu_{m,n,c} \cplus \mu_{m,n,d}}", from=2-5, to=3-3]
		\arrow["{\delta_{m \mtimes n, c,d}}"', from=2-1, to=3-3]
		\arrow["{\mu_{m,n,c \cplus d}}"', from=1-2, to=2-1]
	\end{diagram}
	\begin{diagram}
		{j \action (c \cplus d)} && {(j \action c) \cplus (j \action d)} \\
		& {c \cplus d}
		\arrow["{\delta_{j,c,d}}", from=1-1, to=1-3]
		\arrow["{\eta_{c \cplus d}}", from=2-2, to=1-1]
		\arrow["{\eta_c \cplus \eta_d}"', from=2-2, to=1-3]
	\end{diagram}
	\item Linearity of $\lambda$ and $\rho$:
	\begin{diagram}
		{m \action (o \cplus c)} & {m \action c} &[-5ex]&[-5ex] {m \action (c \cplus o)} & {m \action c} \\
		{(m \action o) \cplus (m \action c)} & {o \cplus (m \action c)} && {(m \action c) \cplus (m \action o)} & {(m \action c) \cplus o}
		\arrow["{\gamma_m \cplus (m \action c)}"', from=2-1, to=2-2]
		\arrow["{\lambda_{m \action c}}"', from=2-2, to=1-2]
		\arrow["{\delta_{m,o,c}}"', from=1-1, to=2-1]
		\arrow["{m \action \lambda_c}", from=1-1, to=1-2]
		\arrow["{(m \action c) \cplus \gamma_m}"', from=2-4, to=2-5]
		\arrow["{\delta_{m,c,o}}", from=1-4, to=2-4]
		\arrow["{\rho_{m \action c}}"', from=2-5, to=1-5]
		\arrow["{m \action \rho_c}", from=1-4, to=1-5]
	\end{diagram}
	\item Linearity of $\alpha$:
	\begin{diagram}[column sep=12ex]
		{m \action (a \cplus (b \cplus c))} && {m \action ((a \cplus b) \cplus c)} \\
		{(m \action a) \cplus (m \action (b \cplus c))} && {(m \action (a \cplus b)) \cplus (m \action c)} \\
		{(m \action a) \cplus ((m \action b) \cplus (m \action c))} && {((m \action a) \cplus (m \action b)) \cplus (m \action c)}
		\arrow["{\alpha_{m \action a, m \action b, m \action c}}"', from=3-1, to=3-3]
		\arrow["{m \action \alpha_{a,b,c}}", from=1-1, to=1-3]
		\arrow["{\delta_{m,a,b \cplus c}}"', from=1-1, to=2-1]
		\arrow["{(m \action a) \cplus \delta_{m,b,c}}"', from=2-1, to=3-1]
		\arrow["{\delta_{m,a \cplus b, c}}", from=1-3, to=2-3]
		\arrow["{\delta_{m,a,b} \cplus (m \action c)}", from=2-3, to=3-3]
	\end{diagram}
\end{enumerate}


	\section{Proofs}
\label{appendix:proofs}

\subsubsection*{\ref*{th:biact-are-bialg}. Theorem.}
\biactegoriesarebialgebras*
\begin{proof}
	The most interesting part of the proof concerns obtaining an $\cat N^\rev \times \cat M$-action from a given $(\cat M, \cat N)$-biactegory.

	Indeed, suppose given a left $\cat M$-action $(\action, \eta^\action, \mu^\action)$, a right $\cat N$-action $(\actionn, \eta^\actionn, \mu^\actionn)$ and a bimodulator $\zeta$.
	We already know how from Equation~\eqref{eq:seq-composition} how to define the carriers of the composite $\cat N^\rev \times \cat M$-action $\bothaction$:
	\begin{equation}
		(n,m) \bothaction c := (m \action c) \actionn n.
	\end{equation}
	The unitor is given canonically by:
	\begin{equation}
		\eta^\bothaction_c : c \nlongto{\eta^{\action}_c} j \actionn c \nlongto{\eta^{\actionn}_{j \action c}} (j \action c) \actionn i \equalto (i,j) \bothaction c.
	\end{equation}
	In defining the multiplicator, it is crucial to have $\zeta$:
	\begin{eqalign}
		\mu^\bothaction_{(n,m),(n',m'),c} :=\ &\ (n,m) \bothaction ((n',m') \bothaction c)\\
		&\equalto (m \action ((m' \action c) \actionn n')) \actionn n\\
		&\nlongto{\zeta_{m, m' \action c, n'} \actionn n} ((m \action (m' \action c)) \actionn n') \actionn n\\
		&\nlongto{(\mu^\action_{m,m',c} \actionn n') \actionn n} (((m \mtimes m') \action c) \actionn n') \actionn n\\
		&\nlongto{\mu^\actionn_{n, n', (m \mtimes m') \action c}} (m \mtimes m') \action c) \actionn (n' \ntimes n)\\
		&\equalto (n \revrel{\ntimes} n', m \mtimes m') \bothaction c.
	\end{eqalign}
	Using strictification and the coherence diagrams of $\zeta$, this data can be proven well-defined.
	For instance, from~\eqref{diag:bimod-coherence-triang2}, we can see $\zeta_{m,c,j}$ is an identity when $\action$ and $\actionn$ are strict.
	This immediately entails $\mu^\bothaction_{(j,j),(n',m'),c}$ is an identity, thereby satisfying~\eqref{diag:action-coherence-ltriang}.

	Given a $\cat N^\rev \times \cat M$-action $(\bothaction, \mu^\bothaction, \eta^\bothaction)$, we get a left $\cat M$-action $(\action, \eta^\action, \mu^\action)$ and a right $\cat N$-action $(\actionn, \eta^\actionn, \mu^\actionn)$ as follows:
	\begin{eqalign}
		m \action c := (i, m) \bothaction c,
		\qquad
		\eta^\action_c &:= \eta^{\bothaction}_c,
		\qquad
		\mu^\action_{m,m',c} := \mu^\bothaction_{(j,m),(j,m'),c}\\
		c \actionn n := (n, j) \bothaction c,
		\qquad
		\eta^\actionn_c &:= \eta^{\bothaction}_c,
		\qquad
		\mu^\actionn_{n,n',c} := \mu^\bothaction_{(n,j),(n',j),c}
	\end{eqalign}
	The bimodulator $\zeta$ is then given by
	\begin{eqalign}
		\zeta_{m,c,n}\ :=\ &\
		m \action (c \actionn n)\\
		&\equalto
		(j, m) \bothaction ((n, j) \bothaction c)\\
		&\nlongto{{\mu^\bothaction}}
		(j \revrel{\ntimes} n, m \mtimes j) \bothaction c\\
		&\nlongto{(\rho^{\cat N}, \rho^{\cat M})}
		(n, m) \bothaction c\\
		&\nlongto{({\lambda^{\cat N}}^{-1}\!\!,\; {\lambda^{\cat M}}^{-1})}
		(n \revrel{\ntimes} j, j \mtimes m) \bothaction c\\
		&\nlongto{{\mu^\bothaction}^{-1}}
		(n,j) \bothaction ((j,m) \bothaction c)\\
		&\equalto
		(m \action c) \actionn n.
	\end{eqalign}
	By strictification (Lemma~\ref{lemma:strictification}), this is a well-defined bimodulator (in fact, substituting the above in~\eqref{diag:bimod-coherence-pentag1}--\eqref{diag:bimod-coherence-triang2} produces diagrams involving only structure morphisms of $\bothaction$, $\cat N$ and $\cat M$, which become identities after the strictification).

	On morphisms, it is easy to observe every lax $(\cat M, \cat N)$-bilinear functor (Definition~\ref{def:bilinear-functor}) is also lax $\cat N^\rev \times \cat M$-linear, since this amounts to be lax $\cat N$-linear on the right and lax $\cat M$-linear on the left, and~\eqref{diag:bilinear-hexagon} corresponds to the fact that
	\begin{equation}
		b_{(m \mtimes j, j \mtimes n), c} = b_{(j \mtimes m, n \mtimes j),c}
	\end{equation}
	for an $\cat N^\rev \times \cat M$-linear structure $b$ on a given functor.
	For the same reasons, the opposite is true too.
	Finally, $(\cat M, \cat N)$-bilinear transformations (Definition~\ref{def:bilinear-transformation}) and transformations which are both $\cat M$-linear and $\cat N$-linear are exactly the same thing.
\end{proof}

\subsubsection*{\ref*{prop:biact-internal-hom}. Proposition.}
\biactinternalhom*
\begin{proof}
	We preemptively put ourselves in the situation in which $\cat C$, $\cat D$ and $\cat E$ are all strict, as well as their categories of scalars.
	Of course this doesn't mean $[\cat D, \cat E]_{\cat P}$ only contains strictly linear functors.
	Let's start to show it is a well-defined object, i.e.~that there are well-defined actions of $\cat N$ and $\cat M$ and a compatibility between them.

	Given a scalar $n:\cat N$, one defines its action on a right $\cat P$-linear functor pointwise $F:\cat D \to \cat E$:
	\begin{equation}
		(n \action F)(d) := n \action F(d)
	\end{equation}
	In this way we inherit all left $\cat N$-structure from $\cat E$, including a little help from its bimodulator to provide $n \action F$ with a $\cat P$-linear structure.
	On the right, a scalar $m:\cat M$ acts \emph{before} the functor is applied:
	\begin{equation}
		(F \action m)(d) := F(m \action d).
	\end{equation}
	Again, all the structure is inherited from $\cat D$, including what is needed for fully defining the $\cat P$-linear structure.
	Hence the right $\cat M$-structure on $[\cat D, \cat E]_{\cat P}$ thus defined is as strict as the one we assumed on $\cat D$:
	\begin{eqalign}
		&(F \action j)(d) = F(j \action d) = F(d),\\
		&((F \action m) \action n)(d) = F(m \action (n \action d)) = F((m \mtimes n) \action d) = (F \action (m \mtimes n))(d).
	\end{eqalign}
	Finally, the bimodulator interchanging these two actions is just the identity:
	\begin{equation}
		(n \action (F \action m))(d) = n \action F(m \action d) = ((n \action F) \action m)(d).
	\end{equation}
	To prove~\eqref{eq:internal-hom-adj}, we ground ourselves in the usual tensor-hom adjunction for categories. Let $\lambda$ denote the currying isomorphism.
	On objects, we have to prove that the balanced structure of a given a $(\cat N, \cat P)$-bilinear functor $G: \cat C \times \cat D \to \cat E$ corresponds to a right $\cat M$-linear structure on its curried version $\lambda G$, and similarly for the left $\cat N$-linear structures.
	Regarding the first, we have
	\begin{equation}
		(\lambda G(c) \action m)(d) = \lambda G(c)(m \action d) = G(c, m \action d) \iso[\varepsilon] G(c \action m, d) = \lambda G(c \action m)(d),
	\end{equation}
	and likewise for the second:
	\begin{equation}
		(n \action \lambda G(c))(d) = n \action (\lambda G(c)(d)) \iso[\ell] n \action G(c,d) = G(n \action c, d) = \lambda G(n \action c)(d).
	\end{equation}
	\emph{Vice versa}, given a strong $(\cat N, \cat M)$-bilinear $G : \cat C \to [\cat D, \cat E]_{\cat P}$, we get $\lambda^{-1}G : \cat C \times \cat D \to \cat E$ is balanced:
	\begin{equation}
		\lambda^{-1}G(c \action m, d) = G(c \action m)(d) \iso[\ell] (G(c) \action m)(d) = G(c)(m \action d) = \lambda^{-1}G(c, m \action d).
	\end{equation}
	Regarding morphisms, given a $(\cat N, \cat P)$-bilinear transformation $\xi_{c,d} : F(c,d) \twoto G(c,d)$, then we can curry this one too to get a family $\lambda\xi_c : \lambda F(c) \twoto \lambda G(c)$ of natural morphisms $(\lambda\xi_c)_d : \lambda F(c)(d) \twoto \lambda G(c)(d)$, which amount to the same data.
	By the previous discussion on linearity, the linear constraints also move across the currying without impediment.
	This concludes the proof.
\end{proof}

\subsubsection*{\ref*{cor:biactegories-are-moncat}. Corollary.}
\biactaremoncat*
\begin{proof}
	We complete the proof by defining the extra structure required to make $\Biact{\cat M}^\lax$ into a monoidal 2-category \cite[Explanation 12.1.3]{johnson2021}:
	\begin{enumerate}
		\item The unit for this product is the biactegory $(\cat M, \mtimes, \mtimes)$. In fact it can be easily verified that the natural isomorphism $\tau$ from Proposition~\ref{prop:tens-is-pseudocoeq} induces an equivalence
		\[\begin{tikzcd}
			{\cat M \tensor \cat C} && {\cat C} \\[-5ex]
			{(m, c)} && {m \action c} \\
			{(n, d)} && {n \action d}
			\arrow["{L}", from=1-1, to=1-3]
			\arrow[""{name=0, anchor=center, inner sep=0}, "{(\varphi, f)}"', from=2-1, to=3-1]
			\arrow[""{name=1, anchor=center, inner sep=0}, "{\varphi \action f}", from=2-3, to=3-3]
			\arrow[shorten <=40pt, shorten >=20pt, maps to, from=0, to=1]
		\end{tikzcd}\]
		and an analogous equivalence $R$ can be defined from the tensor $\cat C \tensor \cat M$.
		This functor is an equivalence since its inverse, $c \mapsto (j, c)$ is essentially surjective thanks to the additional family of isomorphisms $\tau$ in $\cat M \tensor \cat C$ (see Proposition~\ref{prop:tensor-actegory-desc}).
		The bilinear structure can be obtained by using the multiplicator of $\cat C$ as a left and right lineator.
		Therefore, $L$ and $R$ are, respectively, the left and right unitor for $\tensor$.
		\item The associator $A:(\cat D \tensor \cat C) \tensor \cat B \to \cat D \tensor (\cat C \tensor \cat B)$ rebrackets triples on objects and morphisms. Its linear structures are both identities:
		\begin{eqalign}
			\ell_{m,d,c,b} &:= m \actionn (d,(c,b)) \equalto (m \actionn d, (c,b)),\\
			r_{m,d,c,b} &:= (d,(c,b)) \action' m \equalto (d, (c,b) \action' m) \equalto (d, (c, b \action' m))
		\end{eqalign}
		where $\action'$ is the right $\cat M$-action on $\cat B$.
		\item The pentagonator is trivial.
		\item The middle 2-unitor is a natural, bilinear, invertible transformation which makes the following commute:
		\begin{diagram}[sep=4ex]
			{(\cat D \otimes_{\cat M} \cat M) \otimes_{\cat M} \cat C} && {\cat D \otimes_{\cat M} (\cat M \otimes_{\cat M} \cat C)} \\
			& {\cat D \otimes_{\cat M} \cat C} & {}
			\arrow[""{name=0, anchor=center, inner sep=0}, "A", from=1-1, to=1-3]
			\arrow["{R \otimes_{\cat M} \cat C}"', from=1-1, to=2-2]
			\arrow["{\cat D \otimes_{\cat M} L}", from=1-3, to=2-2]
			\arrow["\mathcal M", shorten <=3pt, Rightarrow, from=0, to=2-2]
		\end{diagram}
		It is defined as:
		\begin{equation}
			({\mathcal M}_{\cat D, \cat C})_{((d,m),c)} := (d, m \action c) \nlongto{\tau^{-1}} (d \actionn m, c),
		\end{equation}
		thus its bilinearity and invertibility are satisfied by definition.
		\item The left and right 2-unitors are the natural, bilinear, invertible transformations which make the following commute:
		\begin{diagram}[sep=4ex]
			{(\cat M \otimes_{\cat M} \cat D) \otimes_{\cat M} \cat C} && {\cat D \otimes_{\cat M} \cat C} \\
			& {\cat M \otimes_{\cat M} (\cat D \otimes_{\cat M} \cat C)} \\
			{\cat D \otimes_{\cat M} (\cat C \otimes_{\cat M} \cat M)} && {\cat D \otimes_{\cat M} \cat C} \\
			& {(\cat D \otimes_{\cat M} \cat C) \otimes_{\cat M} \cat M}
			\arrow["A"', from=1-1, to=2-2]
			\arrow["L"', from=2-2, to=1-3]
			\arrow[""{name=0, anchor=center, inner sep=0}, "{L \otimes_{\cat M} \cat C}", from=1-1, to=1-3]
			\arrow[""{name=1, anchor=center, inner sep=0}, "{\cat D \otimes_{\cat M} R}", from=3-1, to=3-3]
			\arrow["A", from=4-2, to=3-1]
			\arrow["R"', from=4-2, to=3-3]
			\arrow["{\mathcal L}", shorten <=3pt, Rightarrow, from=0, to=2-2]
			\arrow["{\mathcal R}", shorten <=3pt, Rightarrow, from=1, to=4-2]
		\end{diagram}
		and are both just identity cells:
		\begin{eqalign}
			({\mathcal L}_{\cat D, \cat C})_{((m,d),c)} &:= (m \actionn d, c) \equalto (m \actionn d, c),\\
			({\mathcal R}_{\cat D, \cat C})_{((d,c),m)} &:= (d, c \action m) \equalto (d, c \action m).
		\end{eqalign}
	\end{enumerate}
	Checking the three coherence diagrams for this structure is tedious but routine.
\end{proof}

\subsubsection*{\ref*{prop:waff-prod}. Proposition.}
\waffproduct*
\begin{proof}
	Following the idea outlined in Remark~\ref{rmk:waff-is-aff-comp}, one can verify easily that the repletion\footnote{Meaning we include not just affine endofunctors, but also all endofunctors naturally isomorphic to those.} of the wide subcategory of affine endofunctors on $\cat C$ is closed under composition, hence it is a monoidal subcategory of $([\cat C, \cat C], 1_{\cat C}, \circ)$.
	Then $(\cat C \times \cat M, (o,j), \waff)$ can be realized as a monoidal subcategory of this one, thus proving it is a well-defined monoidal category.

	Concretely, the associator for $\waff$ is given by
	\begin{eqalign}
		\alpha^\waff_{(c, m), (d, n), (e, p)} &: ((c, m) \waff (d, n)) \waff (e, p)\\
		&\equalto
			(c \cplus (m \action d), m \mtimes n) \waff (e, p) \\
		&\equalto
			((c \cplus (m \action d)) \cplus ((m \mtimes n) \cplus e), (m \mtimes n) \action p)\\
		&\nlongto{(\alpha^{\cplus}_{c, m \action d, (m \mtimes n) \cplus e},\ {\mu^\action}^{-1}_{m, n, p})}
			(c \cplus ((m \action d) \cplus ((m \mtimes n) \action e)), m \action (n \action p)) \\
		&\nlongto{(c \cplus ((m \action d) \cplus {\mu^\action}^{-1}_{m, n, e}),\ 1)}
			(c \cplus ((m \action d) \cplus (m \action (n \action e))), m \action (n \action p))\\
		&\nlongto{(c \cplus \delta^{-1}_{m, d, n \action e},\ 1)}
			(c \cplus (m \action (d \cplus (n \action e))), m \action (n \action p))\\
		&\equalto
			(c, m) \waff (d \cplus (n \action e), n \action p)\\
		&\equalto
			(c, m) \waff ((d, n) \waff (e, p)).
	\end{eqalign}

	Finally, the left and right unitors are defined respectively as

	\begin{minipage}{.45\textwidth}
		\begin{eqalign*}
			\lambda^{\waff}_{(c, m)}  &:  (o, j) \waff (c, m)\\
			&\equalto (o \cplus (j \action c), j \mtimes m)\\
			&\nlongto{(\lambda^{\cplus}_{j \action c}, \lambda^{\mtimes}_m)} (j \action c, m)\\
			&\nlongto{(\eta^{\action}_c, 1)} (c, m).
		\end{eqalign*}
	\end{minipage}%
	\begin{minipage}{.45\textwidth}
		\begin{eqalign}
			\rho^{\waff}_{(c, m)}  &:  (c, m) \waff (o, j)\\
			&\equalto (c \cplus (m \action o), m \mtimes j) \\
			&\nlongto{(c \cplus \gamma_m, \rho^{\mtimes}_m)} (c \cplus o), m)\\
			&\nlongto{(\rho^{\cplus}_c, 1)} (c, m).
		\end{eqalign}
	\end{minipage}

	Notice the role of the distributor and absorber in defining $\alpha^\waff$ and $\rho^\waff$.
\end{proof}

\begin{corollary}
\label{cor:waff-action}
	The monoidal category $(\cat C \times \cat M, (o,j), \waff)$ defined in Proposition~\ref{prop:waff-prod} acts on $\cat C$.
\end{corollary}
\begin{proof}
	We established in the previous proof that $(\cat C \times \cat M, (o,j), \waff)$ is a monoidal subcategory of $([\cat C, \cat C], 1_{\cat C}, \circ)$, thus we obtain an action on $\cat C$ from the latter by restriction of scalars along the inclusion $\cat C \times \cat M \to [\cat C, \cat C]$.

	Concretely, the unitor and multiplicator of this action are given by:
	\begin{eqalign}
		\eta^\ostar_c &: (o, j) \ostar c \equalto o \cplus (j \action c) \nlongto{\lambda^\cplus_{j \action c}} j \action c \nlongto{\eta^\action_c} c,
	\end{eqalign}
	\begin{eqalign}
		\mu^{\ostar}_{(a, m), (b, n), c} &: (a,  m) \ostar ((b, n) \ostar c)\\
		&\equalto (a, m) \ostar (b \cplus (n \action c))\\
		&\equalto a \cplus (m \action (b \cplus (n \action c)))\\
		&\nlongto{a \cplus \delta_{m, b, n \action c}} a \cplus ((m \action b) \cplus (m \action (n \action c)))\\
		&\nlongto{a \cplus ((m \action b) \cplus \alpha^{-1}_{m, n, c})} a \cplus ((m \action b) \cplus ((m \mtimes n) \action c))\\
		&\nlongto{\alpha^{-1}_{a, m \action b, (m \mtimes n) \action c}} (a \cplus (m \action b)) \cplus ((m \mtimes n) \action c)\\
		&\equalto (a \cplus (m \action b), m \mtimes n) \ostar c \\
		&\equalto ((a, m) \waff (b, m)) \ostar c.
	\end{eqalign}
\end{proof}

\subsubsection*{\ref*{lemma:center-charact}. Lemma.}
\centercharact*
\begin{proof}
	Without loss of generality, suppose $\cat C$ is strict.
	Furthermore, given a bilinear endomorphism $(\varphi, \ell, r)$, we know by Lemma~\ref{lemma:strictification-functor} we can consider $\ell$ to be trivial.%
	\footnote{In other words, we are factoring the equivalence we have to construct through that defined in Lemma~\ref{lemma:strictification-functor}.}

	So given $(\varphi, r)$ we obtain an object $\varphi(j)$ and a natural isomorphism
	\begin{equation}
		\upsilon_c := \varphi(j) \ctimes c \nlongto{r_{j,c}} \varphi(j \ctimes c) \equalto \varphi(c \ctimes j) \equalto c \ctimes \varphi(j).
	\end{equation}
	This satisfies Equation~\eqref{eq:upsilon-coherence} by a reduction to Diagram~\eqref{diag:bilinear-hexagon}.

	Vice versa, given $(c, \upsilon)$, we get a bilinear endomorphism $- \ctimes c$, whose left linear structure is trivial and right linear structure is given by
	\begin{eqalign}
		r_{d,b} &: d \ctimes c \ctimes b \nlongto{d \ctimes \upsilon_{b}} d \ctimes b \ctimes c.
	\end{eqalign}
	Well-definedness of this bilinear structure follows from the properties of $\upsilon$, as above.

	Now given $(\varphi, r)$, going back and forth this correspondence yields $(- \ctimes \varphi(j), - \ctimes r)$.
	Hence we must check that $(- \ctimes \varphi(j), - \ctimes r) = (\varphi, r)$ as right $\cat C$-linear functors.
	The functor parts coincide by strictiness of the left $\cat C$-linear structure of $\varphi$.
	If we check whether the identity is a valid right linear natural isomorphism, we are led to contemplate the commutativity of the following diagram (we omitted some identity morphisms arising from strictification):
	\begin{diagram}[sep=4ex]
		{c \ctimes \varphi(j) \ctimes d} & {\varphi(c) \ctimes d} \\
		{c \ctimes \varphi(d)} & {\varphi(c \ctimes d)}
		\arrow["{c \ctimes r_{j,d}}"', from=1-1, to=2-1]
		\arrow["{r_{c,d}}", from=1-2, to=2-2]
		\arrow[Rightarrow, no head, from=2-1, to=2-2]
		\arrow[Rightarrow, no head, from=1-1, to=1-2]
	\end{diagram}
	It is easy to convince oneself this is what is left of Diagram~\eqref{diag:bilinear-hexagon} when $\ell$ and $\zeta$ are both identities.

	Instead, when $(x, \upsilon)$ is given, after a back and forth we get back the pair comprised of $j \ctimes x = x$ and the natural isomorphism $j \ctimes \upsilon_z = \upsilon_z$.

	Finally, let us define this isomorphism on morphisms. Given a right linear transformation $\xi : (\varphi, r ) \twoto (\psi, r')$, we have to prove its image $\xi_j : (\varphi(j), r_j) \to (\varphi(j), r'_j)$, we know the top square in the following diagram commutes:
	\begin{diagram}
		{\varphi(j) \ctimes d} & {\psi(j) \ctimes d} \\
		{\varphi(j \ctimes d)} & {\psi(j \ctimes d)} \\
		{d \ctimes \varphi(j)} & {d \ctimes \psi(j)}
		\arrow["{\xi_j \ctimes d}", from=1-1, to=1-2]
		\arrow["{r_{j,d}}"', from=1-1, to=2-1]
		\arrow["{r'_{j,d}}", from=1-2, to=2-2]
		\arrow["{\xi_{j \ctimes d}}"', from=2-1, to=2-2]
		\arrow[from=3-1, to=3-2]
		\arrow[Rightarrow, no head, from=2-1, to=3-1]
		\arrow[Rightarrow, no head, from=2-2, to=3-2]
	\end{diagram}
	Since the bottom one commutes by assumption on $\varphi$, the composite square proves $\xi_j$ satisfies the required condition (namely Diagram~\eqref{diag:drinfeld-center-mors}).
	Contemplating the same diagram proves that given $f:(c, \upsilon) \to (d, \tau)$, we get a well-defined right linear transformation $\xi : (- \ctimes c, - \ctimes \upsilon) \twoto (- \ctimes d, - \ctimes \tau)$ whose components are defined as $\xi_{b} := b \ctimes c \nlongto{b \ctimes f} b \ctimes d$.
\end{proof}


\end{document}